\documentclass[11pt,reqno,a4paper]{amsart}
\usepackage{a4wide,fullpage}
\usepackage{graphicx,amsmath,amssymb,amsfonts,amsthm,enumitem,bm,xcolor}
\usepackage{cleveref}

\usepackage{color}
\usepackage{longtable}
\usepackage{comment}

\newcommand{\N}{\mathbb{N}}
\newcommand{\R}{\mathbb{R}}

\newcommand{\CM}{\mathcal{M}}
\newcommand{\CC}{\mathcal{C}}

\renewcommand{\l}{\lambda}

\newcommand{\z}{\zeta}

\renewcommand{\(}{\left\(}
\renewcommand{\)}{\right\)}

\newcommand{\pa}[2]{\left(\frac{#1}{#2}\right)}

\newcommand{\flo}[1]{\lfloor #1\rfloor}
\newcommand{\Flo}[1]{\left\lfloor #1\right\rfloor}

\numberwithin{equation}{section}
\theoremstyle{plain}
\newtheorem{theorem}{Theorem}[section]
\newtheorem{lemma}[theorem]{Lemma}

\newtheorem{remark}[theorem]{Remark}
\newtheorem*{remark*}{Remark}

\newtheorem{proposition}[theorem]{Proposition}

\numberwithin{equation}{section}

\renewcommand{\binom}[2]{\left(\begin{smallmatrix}#1\\\\#2\end{smallmatrix}\right)}
\newcommand{\smallbinom}[2]{(\begin{smallmatrix}#1\\#2\end{smallmatrix})}
\renewcommand{\pmod}[1]{\ \left( \mathrm{mod} \, #1 \right)}

\setlist[enumerate]{leftmargin=*,label=\rm{(\arabic*)}}

\makeatletter
\@namedef{subjclassname@2020}{%
	\textup{2020} Mathematics Subject Classification}
\makeatother

\allowdisplaybreaks

\newif\ifdefs

\defstrue

\setlist[itemize]{noitemsep, topsep=0pt}
\ifdefs
\else
\excludecomment{discussion}
\excludecomment{extradetails}
\fi

\title{A generic approach to proving Tur\'{a}n-type inequalities for sequences that admit exact formulas, with an application to unimodal sequences}
\author{Koustav Banerjee}
\author{Kathrin Bringmann}
\address{University of Cologne, Department of Mathematics and Computer Science, Weyertal 86-90, 50931 Cologne, Germany}
\email{kbanerj1@uni-koeln.de}
\email{kbringma@math.uni-koeln.de}
\author{Ben Kane}
\address{The University of Hong Kong, Department of Mathematics, Pokfulam, Hong Kong}
\email{bkane@hku.hk}


\begin{document}
\begin{abstract}
We derive an asymptotic expansion with effective error bound for $u(n)$, counting the number of unimodal sequences of size $n$. We prove that $u(n)$ satisfies the higher order Tur\'{a}n inequalities for $n\geq33$ and that certain second $j$-shifted difference of $u(n)$ are positive. 
\end{abstract}

\maketitle

\section{Introduction and statement of results}\hspace{0 cm}
A polynomial with real coefficients is {\it hyperbolic} if it has real roots. Newton \cite[Theorem 2]{RPS} proved that if a polynomial $\sum_{j=0}^{d}\smallbinom{d}{j}a(j)x^j$ is hyperbolic, then $a^2(j)\ge a(j-1)a(j+1)$ for $0<j\le d-1$. A sequence $(a(n))_{n\ge 0}$ of positive real numbers is {\it log-concave} if for all $n\in\N$, 
\begin{equation}\label{log-concave}
a^2(n)\ge a(n-1)a(n+1).
\end{equation} 
The {\it Jensen polynomial} of degree $d$ associated to a sequence with shift $n$, $a(n)$ is defined by $J^{d,n}_{a}(x):=\sum_{0\le j\le d}\smallbinom{d}{j}a(n+j)x^j$. By \cite[Section 6]{ChenSpt}, \eqref{log-concave} is equivalent to $J^{2,n}_a$ being hyperbolic. First, Nicolas \cite{N} and later, DeSalvo and Pak \cite{DP} proved that the partition function $p(n)$ is log-concave for $n\ge 26$. By \cite[Section 6]{ChenSpt}, $J^{3,n}_a$ is hyperbolic iff the discriminant of $J^{3,n}_a$ is non-negative. Following \cite[equation (1.2)]{CJW}, a sequence $(a(n))_{n\ge 0}$ of positive real numbers satisfies the {\it higher order Tur\'{a}n inequalities} for $n\in\N$ if 
\begin{align}\label{Turandef}
4\left(a^2(n)\!-\!a(n-1)a(n+1)\right)&\!\left(a^2(n+1)\!\ge\!a(n)a(n+2)\right)\nonumber\\
&\!-\!\left(a(n)a(n+1)\!-\!a(n-1)a(n+2)\right)^2\!\ge\! 0.
\end{align}
Chen--Jia--Wang \cite{CJW} proved that $p(n)$ satisfies the higher order Tur\'{a}n inequalities for $n\ge 95$. Their proof was based on Rademacher's exact formula for $p(n)$ \cite{Ra} which was a substantial improvement of the Hardy--Ramanujan divergent series for $p(n)$ (see \cite{HR}) and Lehmer's error estimates \cite{L1,L2}. Due to the seminal work of Griffin--Ono--Rolen--Zagier \cite{GORZ}, for each degree $d$, $J^{d,n}_p$ is hyperbolic for $n$ sufficiently large. In \cite{B}, studying the asymptotic expansion (up to any order) of $p(n)$ with an effective error bound, a unified framework was given to prove various inequalities for $p(n)$. For another notable instance Ono--Pujahari--Rolen \cite{OPR} proved that the plane partition function\footnote{For more background on $\operatorname{pp}(n)$, we refer the reader to \cite[Chapter 11]{Andbook}.} $\operatorname{pp}(n)$ is log-concave for $n\ge 12$ by significantly improving the Wright's asymptotic formula \cite{W1} for $\operatorname{pp}(n)$. Pandey \cite{BP} studied the hyperbolicity of $J^{d,n}_{\text{pp}}$ and in particular, showed that $\operatorname{pp}(n)$ satisfies the higher order Tur\'{a}n inequalities for $n\ge 26$.

Let $u(n)$ count the number of {\it unimodal sequences} of $n$,
\[
1\le a_1\le\dots\le a_r\le c\ge b_s\ge\dots\ge b_1\ge1\ \ \ \sum_{j=1}^{r}a_j+c+\sum_{j=1}^{s}b_s=n\ \ r,s\in\mathbb{N}_0.
\]
The generating function for unimodal sequences is given by (see \cite[equation (3.2) and (3.3)]{A})
\begin{equation}\label{unimodalgen}
\sum_{n\geq0}u(n)q^n=\sum_{n\geq0}\frac{q^n}{(q;q)^2_n} =\frac{ 1}{\left(q;q\right)_{\infty}^2}\sum_{n\geq0} (-1)^n q^{\frac{n(n+1)}{2}},
\end{equation}
where for $n\in\N_0\cup\{\infty\}$, $(a;q)_{n}:=\prod_{j=0}^{n-1}(1-aq^{j})$. Unlike $p(n)$ which occur as Fourier coefficients of the modular form $P(q):=(q;q)_{\infty}^{-1}$, we see from \eqref{unimodalgen} that $u(n)$ is associated to the product of a modular form $P^2(q)$ and the partial theta function $\sum_{n\geq0} (-1)^n q^{\frac{n(n+1)}{2}}$. The second author and Nazaroglu \cite{BN} gave a Rademacher-type exact formula for $u(n)$. Using this exact formula, Bridges and the second author \cite{BB} proved that $u(n)$ is log-concave for $n\in \mathbb{N}\setminus \{1,5,7\}$.

In this paper, we determine the asymptotic expansion (up to any order) for $u(n)$ and then for its shifted version $u(n+s)$ with $s\in \mathbb{N}$. We then use this to study the higher order Tur\'{a}n inequalities and second-order shifted differences for $u(n)$. To state the main results, we use the notation $f(n)=O_{\le c}(g(n))$ if $|f(n)|\le c \left|g(n)\right|$ for $n\ge n_0$ and $n_0\in \mathbb{N}$.

\begin{theorem}\label{bbbthm}
	Let $N\in\N_{\ge 3}$. For $n\ge n_N$, we have
	\begin{equation*}
		u(n)=\frac{e^{2\pi\sqrt{\frac{n}{3}}}}{8\cdot 3^{\frac 34}\sqrt{\pi}n^{\frac{5}{4}}} \left(\sum_{m=0}^{N+1}\dfrac{A(m)}{n^{\frac m2}}+O_{\leq C_N}\left(n^{-\frac{N+2}{2}}\right)\right),
	\end{equation*}
	where $n_N$, $A(m)$, and $C_N$ are defined in  \eqref{cutoff}, \eqref{bbblem7eqn1}, and \eqref{bberrordef}, respectively.  
\end{theorem}
\begin{remark}\label{bbbrmk1}
\noindent

\noindent
\begin{enumerate}
\item For $N=3$, the constants $A(m)$ $(0\leq m\leq 4)$ are given in \cite[equation (4.16)]{BB}, and $n_{N}=100000$, $C_3=478$ are as in \cite[Theorem 1.1]{BB}.
\item The convolution of non-negative sequences $(a(n))_{n\ge 0}$, $(b(n))_{n\ge 0}$ is $(c(n))_{n\ge 0}$ with $c(n):=\sum_{k=0}^{n}a(k)b(n-k)$. From \cite{Menon}, log-concavity is preserved under convolution, i.e., if $(a(n))_{n\ge 0}$, $(b(n))_{n\ge 0}$ are log-concave and $a(0)=b(0)=1$, then so is $(c(n))_{n\ge 0}$. Since the higher order Tur\'an inequalities hold for $p(n)$ for $n$ sufficiently large \cite{CJW} and the partition generating function $P(q)$ appears as a factor in \eqref{unimodalgen}, it is natural to ask about the behaviour of the  Tur\'an inequalities under convolution. Specifically, for some $N_0\in\N_0$, suppose that $(a(n))_{n\geq 0}$ and $(b(n))_{n\geq 0}$ are two sequences of real numbers which satisfy the higher oder Tur\'{a}n inequalities for $n\geq N_0$. Under what conditions does there exist $N_1\!\in\!\N_0$ such that the convolution of these two sequences also satisfy (if at all) the higher order Tur\'{a}n inequalities for $n\geq N_1$? How can one determine an effective estimate of the cutoff $N_1$? 
\end{enumerate}
\end{remark}

In order to prove log-concavity and its higher order analogs, we give an error bound for the asymptotic expansion of the shifted unimodal sequence $u(n+s)$.

\begin{theorem}\label{bbbthmshift}
	Let $N\in\N_{\ge 3}$ and $s\in \mathbb{N}$. For $n\ge n_N(s)$, we have
	\begin{equation*}
	u(n+s)=\dfrac{e^{2\pi\sqrt{\frac{n}{3}}}}{8\cdot 3^{\frac 34}\sqrt{\pi}n^{\frac{5}{4}}} \left(\sum_{m=0}^{N+1}\dfrac{A_s(m)}{n^{\frac m2}}+O_{\leq C_{N}(s)}\left(n^{-\frac{N+2}{2}}\right)\right),
	\end{equation*}
	where $n_N(s)$, $A_s(m)$, and $C_N(s)$ are defined in \eqref{bbshiftcutoff}, \eqref{label2}, and \eqref{bbshiftlem6eqn0a}, respectively. 
\end{theorem}

Let $\Delta$ be the {\it backward difference operator} defined on a sequence $(a(n))_{n\ge 0}$ by
\[
\Delta(a)(n):=a(n)-a(n-1),
\]
and for $r\in \mathbb{N}$, the $r$-fold application of $\Delta$ is given as $\Delta^r(a)(n):=\Delta(\Delta^{r-1}(a))(n)$. Gupta \cite{G2} proved that $(p(n))_{n\ge 2}$ is {\it convex}, i.e., $\Delta^2(p)(n)>0$ for $n\ge 2$ and furthermore, he investigated the positivity of $\Delta^r(p)(n)$ for $n$ sufficiently large. For a more detailed study on finite differences of the partition function, we refer the reader to \cite{A0,G2,KK,O}. Recently, Gomez--Males--Rolen \cite{GMR} extended the definition of $\Delta$ for a sequence $(a(n))_{n\ge 0}$ with $1\le j\le n$,
\begin{equation*}
\Delta_j(a)(n):=a(n)-a(n-j).
\end{equation*}
Similarly, the {\it second $j$-shifted difference} is defined as
\begin{equation}\label{2ndshift}
\Delta^{[2]}_j(a)(n):=a(n)-2a(n-j)+a(n-2j).
\end{equation}
In \cite[Theorem 1.2]{GMR}, using Rademacher’s exact formula for $p(n)$ \cite{Ra}, the authors showed that $(p(n))_{n\ge n(j)}$ satisfies {\it shifted convexity}, i.e., $j\in \mathbb{N}$, for $n\ge n(j):=\max\{2,16j^2+\frac{1}{24}\}$, we have\vspace{-.1cm} $\Delta^{[2]}_j(p)(n)\ge 0$. In this paper we prove that $u(n)$ also satisfies shifted convexity.

\begin{theorem}\label{secondshiftedthm}
For $j\in \mathbb{N}$, there exist effectively computable constants $n_\Delta(j)$ (see \eqref{secondshiftedcutoff}) such that for $n\ge n_\Delta(j)$,
\[
\Delta^{[2]}_j(u)(n)>0.
\]
\end{theorem}

\begin{remark}
Positivity of $\Delta^{[2]}_j(u)(n)$ for $n,j\in \mathbb{N}$ can also be shown by elementary combinatorial argument via constructing injections. However, in this paper, since our approach is analytic, we do not present a combinatorial proof.
\end{remark}

We finally use Theorem \ref{bbbthmshift} to verify the higher order Tur\'{a}n inequality for $u(n)$.
\begin{theorem}\label{HigherorderTuranthm}
For $n\ge 33$, $u(n)$ satisfies the higher order Tur\'{a}n inequality. In particular, those $n\in\N$ for which \eqref{Turandef} does not hold are $\{n\in\N: n\leq 26\}\cup \{28,30,32\}$.
\end{theorem}

We note that Theorem \ref{HigherorderTuranthm} is delicate and sharp in the sense that it exactly identifies when the higher order Tur\'{a}n inequalities hold. This sort of result requires two important steps. Firstly, one needs to establish a theorem like Theorem \ref{bbbthmshift} which yields the claim for $n$ sufficiently large. While it is somewhat straightforward to conclude that the inequalities hold for $n$ sufficiently large, one needs a delicate analysis to obtain explicit bounds. Moreover, even if one obtains a bound, it is often too large to be feasibly checked (an original attempt of Theorem \ref{bbbthmshift} required checking in inequality for roughly $n\leq 10^{30}$, which would take about $3\cdot 10^{19}$ years if each check took one millisecond, which is entirely infeasible). Tweaking the arguments carefully, we obtain a version of Theorem \ref{bbbthmshift} which yields the claim in Theorem \ref{HigherorderTuranthm} for $n\geq 9.4\cdot 10^9$. This reduces the problem to the second task: using a computer to check the bound for $n<9.4\cdot 10^9$. This bound, while no longer entirely infeasible, is still rather large, bringing about multiple issues. A first attempt would be to directly follow the method used in \cite[Section 4]{BKRT} to complete the proof of \cite[Corollary 1.4]{BKRT} from \cite[Corollary 1.3]{BKRT}. The method there ran a calculation in series, requiring one to finish all of the previous cases before moving on to the next case. In \cite{BKRT}, one had to check $10^8$ cases, and the running time was 71 days. If the running time were linear (it is a bit worse than linear in reality), then it would take roughly 18 years to complete the calculation in this paper. Even worse, the series calculation requires one to keep the answers from previous cases to compute the next case, and the data for the $n$-th case grows quickly with $n$ (roughly like $e^{\sqrt{n}}$), requiring large amounts of storage space. The calculation in \cite{BKRT} produced roughly 100GB of data to check $n\leq 10^8$, and it is not entirely clear if the amount of data for this roughly 18-year calculation could be feasibly collected with current technology. To overcome these difficulties, we develop a (mostly) parallel algorithm for the calculation in this paper (see Subsection \ref{sec:computations}). This helps to relieve both of the issues noted above. First of all, since the algorithm runs in parallel, one can start many calculations simultaneously, across many different computer cores/CPUs; this reduced a multi-year calculation to a matter of months. Secondly, since the calculations run in parallel, they essentially do not depend on other cases (one requires a few cases side-by-side, as the inequality requires a few of the $u(n)$ to check), so one can remove these from memory when the cases are done. A small number of cases needed to be kept in memory, and the data did not need to be retained, as the only data required in the end was whether the case had been verified by the calculation or not. This reduces the data issues to a manageable level.

Our method is applicable to proving Tur\'{a}n-type inequalities for sequences that admit a Rademacher-type exact formula. Here we mention two such instances arising from coefficients of mixed mock/false modular objects. For instance, using the exact formula for the so-called partitions without sequences $p_2(n)$ due to Bridges and the second author \cite{BB1}, Mauth \cite{Mauth} proved $(p_2(n))_{n\ge 482}$ is log-concave. One can follow this approach to prove Tur\'{a}n-type inequalities (e.g. higher order Tur\'{a}n inequalities) for $p_2(n)$. Another such example arises from irreducible characters of certain vertex operator algebras considered by Cesana \cite{Cesana}.

The paper is organized as follows. In Section \ref{sec2}, we state the Rademacher-type formula for $u(n)$ and gives estimates for $I$-Bessel functions. In Section \ref{S:main}, we work out the asymptotic expansion for the dominating term of the series for $u(n)$. In Section \ref{S:kge2}, we estimate certain error terms. In Section \ref{S:shift}, we first prove Theorem \ref{bbbthm} and using that, we show Theorem \ref{bbbthmshift} which consequently leads to Theorem \ref{secondshiftedthm}. Finally, we prove Theorem \ref{HigherorderTuranthm} in Section \ref{sec:HigherTuran}. In particular, in Subsection \ref{sec:HigherTuranSuffLarge} we use Theorem \ref{bbbthmshift} to show that Theorem \ref{HigherorderTuranthm} holds for $n\geq 9.4\cdot 10^9$, and in Subsection \ref{sec:computations} we develop a computer algorithm to check Theorem \ref{HigherorderTuranthm} for $n< 9.4\cdot 10^9$.

\section*{Acknowledgements}
This work was funded by the European Research Council (ERC) under the European Union's Horizon 2020 research and innovation programme (grant agreement No. 101001179), by the Austrian Science Fund (FWF): W1214-N15, project DK6, and by grants from the Research Grants Council of the Hong Kong SAR, China (project numbers HKU 17314122 and HKU 17305923). The authors thank Caner Nazaroglu and Ali Uncu for help with initial numerical checks, Carsten Schneider and the Research Institute for Symbolic Comptuation (RISC) for allowing them access to the Mach 2 supercomputer, and Johann Messner of the Mach 2 team for helping to parallelize their code and greatly increase the speed and efficiency of their calculations on the Mach 2 supercomputer, and Larry Rolen for his comments and suggestions. Some calculations were performed using research computing facilities offered by Information Technology Services, The University of Hong Kong (HKU ITS), and the authors thank HKU ITS for access to these high performance computing facilities.

\section{Preliminary considerations}\label{sec2}

We have the following exact formula\footnote{Note that the definition of $u(n)$ slightly differs in \cite{BN} and as written the following matches \cite[Theorem 2.1]{BB}.} for $u(n)$. To state the exact formula, we define for $n,r\in \mathbb{Z}$ and $k\in\mathbb{N}$ the {\it Kloosterman sum}
\[
K_k(n,r):=e^{\frac{3\pi i}{4}}(-1)^r\underset{\text{gcd}(h,k)=1}{\sum_{0\le h<k}}\nu_{\eta}\left(M_{h,k}\right)^{-1}\zeta^{-(24n+1)h+\left(12r^2+12r+1\right)h'}_{24k},
\]
where $h'$ is a solution of $hh'\equiv -1\pmod k$, $M_{h,k}:=\left(\!\begin{smallmatrix}
h'& -\frac{hh'+1}{k}\\
k & -h \\
\end{smallmatrix}\!\right)$, $\zeta_{\ell}:=e^{\frac{2\pi i}{\ell}}$ for $\ell\in \mathbb{N}$, and $\nu_{\eta}$ is the multiplier for $\eta$ (see \cite[Theorem 3.4]{Apostol}).
\begin{theorem}[\cite{BN}, Theorem 1.3, negative of the second term]\label{T:ExactFormula}
	We have
	\begin{multline*}
		u(n) = \frac{\pi}{2^\frac34\sqrt3(24n+1)^\frac34}\sum_{k\ge1} \sum_{r=0}^{2k-1} \frac{K_k(n,r)}{k^2}\\
		\times \int_{-1}^1 \left(1-x^2\right)^\frac34 \cot\left(\frac{\pi}{2k}\left(\frac{x}{\sqrt6}-r-\frac12\right)\right) I_\frac32\left(\frac{\pi}{3\sqrt2k}\sqrt{\left(1-x^2\right)(24n+1)}\right) dx,
	\end{multline*}
	where $I_\frac32$ is the $I$-Bessel function of order $\frac32$ and $K_k(n,r)$ is a certain Kloostermann-type sum (in particular $|K_k(n,r)|\le k$).
\end{theorem}
We require the following bounds for Bessel functions (see \cite[Lemma 2.2 (1), (3)]{BKRT}).
\begin{lemma}\label{lem:Bessel}
	\hspace{12mm}

	\begin{enumerate}
		[wide, labelwidth=!, labelindent=0pt]
		\item [\normalfont(1)] For $x\ge 1$, we have
		\begin{equation*}
		I_\kappa(x) \le \sqrt{\frac{2}{\pi x}}e^x.
		\end{equation*}
		\item [\normalfont(2)] For $0 \le x < 1$, we have
		\begin{equation*}
		I_\kappa(x) \le \frac{2^{1-\kappa}x^{\kappa}}{\Gamma(\kappa + 1)}.
		\end{equation*}
	\end{enumerate}
\end{lemma}

\section{The main term}\label{S:main}

Using \cite[page 5]{BN} and the facts\footnote{For the identity on $I$-Bessel function, see \cite[10.2.13]{AS}} that
\[
K_1(n,r)=(-1)^{r+1}\ \ \text{ and }\ \	I_\frac32(w) = \frac{1}{\sqrt{2\pi w}}\left(\left(1-\frac1w\right)e^w+\left(1+\frac1w\right)e^{-w}\right),
\]
the contribution from $k=1$ in \Cref{T:ExactFormula} equals
\begin{align}\nonumber
	&\frac{2^\frac14\pi}{\sqrt3(24n+1)^\frac34}\int_{-1}^1 \left(1-x^2\right)^\frac34\cot\left(\frac\pi2\left(\frac{x}{\sqrt6}+\frac12\right)\right) I_\frac32\left(\frac{2\pi}{\sqrt3}\sqrt{\left(1-x^2\right)\left(n+\frac{1}{24}\right)}\right)dx\\\nonumber
	&\hspace{1.6cm}=\frac{1}{24n+1}\int_{-1}^1 \cot\left(\frac\pi2\left(\frac{x}{\sqrt6}+\frac12\right)\right) \left(\left(\sqrt{1-x^2}-\frac{3\sqrt2}{\pi\sqrt{24n+1}}\right) e^{\frac{\pi}{3\sqrt2}\sqrt{\left(1-x^2\right)(24n+1)}}\right.\\\label{E:k1int}
	&\hspace{5.5cm}\left.+ \left(\sqrt{1-x^2}+\frac{3\sqrt2}{\pi\sqrt{24n+1}}\right) e^{-\frac{\pi}{3\sqrt2}\sqrt{\left(1-x^2\right)(24n+1)}}\right) dx,
\end{align}

For $n\ge 2$ we bound the second term in \eqref{E:k1int} against
\begin{multline}\label{E:error1}
	\left|\frac{1}{24n+1}\int_{-1}^1 \cot\left(\frac\pi2\left(\frac{x}{\sqrt6}+\frac12\right)\right) \left(\sqrt{1-x^2}+\frac{3\sqrt2}{\pi\sqrt{24n+1}}\right) e^{-\frac{\pi}{3\sqrt2}\sqrt{\left(1-x^2\right)(24n+1)}} dx\right|\\
	\le \frac{28}{24n+1}.
\end{multline}

We next split the integral for the first term in \eqref{E:k1int} as
\begin{equation*}
	\int_{-1}^1 = \int_{|x|\le n^{-\frac18}} + \int_{n^{-\frac18}\le|x|\le1}.
\end{equation*}
As in \cite[equation (4.3)]{BB}, we bound the contribution from the second range as
\begin{multline}\label{E:error2}
	\left|\frac{1}{24n+1}\int_{n^{-\frac18}\le|x|\le1} \cot\left(\frac\pi2\left(\frac{x}{\sqrt6}+\frac12\right)\right) \left(\sqrt{1-x^2}-\frac{3\sqrt2}{\pi\sqrt{24n+1}}\right) e^{\frac{\pi}{3\sqrt2}\sqrt{\left(1-x^2\right)(24n+1)}} dx\right|\\
	\le \frac{14}{24n+1}e^{2\pi\sqrt\frac n3-\frac{\pi n^\frac14}{\sqrt3}}.
\end{multline}

We are left to obtain an asymptotic expansion for
\begin{multline}\label{bbb1}
	\dfrac{1}{24n+1}\int_{-n^{-\frac{1}{8}}}^{n^{-\frac{1}{8}}} \cot \left(\frac{\pi}{2}\left(\dfrac{x}{\sqrt{6}}+\dfrac{1}{2}\right)\right)\left(\sqrt{1-x^2}-\dfrac{3\sqrt{2}}{\pi\sqrt{24n+1}}\right)e^{\frac{\pi}{3\sqrt{2}}\sqrt{(1-x^2)(24n+1)}}dx.
\end{multline}
Let $N\in\mathbb N$. The error term needs to take the shape
$\dfrac{e^{2\pi\sqrt{\frac{n}{3}}}}{n^{\frac{5}{4}}}O(n^{-\frac{N+2}{2}}).$
Thus we need to expand the integrand up to $O(n^{-\frac{N+2}{2}})$. Recall a result from \cite[Lemma 7.5]{BPRZ}.

\begin{lemma}\label{prelim}
	For $k\in\N$ and $0<x<\frac12$, we have
	\[
		\left|\sum_{m\ge k} (-1)^m \binom{\frac12}{m} x^m\right| < \frac{2}{k^\frac32}x^k.
	\]
\end{lemma}
A direct calculation using \cite[equations (1) and (2)]{Ro} yields lower and upper bounds for binomial coefficients which we use throughout the paper.

\begin{lemma}\label{prelim1}
	For $n\in\N$, we have
	$$
		{\frac{4^n}{\sqrt{\pi n}}\left(1-\frac{1}{8n}\right)	<}\binom{2n}{n}< \frac{4^n}{\sqrt{\pi n}}.
	$$
\end{lemma}

We proceed to estimate the integrand in \eqref{bbb1}. For $N\in\N$ ($N\ge 3$ throughout), set
\begin{equation}\label{eqn:ydef}
y=y_{n}(x):=\frac{\pi\sqrt{24n+1}}{3\sqrt{2}}\sum_{m=2}^{2N+5}(-1)^m\binom{\frac{1}{2}}{m} x^{2m}.
\end{equation}

\begin{lemma}\label{bbblem1}
For $x\in\mathbb R$ with $|x|\leq n^{-\frac{1}{8}}$, $n\ge17$, and $N\ge 3$, we have
\begin{align*}
e^{ \frac{\pi}{3\sqrt{2}}\sqrt{(1\!-\!x^2)(24n+1)}}\!=\!e^{\frac{\pi\sqrt{24n+1}}{3\sqrt{2}}\!\left(1\!-\!\frac{x^2}{2}\right)}\!\left(\sum_{j=0}^{N+1}\!\frac{y^j}{j!}\!+\!O_{\le\frac{0.7^{N+2}}{(N+2)!}}\left(x^{4N+8}n^{\frac{N+2}{2}}\right)\right)\! \left(1\!+\!O_{\leq 0.4}\left( n^{-\frac{N+2}{2}}\right)\right).
\end{align*}
\end{lemma}

\begin{proof}
	By Taylor expanding $\sqrt{1-x^2}$ and noting that $|x|\le n^{-\frac18}< \frac1{\sqrt{2}}<1$ for $n\ge17$, we have
	\begin{align*}
		e^{\frac{\pi}{3\sqrt2}\sqrt{\left(1-x^2\right)(24n+1)}}&= \exp\left(\frac{\pi\sqrt{24n+1}}{3\sqrt{2}}\left(1-\frac{x^2}{2}\right)+y+\frac{\pi\sqrt{24n+1}}{3\sqrt{2}}\sum_{m\ge2N+6} (-1)^m \binom{\frac12}{m}x^{2m}\right)\\
		&=\exp \left(\frac{\pi\sqrt{24n+1}}{3\sqrt{2}}\left(1-\frac{x^2}{2}\right)+y+\frac{\pi\sqrt{24n+1}}{3\sqrt{2}}\ O_{\le\frac{1}{12\sqrt{3}}}\left( x^{4N+12}\right)\right),
	\end{align*}
applying \Cref{prelim} with $k=2N+6$ and $x\mapsto x^2$, for $n\ge17$ in the final step.

Using the fact that $e^u\leq 1+2|u|$ for $u\in [-1,1]$ and $|x|\leq n^{-\frac{1}{8}}$, it follows that 
	\begin{equation*}
		\textnormal{exp} \left(\frac{\pi\sqrt{24n+1}}{3\sqrt{2}}\ O_{\le \frac{1}{12\sqrt3}}\left( x^{4N+12}\right)\right)=1+O_{\leq 0.4}\left(n^{-\frac{N+2}{2}}\right).
	\end{equation*}
For use below, we note that, for $m\in\N$,
\begin{equation}\label{E:binomial}
\binom{\frac12}{m} = \frac{(-1)^{m+1}}{2m\cdot4^{m-1}}\binom{2m-2}{m-1}.
\end{equation}

We next bound $e^y$. Since every term in the sum is negative, we have $y<0$ and for $n\ge 17$, 
\begin{align*}
|y|&\le \frac{\pi\sqrt{24+\frac1{17}}}{3\sqrt2}\sqrt{n}x^4\sum_{m=2}^{2N+5} \left|\binom{\frac12}{m}\right| x^{2m-4}=\frac{\pi\sqrt{24+\frac1{17}}}{3\sqrt2}\sqrt{n}x^4\sum_{m=2}^{2N+5}\dfrac{\binom{2m-2}{m-1}}{2m \cdot 4^{m-1}}x^{2m-4}\\
&\le \frac{\pi\sqrt{24+\frac1{17}}}{6\sqrt2}\sqrt{n}x^4\sum_{m=2}^{2N+5} \frac{x^{2m-4}}{m\sqrt{m-1}}\le \frac{\pi\sqrt{24+\frac1{17}}}{6\sqrt2}\sqrt n x^4\!\left(\sum_{m=2}^{4}\frac{2^{2-m}}{m\sqrt{m-1}} + \frac{1}{40}\right)\!<0.7\sqrt n x^4,
\end{align*}
using \eqref{E:binomial} in the second step and Lemma \ref{prelim1} in the third step. In particular $|y|<0.7$ for $|x|\leq n^{-\frac{1}{8}}$. Hence,
	\[
		\left|e^y-\sum_{j=0}^{N+1} \frac{y^j}{j!}\right| \le \frac{|y|^{N+2}}{(N+2)!}<  \frac{0.7^{N+2}}{(N+2)!} x^{4N+8}n^{\frac{N+2}{2}}.
\]
	Combining, the lemma follows.
\end{proof}
Next we require bounds for error terms in the Taylor expansions of the two functions in the integrand in \eqref{bbb1}, namely, $$\phi(x):=\cot \left(\frac{\pi}{2}\left(\frac{x}{\sqrt{6}}+\frac{1}{2}\right)\right)\ \ \text{and}\ \ \sqrt{1-x^2}.$$

To estimate the Taylor expansion of $\phi(x)$, we start with the following representation:
\begin{equation*}
	\phi(x)=\sum_{m=0}^{N+1}\frac{\phi^{(2m)}(0)}{(2m)!}x^{2m}+P_{\textnormal{odd}}(x)+O_{\leq C}\left(x^{2N+4}\right),
\end{equation*}
with $P_{\text{odd}}$ an odd polynomial, so not contributing to the integral. It remains to determine $C$.
\begin{lemma}\label{bbblem2}
For $N\ge 3$ and $|x|\leq \frac{1}{2}$, we have
\begin{equation*}
	\cot \left(\frac{\pi}{2}\left(\frac{x}{\sqrt{6}}+\frac{1}{2}\right)\right)=\sum_{m=0}^{N+1}\frac{\phi^{(2m)}(0)}{(2m)!}x^{2m}+P_{\textnormal{odd}}(x)+O_{\le2.3 \pa{2}{3}^{N+2}}\left(x^{2N+4}\right).
\end{equation*}
\end{lemma}

\begin{proof}
	From the Taylor expansion $\cot(x) =\sum_{m\ge0} \frac{(-1)^m2^{2m}}{(2m)!}B_{2m}x^{2m-1},$ it follows that
	\begin{multline*}
		\cot\left(\frac\pi2\left(\frac{x}{\sqrt6}+\frac12\right)\right) = \sum_{m\ge0} \frac{(-4)^mB_{2m}}{(2m)!}\pa\pi4^{2m-1}\left(1+\sqrt\frac23x\right)^{2m-1}\\
		= \frac4\pi\sum_{m\ge0} (-1)^m\left(\sqrt\frac23x\right)^m + \sum_{m\ge1} \frac{(-4)^mB_{2m}}{(2m)!}\pa\pi4^{2m-1}\left(1+\sqrt\frac23x\right)^{2m-1} =: S_1(x) + S_2(x),
	\end{multline*}
where $B_n$ is the $n$-th Bernoulli number. We now split $S_1(x)$ as 
	\begin{align*}
		S_1(x) 
		&= \frac4\pi\sum_{m=0}^{N+1} \pa23^mx^{2m} + P_{\text{odd}}^{[1]}(x) + \frac4\pi\sum_{m\ge2N+4} (-1)^m\left(\sqrt\frac23x\right)^m,
	\end{align*}
	where $P_{\text{odd}}^{[1]}(x)$ is an odd polynomial. We bound the final term as
	\[
		\left|\frac4\pi\sum_{m\ge2N+4} (-1)^m\left(\sqrt\frac23x\right)^m\right| \le \frac{4\pa23^{N+2}x^{2N+4}}{\pi\left(1-\frac1{\sqrt6}\right)} \le 2.2 \pa{2}{3}^{N+2} x^{2N+4}.
	\]
Thus we obtain 	\begin{equation}\label{bbbeqn2}
		S_1(x) = \frac4\pi\sum_{m=0}^{N+1} \pa23^mx^{2m} + P_{\text{odd}}^{[1]}(x) + O_{\le2.2\pa{2}{3}^{N+2}}\left(x^{2N+4}\right).
	\end{equation}
	
	We next split $S_2(x)$ as 
	\begin{align}\nonumber
		S_2(x) &= \sum_{m\ge1} \frac{(-4)^mB_{2m}}{(2m)!}\pa\pi4^{2m-1}\sum_{\ell=0}^{2m-1} \binom{2m-1}{\ell}\left(\sqrt\frac23x\right)^\ell= \sum_{\ell\ge0} \sum_{m\ge\Flo{\frac\ell2}+1} c(\ell,m)\left(\frac23\right)^\frac\ell2 x^\ell\\\label{E:S2}
	&= \sum_{\ell=0}^{N+1} \sum_{m\ge\ell+1} c(2\ell,m)\pa23^\ell x^{2\ell} + P_{\text{odd}}^{[2]}(x) + \sum_{\ell\ge2N+4} \sum_{m\ge\Flo{\frac\ell2}+1} c(\ell,m)\left(\frac23\right)^{\frac{\ell}{2}} x^\ell,
	\end{align}
	where $c(\ell,m):=\frac{(-4)^mB_{2m}}{(2m)!}(\frac\pi4)^{2m-1}\smallbinom{2m-1}{\ell}$ and $P_{\text{odd}}^{[2]}(x)$ is an odd polynomial. Due to uniqueness of Taylor series, it follows that
	\begin{equation}\label{bbneweqn0}
		\sum_{\ell=0}^{N+1} \frac{\phi^{(2\ell)}(0)}{(2\ell)!}x^{2\ell} = \frac4\pi\sum_{\ell=0}^{N+1} \pa23^\ell x^{2\ell} + \sum_{\ell=0}^{N+1} \sum_{m\ge\ell+1} c(2\ell,m)\pa23^\ell x^{2\ell},
	\end{equation}
	and $P_{\text{odd}}(x)=P_{\text{odd}}^{[1]}(x)+P_{\text{odd}}^{[2]}(x)$. So, it remains to find an upper bound for the absolute value of the final term in \eqref{E:S2}. First, we estimate
	\[
		\left|\sum_{\ell\ge2N+4} \sum_{m\ge\Flo{\frac\ell2}+1} c(\ell,m)\left(\frac23\right)^{\frac{\ell}{2}} x^\ell\right| \le \pa23^{N+2}x^{2N+4}\sum_{\ell\ge2N+4} \sum_{m\ge\Flo{\frac\ell2}+1} |c(\ell,m)|.
	\]
Using that for $m\in\N$, we have $\zeta(2m) = \frac{(-1)^{m+1}B_{2m}(2\pi)^{2m}}{2(2m)!}\le \zeta(2)$,
we bound the inner sum as
	\begin{align}\nonumber
		\sum_{m\ge\Flo{\frac\ell2}+1}\! |c(\ell,m)| &\le \!\sum_{m\ge\Flo{\frac\ell2}+1}\! \frac{4^m|B_{2m}|}{(2m)!}\pa\pi4^{2m-1}\binom{2m-1}{m-1} = \frac2\pi\!\sum_{m\ge\Flo{\frac\ell2}+1}\! \frac{\z(2m)}{4^{2m-1}}\binom{2m-1}{m-1}\\\label{bbneweqn1}
		&\le \frac4\pi\!\sum_{m\ge\Flo{\frac\ell2}+1}\! \frac{\z(2m)}{4^{2m-1}}\binom{2m-2}{m-1} \le \frac{4\z(2)}{\pi^\frac32\sqrt{\Flo{\frac\ell2}}}\!\sum_{m\ge\Flo{\frac\ell2}+1}\! 4^{-m} = \frac{2\sqrt\pi }{9\sqrt{\Flo{\frac\ell2}}4^{\Flo{\frac\ell2}}},\hspace{-0.2cm}
	\end{align}
	using \Cref{prelim1} in the penultimate step and $\z(2)=\frac{\pi^2}6$ in the final step. Consequently
	\begin{align*}
		&\left|\sum_{\ell\ge2N+4} \sum_{m\ge\Flo{\frac\ell2}+1} c(\ell,m)\left(\sqrt\frac23\right)^\ell x^\ell\right| \le \frac{2\sqrt\pi}{9}\pa{2}{3}^{N+2}x^{2N+4}\sum_{\ell\ge2N+4} \frac{4^{-\Flo{\frac\ell2}}}{\sqrt{\Flo{\frac\ell2}}}\\
		&\le \tfrac{2^{\frac 32}\sqrt\pi}{9\sqrt{2N+3}}\pa{2}{3}^{N+2}x^{2N+4}\sum_{\ell\ge2N+4} 4^{-\Flo{\tfrac\ell2}} = \frac{2^{\frac 52}\sqrt\pi x^{2N+4}}{9\cdot 6^{N+2}\sqrt{2N+3}}\sum_{\ell\ge 0} 4^{-\ell}= \frac{2^{\frac 92}\sqrt\pi x^{2N+4}}{27\cdot6^{N+2}\sqrt{2N+3}}\\
		& \le  \frac{ x^{2N+4}}{2\cdot6^{N+2}}.
	\end{align*}
Combining with \eqref{bbbeqn2}, we conclude the claim.
\end{proof}

We next bound the tail in the Taylor expansion of $\sqrt{1-x^2}$.

\begin{lemma}\label{newcor1}
	For $N\ge 3$ and $|x|\le \frac 12$, we have
	$$
		\left|\sqrt{1-x^2}-\sum_{m=0}^{N+1}(-1)^m \binom{\frac12}{m}x^{2m}\right|\le\frac{x^{2N+4}}{15\sqrt{\pi}}.
	$$
\end{lemma}
\begin{proof}
	Applying \eqref{E:binomial}, \Cref{prelim1}, and $|x|\le \frac 12$, we have
	\[
		\left|\sqrt{1-x^2}-\sum_{m=0}^{N+1}(-1)^m \binom{\frac12}{m}x^{2m}\right| \le \frac{2x^{2N+4}}{3\sqrt\pi(N+2)\sqrt{N+1}}\le \frac{x^{2N+4}}{15\sqrt\pi}.\qedhere
	\]
\end{proof}

Combining \Cref{newcor1} with Lemmas \ref{bbblem1} and \ref{bbblem2}, we see that \eqref{bbb1} equals
\begin{align}\nonumber
	&\frac{e^{\frac{\pi\sqrt{24n+1}}{3\sqrt2}}}{24n+1}\int_{-n^{-\frac18}}^{n^{-\frac18}} \left(\sum_{m=0}^{N+1} \frac{\phi^{(2m)}(0)}{(2m)!}x^{2m}+O_{\le2.3\pa{2}{3}^{N+2}}\left(x^{2N+4}\right)\right)\\
	\nonumber
	&\hspace{2cm}\times \left(\sum_{m=0}^{N+1}(-1)^m \binom{\frac12}{m}x^{2m} + O_{\le \frac{1}{15\sqrt{\pi}}}\left(x^{2N+4}\right) - \frac{3\sqrt2}{\pi\sqrt{24n+1}}\right)e^{-\frac{\pi\sqrt{24n+1}}{3\sqrt2}\frac{x^2}{2}}\\
	\label{bbbeqn3}
	&\hspace{2cm}\times \left(\sum_{j=0}^{N+1} \frac{y_n^j(x)}{j!}+O_{\le\frac{0.7^{N+2}}{(N+2)!}}\left(x^{4N+8}n^\frac{N+2}{2}\right)\right) \left(1+O_{\le0.4}\left(n^{-\frac{N+2}{2}}\right)\right) dx.
\end{align}

Setting $\l(n):=\sqrt{\frac{\pi}{6\sqrt2}\sqrt{24n+1}}$, it is not hard to see that 
\begin{equation}\label{bbbeqn4}
	1.3 n^{\frac{1}{4}} \leq \l(n) \leq 1.4 n^{\frac{1}{4}}.
\end{equation}

Next, by making the change of variables $x\mapsto \frac{x}{\l(n)}$ in \eqref{bbbeqn3}, for $n\ge17$ we get
\begin{equation}\label{bbbeqn11}
\frac{e^{\frac{\pi\sqrt{24n+1}}{3\sqrt2}}}{(24n+1)\l(n)} \int_{-\l(n)n^{-\frac18}}^{\l(n)n^{-\frac18}} P_{n}^{[1]}(x)P_{n}^{[2]}(x)P_{n}^{[3]}(x)P_{n}^{[4]}e^{-x^2} dx,
\end{equation}
where
\begin{align*}
P_{n}^{[1]}(x)&=P_{n}^{[1]}(N;x) := \sum_{m=0}^{N+1} \frac{\phi^{(2m)}(0)}{(2m)!}\pa{x}{\l(n)}^{2m} + O_{\le2.3 \pa{2}{3}^{N+2}}\left(\pa{x}{\l(n)}^{2N+4}\right),\\
	P_{n}^{[2]}(x)&=	P_{n}^{[2]}(N;x) := \sum_{m=0}^{N+1}(-1)^m \binom{\frac12}{m}\pa{x}{\l(n)}^{2m} + O_{\le \frac{1}{15\sqrt{\pi}}}\left(\pa{x}{\l(n)}^{2N+4}\right)-\frac{1}{2\l^{2}(n)},\\
P_{n}^{[3]}(x)&=P_{n}^{[3]}(N;x) := \sum_{j=0}^{N+1} \frac{y_n^j\left(\frac{x}{\l(n)}\right)}{j!}+O_{\le\frac{0.5^{N+2}}{(N+2)!}}\left(\pa{x^2}{\l(n)}^{2N+4}\right), \\
P^{[4]}_n&=P_{n}^{[4]}(N):= 1 + O_{\le 0.4}\left(n^{-\frac{N+2}{2}}\right),
\end{align*}

Using definition \eqref{eqn:ydef}, we next write
\begin{equation*}
	y_n\pa{x}{\l(n)} = 2\sum_{m=2}^{N+2}(-1)^m \binom{\frac12}{m}\frac{x^{2m}}{\l^{2m-2}(n)} - 2x^2F_N\pa{x}{\l(n)},
\end{equation*}
where for $|X|\le\frac12$, we have, by \eqref{E:binomial} and \Cref{prelim1},
\[
	F_N(X) := \sum_{m=N+2}^{2N+4}(-1)^m \binom{\frac12}{m+1}X^{2m}= O_{\le3\cdot 10^{-2}}\left(X^{2N+4}\right).
\]
Thus for $\frac{|x|}{\l(n)}\leq  n^{-\frac{1}{8}}\leq \frac{1}{2}$, using \eqref{bbbeqn4}, we have that $2x^2F_N(\frac{x}{\l(n)})=O_{\le 6\cdot 10^{-2}}(\frac{x^{2N+6}}{\l^{2N+4}(n)})$. Consequently, for $n\geq  256$ we obtain 
\begin{equation}\label{bbneweqn01}
	y_n\pa{x}{\l(n)} = \sum_{m=1}^{N+1}\frac{t_m(x)}{\l^{2m}(n)} + O_{\le6\cdot 10^{-2}}\left(\frac{x^{2N+6}}{\l^{2N+4}(n)}\right),
\end{equation}
where
\begin{equation}\label{deftm}
t_m(x):=2(-1)^{m+1}\binom{\frac{1}{2}}{m+1}x^{2m+2}\ \ \text{for}\ \ 1\le m\le N+1.
\end{equation}

Using that $W(j)=0$ if $j\ge T+1$, we write 
\begin{align}\label{cauchyproduct}
\!&\sum_{m=0}^{T}\!Y(m)X^m\!\sum_{\ell=0}^{T}\! W(\ell)X^{\ell}\nonumber\\ 
&\hspace{1.5 cm}=\!\sum_{m=0}^{T}\!\sum_{k=0}^{m}\!Y(k)W(m-k)X^m+X^{T+1}\sum_{m=0}^{T-1}\sum_{k=m}^{T-1}Y(k+1)W(m-k+T)X^m.
\end{align}

We need the following two inequalities. Using that $|\smallbinom{\frac12}1|=\frac 12$, \eqref{E:binomial}, and \Cref{prelim1} with $n=m-1$, we obtain for $m\in\N$,
\begin{equation}\label{prelim2}
	\left|\binom{\frac12}{m}\right|\le \frac{1}{2 m^{\frac{3}{2}}}.
\end{equation}
For $\frac{|x|}{\l(n)}\leq n^{-\frac{1}{8}}$ and $m\in\N$, we note that if $n\ge 10^4$, then we have 
\begin{equation}\label{prelim3}
	\left|\pa{x}{\l(n)}^{2m}\right|<10^{-m}.
\end{equation}

\begin{lemma}\label{bbnewlem1}
	For $N\geq 3$ and $n\ge 10^4$, we have
	\begin{align*}
		P_{n}^{[1]}(x) P_{n}^{[2]}(x) P_{n}^{[4]}=\sum_{m=0}^{N+1}\frac{C^{[1]}_m(x)}{\l^{2m}(n)}
		+O_{\le 1}\!\left(\frac{0.7+0.2\cdot 0.4^{N}x^{2N+2}+1.5\cdot 0.6^{N}x^{2N+4}}{n^{\frac{N+2}{2}}}\right)\!,
	\end{align*}
	where $C^{[1]}_0(x):=1$, and for $m \in \N$,
	\begin{equation*}
		C^{[1]}_m(x) := x^{2m}\sum_{k=0}^{m}(-1)^{k+m}\dfrac{\phi^{(2k)}(0)}{(2k)!}\binom{\frac{1}{2}}{m-k} - \dfrac{\phi^{(2m-2)}(0)}{2 (2m-2)!}x^{2m-2}.
	\end{equation*}
\end{lemma}

\begin{proof}
Using \eqref{cauchyproduct}, we compute
	\begin{align}\nonumber
		P_{n}^{[1]}(x) P_{n}^{[2]}(x) &= \sum_{m=0}^{N+1} \pa{x}{\l(n)}^{2m} \sum_{k=0}^m(-1)^{k+m} \frac{\phi^{(2k)}(0)}{(2k)!}\binom{\frac12}{m-k} - \frac{1}{2\l^2(n)}\sum_{m=0}^{N+1} \frac{\phi^{(2m)}(0)}{(2m)!}\pa{x}{\l(n)}^{2m}\\
		\nonumber
		&\quad + \pa{x}{\l(n)}^{2N+4}\sum_{m=0}^N \pa{x}{\l(n)}^{2m} \sum_{k=m}^N (-1)^{k+m+N+1} \frac{\phi^{(2k+2)}(0)}{(2k+2)!}\binom{\frac12}{m-k+N+1}\\\nonumber
		&\quad + O_{\le \frac{1}{15\sqrt{\pi}}}\left(\pa{x}{\l(n)}^{2N+4}\right)\sum_{m=0}^{N+1} \frac{\phi^{(2m)}(0)}{(2m)!}\pa{x}{\l(n)}^{2m}\\\label{bbnewlem1eqn1}
		&\quad + O_{\le 2.3\pa{2}{3}^{N+2}}\left(\pa{x}{\l(n)}^{2N+4}\right)\sum_{m=0}^{N+1}(-1)^m \binom{\frac12}{m}\pa{x}{\l(n)}^{2m}\\\nonumber
		&\quad + O_{\le 8.7\cdot 10^{-2}\pa{2}{3}^{N+2}}\left(\pa{x}{\l(n)}^{4N+8}\right)-\frac{1}{2\l^{2}(n)}O_{\le 2.3\pa{2}{3}^{N+2}}\left(\pa{x}{\l(n)}^{2N+4}\right).
	\end{align}
	
Note that
\begin{equation}\label{bbnewfact1}
\frac{\phi^{(0)}(0)}{0!}=\left[\cot\left(\frac{\pi}{2}\left(\frac{x}{\sqrt{6}}+\frac 12\right)\right)\right]_{x=0}=1.
\end{equation}	

Using this, we have
\begin{multline}\label{bbnewlem1eqn1A}
	\sum_{m=0}^{N+1}\pa{x}{\l(n)}^{2m}\sum_{k=0}^{m}(-1)^{k+m}\dfrac{\phi^{(2k)}(0)}{(2k)!}\binom{\frac12}{m-k}-\frac{1}{2 \l^2(n)}\sum_{m=0}^{N+1}\dfrac{\phi^{(2m)}(0)}{(2m)!}\pa{x}{\l(n)}^{2m}\\
	=\sum_{m=0}^{N+1}\frac{C^{[1]}_m(x)}{\l^{2m}(n)}-\frac{\phi^{(2N+2)}(0)x^{2N+2}}{2(2N+2)!\l^{2N+4}(n)}.
\end{multline}

Thus, \eqref{bbnewlem1eqn1A} and \eqref{bbnewfact1} imply that for $0\le m \le N+1$,
\begin{equation}\label{bbnewfact2}
C^{[1]}_m(x)=
\begin{cases}
1 &\quad \text{if}\ m=0,\\
 x^{2m}\sum_{k=0}^{m}(-1)^{k+m}\dfrac{\phi^{(2k)}(0)}{(2k)!}\binom{\frac12}{m-k}-\dfrac{\phi^{(2m-2)}(0)}{2(2m-2)!}x^{2m-2} &\quad \text{if}\ m \ge 1.
\end{cases}
\end{equation}

Using $|x|<\l(n) n^{-\frac 18}$, $n\ge10^4$, and \eqref{bbbeqn4}, we conclude
\begin{multline}\label{bbnewlem1eqn2}
O_{\le 8.7\cdot 10^{-2}\pa{2}{3}^{N+2}}\left(\pa{x}{\l(n)}^{4N+8}\right)-\frac{1}{2\l^{2}(n)}O_{\le 2.3\pa{2}{3}^{N+2}}\left(\pa{x}{\l(n)}^{2N+4}\right)\\
= O_{\le 1 }\left(\frac{3.9\cdot 10^{-2}\pa{2}{3}^{N}+ 1.2\cdot 10^{-3}\ 0.4^{N}x^{2N+4}}{n^{\frac{N+2}{2}}}\right).
\end{multline}

	For $n\ge10^4$, using \eqref{prelim2}, \eqref{prelim3}, and \eqref{bbbeqn4}, it follows that
	\begin{equation}\label{bbnewlem1eqn3}
			2.3\pa{2}{3}^{N+2}\pa{x}{\l(n)}^{2N+4}\Biggl|\sum_{m=0}^{N+1}(-1)^m \binom{\frac12}{m}\pa{x}{\l(n)}^{2m}\Biggr| \leq  0.4^{N+1}x^{2N+4}n^{-\frac{N+2}{2}}.
	\end{equation}
	We next bound $\frac{\phi^{(2\ell)}(0)}{(2\ell)!}$ from above. For $\ell\in \mathbb{N}$, comparing the coefficients of $x^{2\ell}$ in \eqref{bbneweqn0} yields
	\[
		\dfrac{\phi^{(2\ell)}(0)}{(2\ell)!}=\dfrac 4\pi\pa{2}{3}^{\ell}+\pa{2}{3}^{\ell}\sum_{m\ge \ell+1}c(2\ell,m).
	\]
	Therefore, using \eqref{bbneweqn1}, for $\ell\in\N$,
	\begin{equation}\label{bbnewlem1eqn3a}
		\dfrac{\left|\phi^{(2\ell)}(0)\right|}{(2\ell)!}\leq  \pa{2}{3}^{\ell}\left(\frac 4\pi+\dfrac{2\sqrt{\pi}}{9\cdot4^\ell\sqrt\ell}\right).
	\end{equation}

 Now for $n\ge 10^4$, we have, employing \eqref{bbnewfact1}, \eqref{bbbeqn4}, \eqref{prelim3}, and \eqref{bbnewlem1eqn3a},
	\begin{equation}\label{bbnewlem1eqn4}
	\frac{1}{15\sqrt{\pi}}\pa{x}{\l(n)}^{2N+4}\Biggl|\sum_{m=0}^{N+1} \frac{\phi^{(2m)}(0)}{(2m)!}\pa{x}{\l(n)}^{2m}\Biggr|\leq 1.6\cdot 10^{-2}\cdot 0.6^{N}x^{2N+4} n^{-\frac{N+2}{2}}.
	\end{equation}

	Next, for $n\ge10^4$, we obtain, using \eqref{bbbeqn4}, \eqref{prelim2}, and \eqref{bbnewlem1eqn3a},
	\begin{multline}\label{bbnewlem1eqn5}
		\pa{x}{\l(n)}^{2N+4}\left|\sum_{m=0}^{N}\pa{x}{\l(n)}^{2m}\sum_{k=m}^{N}(-1)^{k+m+N+1}\dfrac{\phi^{(2k+2)}(0)}{(2k+2)!}\binom{\frac12}{m-k+N+1}\right|\\
		\leq 1.4\cdot 1.3^{-2N}x^{2N+4} n^{-\frac{N+2}{2}}.\hspace{-0.1cm}
	\end{multline}

	Finally, for $N \geq 3$, applying \eqref{bbnewlem1eqn3a} with $\ell=N+1$, we bound, using \eqref{bbbeqn4}, and \eqref{bbnewlem1eqn3a},
	\begin{equation}\label{bbnewlem1eqn6}
		\left|\frac{\phi^{(2N+2)}(0)}{2 (2N+2)!}\frac{x^{2N+2}}{\l^{2N+4}(n)}\right| \le 0.2\cdot  0.4^{N}x^{2N+2}n^{-\frac{N+2}{2}}.
	\end{equation}

Applying \eqref{bbnewlem1eqn6}, \eqref{bbnewlem1eqn5}, \eqref{bbnewlem1eqn4}, \eqref{bbnewlem1eqn3}, \eqref{bbnewlem1eqn2}, and \eqref{bbnewlem1eqn1A} to \eqref{bbnewlem1eqn1}, we obtain
	\begin{equation}\label{bbnewlem1part1eqnfinal}
		P_{n}^{[1]}(x) P_{n}^{[2]}(x)\!=\!\sum_{m=0}^{N+1}\!\frac{C^{[1]}_m(x)}{\l^{2m}(n)}\!+\!O_{\le 1}\!\left(\frac{1.2\cdot 10^{-2}\!+\!0.2\cdot 0.4^{N}x^{2N+2}\!+\!1.5\cdot 0.6^{N}x^{2N+4}}{n^{\frac{N+2}{2}}}\right)\!.\hspace{-0.2cm}
	\end{equation}

	Now, multiplying \eqref{bbnewlem1part1eqnfinal} with $P_{n}^{[4]}$, we have
	\begin{align}
		P_{n}^{[1]}(x) P_{n}^{[2]}(x) P_{n}^{[4]}&=\sum_{m=0}^{N+1}\frac{C^{[1]}_m(x)}{\l^{2m}(n)}+ \sum_{m=0}^{N+1}\frac{C^{[1]}_m(x)}{\l^{2m}(n)}O_{\le 0.4}\left(n^{-\frac{N+2}{2}}\right)\nonumber\\
		&\hspace{0.2cm}+ O_{\le 1}\left(\frac{1.3\cdot 10^{-2}+0.2\cdot 0.4^{N}x^{2N+2}+1.5\cdot 0.6^{N}x^{2N+4}}{n^{\frac{N+2}{2}}}\right).\label{bbnewlem1eqn7}
	\end{align}
	For $m\in\N$, we have $\l^{-2}(n)\le\frac{1}{13^2}$ using \eqref{bbbeqn4} with $n\ge 10^4$. Combining this with \eqref{bbnewfact2}, \eqref{bbnewfact1}, \eqref{prelim2}, \eqref{prelim3}, \eqref{bbnewlem1eqn3a}, and the fact that 
$|\smallbinom{\frac 12}{m-k}|\le1$ for $1\le k\le m$, we bound
	\begin{align}\label{bbnewlem1eqn8}
		\left|\sum_{m=0}^{N+1}\! \frac{C^{[1]}_m(x)}{\l^{2m}(n)}\right|&\!\le\! 1\!+\!\frac{1}{2\cdot 13^2}\!\Biggl(1\!+\!10\sum_{m=2}^{N+1}\!\frac{\left|\phi^{(2m-2)}(0)\right| 10^{-m}}{ (2m-2)!}\Biggr)\!+\!\frac 12\!\sum_{m=1}^{N+1}\!\frac{10^{-m}}{m^{\frac 32}}\!+\!\sum_{m=1}^{N+1}\!10^{-m}\sum_{k=1}^{m}\!\dfrac{\left|\phi^{(2k)}(0)\right|}{(2k)!}\nonumber\\
	&\leq 1.4.
	\end{align}
Therefore, we obtain
	\begin{equation*}
		0.4n^{-\frac{N+2}{2}}\left|\sum_{m=0}^{N+1}\frac{C^{[1]}_m(x)}{\l^{2m}(n)}\right| \leq 0.6 n^{-\frac{N+2}{2}}.
	\end{equation*}
	Combining this with \eqref{bbnewlem1eqn7}, we conclude the statement.
\end{proof}

We next approximate $P_{n}^{[3]}(x)$. For this define
\begin{equation}\label{pndef}
p(x)=p_N(x):=\sum_{j=0}^{N}\frac{\left(\zeta\left(\frac 32\right)-1\right)^j}{j!}x^{2j}.
\end{equation}

\begin{lemma}\label{bbnewlem2}
For $N \geq 3$ and $n\ge 10^4$, we have 
	\begin{equation*}
		P_{n}^{[3]}(x)=\sum_{m=0}^{N+1}\frac{C^{[2]}_m(x)}{\l^{2m}(n)}+O_{\le 1}\left(\frac{0.2\cdot 1.3^{-2N}x^{2N+6}\left(1+3p(x)\right)+\frac{8.8\cdot 10^{-2} 0.3^{N}}{(N+2)!}x^{4N+8}}{n^{\frac{N+2}{2}}}\right),
	\end{equation*}
	where
	\[
		C^{[2]}_m(x):=x^{2m}\sum_{j=0}^{m}\frac{2^j\mathcal{C}_{j,m}}{j!}x^{2j},
	\]
	with
	\begin{equation}\label{bbnewlem2def}
		\CC_{j,m} := \sum_{\substack{\ell_1,\dots,\ell_{N+1}\ge0\\\sum_{k=1}^{N+1}\ell_k=j\\\sum_{k=1}^{N+1}k\ell_k=m}} \binom{j}{\ell_1,\dots,\ell_{N+1}}\prod_{k=1}^{N+1} \left((-1)^{k+1}\binom{\frac12}{k+1}\right)^{\ell_k}.
	\end{equation}
\end{lemma}
\begin{proof}
	Recalling the definition of $P_{n}^{[3]}(x)$, we have
	\begin{equation}\label{bbnewlem2eqn1}
		\hspace{-0.05cm}P_{n}^{[3]}(x) = \sum_{j=0}^{N+1} \frac{y_n^j\pa{x}{\l(n)}}{j!} + O_{\le \frac{8.8\cdot 10^{-2} 0.3^{N}}{(N+2)!}}\left(\frac{x^{4N+8}}{n^{\frac{N+2}{2}}}\right).
	\end{equation}

For $k\in\N$ and $n\ge10^4$, we have, using \eqref{bbbeqn4}, $|x|\le \l(n) n^{-\frac 18}$, \eqref{deftm}, and  \eqref{prelim2}
\begin{equation}\label{bbnewlem2eqn1a}
\left|\frac{t_k(x)}{\l^{2k}(n)}\right| \le  \frac7{10^{k}}.
\end{equation}
Setting $X_k:=\frac{t_k(x)}{\l^{2k}(n)}$ for $1\le k\le N+1$ and $X_{N+2}:=\frac{6\cdot 10^{-2}x^{2N+6}}{\l^{2N+4}(n)}$, we rewrite
\begin{align}\nonumber 
\sum_{j=0}^{N+1} \frac{1}{j!}\left(\sum_{k=1}^{N+2} X_k\right)^j &= \sum_{j=0}^{N+1} \frac{1}{j!}\sum_{\ell=0}^j \binom j\ell\left(\sum_{k=1}^{N+1} X_k\right)^\ell X_{N+2}^{j-\ell}\\\nonumber
&= \sum_{j=0}^{N+1} \frac{1}{j!}\sum_{\substack{\ell_1,\dots,\ell_{N+1}\ge0\\\sum_{k=1}^{N+1}\ell_k=j\\0\le\sum_{k=1}^{N+1}k\ell_k\le N+1}} \binom{j}{\ell_1,\dots,\ell_{N+1}}\prod_{k=1}^{N+1} t^{\ell_k}_k(x)\l(n)^{-2\sum\limits_{k=1}^{N+1} k\ell_k}
\nonumber\\
\nonumber 
&\hspace{2cm}+ \sum_{j=0}^{N+1} \frac{1}{j!}\sum_{\substack{\ell_1,\dots,\ell_{N+1}\ge0\\\sum_{k=1}^{N+1}\ell_k=j\\\sum_{k=1}^{N+1}k\ell_k\ge N+2}} \binom{j}{\ell_1,\dots,\ell_{N+1}}\prod_{k=1}^{N+1} t^{\ell_k}_k(x)\l(n)^{-2\sum\limits_{k=1}^{N+1} k\ell_k}\\
\label{bbnewlem2eqn2}
&\hspace{2cm}+ \sum_{j=1}^{N+1} \frac{1}{j!}\sum_{\ell=0}^{j-1} \binom j\ell\left(\sum_{k=1}^{N+1} \frac{t_k(x)}{\l^{2k}(n)}\right)^\ell\left(\frac{6\cdot 10^{-2}x^{2N+6}}{\l^{2N+4}(n)}\right)^{j-\ell}.
\end{align}
Using $|x|\le \l(n) n^{-\frac 18}$, \eqref{bbbeqn4}, \eqref{bbnewlem2eqn1a}, and bounding $\sum_{j=1}^{N+1}\frac{1}{j!}(\frac{7}{9})^{j}\ge \frac79$, we estimate
	\begin{equation}\label{bbnewlem2eqn3}
	\sum_{j=1}^{N+1} \frac{1}{j!}\sum_{\ell=0}^{j-1} \binom j\ell\left(\sum_{k=1}^{N+1} \frac{t_k(x)}{\l^{2k}(n)}\right)^\ell\left(\frac{6\cdot 10^{-2}x^{2N+6}}{\l^{2N+4}(n)}\right)^{j-\ell}\le   \frac{0.2x^{2N+6}}{1.3^{2N}n^{\frac{N+2}{2}}}.
	\end{equation}
	Using \eqref{deftm}, \eqref{prelim2}, $|x|\le \l(n)n^{-\frac 18}$, \eqref{bbbeqn4}, and \eqref{pndef}, we bound 
	\begin{align}\label{bbnewlem2eqn4}
		\left|\sum_{j=0}^{N+1}\! \frac{1}{j!}\!\!\!\!\sum_{\substack{\ell_1,\dots,\ell_{N+1}\ge0\\\sum_{k=1}^{N+1}\ell_k=j\\\sum_{k=1}^{N+1}k\ell_k\ge N+2}}\!\!\!\!\!\! \binom{j}{\ell_1,\dots,\ell_{N+1}}\!\!\prod_{k=1}^{N+1} t_k^{\ell_k}(x)\l(n)^{-2\sum\limits_{k=1}^{N+1} k\ell_k}\right|
		&\le 0.6\cdot 1.3^{-2N}p(x)x^{2N+6} n^{-\frac{N+2}{2}}.
	\end{align}

Applying \eqref{bbnewlem2eqn3} and \eqref{bbnewlem2eqn4} to \eqref{bbnewlem2eqn2}, and then combining with \eqref{bbnewlem2eqn1} using \eqref{bbneweqn01}, 
	\begin{align*}
		P_{n}^{[3]}(x) &= \sum_{j=0}^{N+1} \frac{1}{j!}\hspace{-0.2cm}\sum_{\substack{\ell_1,\dots,\ell_{N+1}\ge0\\\sum_{k=1}^{N+1}\ell_k=j\\0\le\sum_{k=1}^{N+1}k\ell_k\le N+1}}\hspace{-0.2cm} \binom j{\ell_1,\dots,\ell_{N+1}} \prod_{k=1}^{N+1} t_k^{\ell_k}(x)\l(n)^{-2\sum\limits_{k=1}^{N+1} k\ell_k}\\[-30pt]\\
		& \hspace{5 cm} + O_{\le 1}\left(\frac{0.2\cdot 1.3^{-2N}x^{2N+6}\left(1+3p(x)\right)+\frac{8.8\cdot 10^{-2}0.3^{N}}{(N+2)!}x^{4N+8}}{n^{\frac{N+2}{2}}}\right).
	\end{align*}
	Using $\sum_kk\ell_k\ge\sum_k\ell_k=j$, we simplify the multinomial sum as 
	\begin{equation*}
		\sum_{j=0}^{N+1} \frac{1}{j!}\sum_{\substack{\ell_1,\dots,\ell_{N+1}\ge0\\\sum_{k=1}^{N+1}\ell_k=j\\j\le\sum_{k=1}^{N+1}k\ell_k\le N+1}} \binom{j}{\ell_1,\dots,\ell_{N+1}}\prod_{k=1}^{N+1} t_k^{\ell_k}(x)\l(n)^{-2\sum\limits_{k=1}^{N+1} k\ell_k}
		=\sum_{j=0}^{N+1} \frac{2^j}{j!}x^{2j}\sum_{m=j}^{N+1} \CC_{j,m}\pa{x}{\l(n)}^{2m}.
	\end{equation*}
	Combining gives the claim.
\end{proof}

Next, we combine Lemmas \ref{bbnewlem1} and \ref{bbnewlem2} to get the following approximation.

\begin{lemma}\label{bbnewlem3}
	For $N\geq 3$ and $n\ge 10^4$,
	\[
	P^{[1]}_{n}(x) P^{[2]}_{n}(x) P^{[3]}_{n}(x) P^{[4]}_{n}=\sum_{m=0}^{N+1}\dfrac{C_m(x)}{\l^{2m}(n)}+O_{\le 1}\left(\frac{\mathcal{P}^{[1]}(x)}{n^{\frac{N+2}{2}}}\right),
	\]
	where \text{$\left(\text{with}\ p(x)\ \text{be as in}\ \eqref{pndef}\right)$}
	\begin{align}\label{bbnewlem3neweqn}
		C_m(x)&:=\sum_{k=0}^{m}C^{[1]}_k(x)C^{[2]}_{m-k}(x),\\
		\mathcal{P}^{[1]}(x)&=\!\mathcal{P}^{[1]}_N(x)\!:=\!182.7\!+\!4.5\cdot 1.4^{N}\!x^2\!+\!1.5p(x)\!+\!13.3\cdot 1.4^N\!p(x)x^2\!+\!192.3\cdot 1.3^{-2N}\!x^{2N+2}\!\nonumber\\
		&+\!0.7^{N}\!\left(1275.1\!+\!5.2p(x)\right)\!x^{2N+4}\!+\!1.3^{-2N}\!\left(0.3\!+\!17.1(N+1)p(x)\right)\!x^{2N+6}\!+\!\frac{0.2\cdot0.3^{N}\!x^{4N+8}}{(N+2)!}.\nonumber
	\end{align}
\end{lemma}

\begin{proof}
	From Lemmas \ref{bbnewlem1} and \ref{bbnewlem2}, we have, using \eqref{cauchyproduct}, 
	\begin{align}\nonumber
		P^{[1]}_{n}(x) P^{[2]}_{n}(x) P^{[3]}_{n}(x) P^{[4]}_{n}&=  \sum_{m=0}^{N+1}\frac{C_m(x)}{\l^{2m}(n)}+\l^{-2N-4}(n)\sum_{m=0}^{N}\sum_{k=m}^{N}\dfrac{C^{[1]}_{k+1}(x)C^{[2]}_{m-k+N+1}(x)}{\l^{2m}(n)}\\\nonumber
		&\hspace{0.1 cm}+\sum_{m=0}^{N+1}\dfrac{C^{[1]}_m(x)}{\l^{2m}(n)}O_{\le 1}\!\left(\frac{0.2\cdot 1.3^{-2N}\!x^{2N+6}\!\left(1\!+\!3p(x)\right)\!+\!\frac{8.8\cdot 10^{-2} 0.3^{N}}{(N+2)!}x^{4N+8}}{n^{\frac{N+2}{2}}}\right)\\\nonumber
		&\hspace{0.1 cm}+\sum_{m=0}^{N+1}\dfrac{C^{[2]}_m(x)}{\l^{2m}(n)}O_{\le 1}\!\left(\frac{0.7+0.2\cdot 0.4^{N}x^{2N+2}+1.5\cdot 0.6^{N}x^{2N+4}}{n^{\frac{N+2}{2}}}\right)\\\nonumber
		&\hspace{0.1 cm}+ O_{\le 1}\!\left(\frac{0.7+0.2\cdot 0.4^{N}x^{2N+2}+1.5\cdot 0.6^{N}x^{2N+4}}{n^{\frac{N+2}{2}}}\right)\\\label{bbnewlem3eqn1}
		&\hspace{0.1 cm}\times O_{\le 1}\!\left(\frac{0.2\cdot 1.3^{-2N}\!x^{2N+6}\!\left(1\!+\!3p(x)\right)\!+\!\frac{8.8\cdot 10^{-2} 0.3^{N}}{(N+2)!}x^{4N+8}}{n^{\frac{N+2}{2}}}\right).
	\end{align} 
Using $|x|\le \l(n) n^{-\frac 18}$ and \eqref{bbbeqn4}, for $n\ge 10^4$, we bound the product of two error terms in the last two line of \eqref{bbnewlem3eqn1} against
	\begin{align}\label{bbnewlem3eqn2}
		&O_{\le 1}\!\left(\frac{0.6\!+\!1.5p(x)\!+\!4.5\cdot 1.4^N\! x^2\!+\!13.3\cdot 1.4^{N}\!p(x)x^2\!+\!\frac{0.3\cdot 0.5^{N}}{(N+2)!}x^{2N+2}\!+\!\frac{2\cdot 0.7^{N}}{(N+2)!}x^{2N+4}}{n^{\frac{N+2}{2}}}\right).
	\end{align}
	From Lemma \ref{bbnewlem2}, for $(m,j)\in \mathbb{N}^2_0$ we have, using \eqref{prelim2},
	\begin{equation}\label{bbnewlem3eqn2ao}
	\left|\CC_{j,m}\right| \le \left(\frac12\sum_{k=1}^{N+1}(k+1)^{-\frac32}\right)^j.
	\end{equation}
Using the definition of $C^{[2]}_m(x)$,  \eqref{bbnewlem3eqn2ao}, $|x|\le \l(n) n^{-\frac 18}$, and \eqref{bbbeqn4}, we have for $n\ge 10^4$, 
	\begin{align}\nonumber
		\frac{0.7+0.2\cdot 0.4^{N}x^{2N+2}+1.5\cdot 0.6^{N}x^{2N+4}}{n^{\frac{N+2}{2}}}\left|\sum_{m=0}^{N+1}\dfrac{C^{[2]}_m(x)}{\l^{2m}(n)}\right|&\\\label{bbnewlem3eqn3}
		&\hspace{-3.5 cm} \le \frac{182.1+52.1\cdot 0.4^{N}x^{2N+2}+390.2\cdot 0.6^{N}x^{2N+4}}{n^{\frac{N+2}{2}}}.
	\end{align}
By \eqref{bbnewlem1eqn8}, we get 
	\begin{align}\label{bbnewlem3eqn4}
	&\frac{0.2\cdot\! 1.3^{-2N}\!x^{2N+6}\!\left(1\!+\!3p(x)\right)\!+\!\frac{8.8\cdot 10^{-2} 0.3^{N}}{(N+2)!}\!x^{4N+8}}{n^{\frac{N+2}{2}}}
	 \left|\sum_{m=0}^{N+1}\dfrac{C^{[1]}_m(x)}{\l^{2m}(n)}\right|\nonumber\\
	&\hspace{6.5 cm}\le\!\frac{0.3\cdot\! 1.3^{-2N}\!x^{2N+6}\!\left(1\!+\!3p(x)\right)\!+\!\frac{0.2\cdot  0.3^{N}}{(N+2)!}\!x^{4N+8}}{n^{\frac{N+2}{2}}}.
	\end{align}
Using \eqref{bbbeqn4}, we first bound the remaining sum as
	\begin{align}\nonumber
		&\left|\l^{-2N-4}(n)\sum_{m=0}^{N}\sum_{k=m}^{N}\dfrac{C^{[1]}_{k+1}(x)C^{[2]}_{m-k+N+1}(x)}{\l^{2m}(n)}\right|\\\label{bbnewlem3eqn5}
		&\hspace{3 cm}\le 0.4\cdot 1.3^{-2N}n^{-\frac{N+2}{2}}\sum_{m=0}^{N}\sum_{k=m}^{N}\frac{\left|C^{[1]}_{k+1}(x)\right|\left|C^{[2]}_{m-k+N+1}(x)\right|}{\l^{2m}(n)}.
	\end{align}
	Consequently, using the definition of $C_{m}^{[1]}(x)$ from Lemma \ref{bbnewlem1}, for $k\in \mathbb{N}_0$, we estimate
	\begin{equation}\label{bbnewlem3eqn6}
		\left|C^{[1]}_{k+1}(x)\right|\le x^{2k+2}\sum_{\ell=0}^{k+1}\frac{\left|\phi^{(2\ell)}(0)\right|}{(2\ell)!}+ \dfrac{\left|\phi^{(2k)}(0)\right|}{(2k)!}x^{2k}.
	\end{equation}
For $0\le m\le k\le N$, using the definition of $C_{m}^{[2]}(x)$ from Lemma \ref{bbnewlem2}, \eqref{bbnewlem3eqn2ao}, \eqref{pndef}, $|x|\le \l(n) n^{-\frac 18}$, and \eqref{bbbeqn4}, we have, for $n\ge 10^4$, 
\begin{align}\label{bbnewlem3eqn7}
\frac{\left|C^{[2]}_{m-k+N+1}(x)\right|}{\l^{2m}(n)}\le  e^4\pa{4}{5}^m x^{2N+2-2k}+\frac{\left(\zeta\left(\frac 32\right)-1\right)^{m+1}}{(m+1)!\l^{2m}(n)}x^{4m-2k+2N+4}p(x).
\end{align}
Combining \eqref{bbnewlem3eqn6} and \eqref{bbnewlem3eqn7} and using $|x|\le \l(n) n^{-\frac 18}$ and \eqref{bbbeqn4}, we get
\begin{align}\nonumber
&\frac{\left|C^{[1]}_{k+1}(x)\right|\left|C^{[2]}_{m-k+N+1}(x)\right|}{\l^{2m}(n)} \le e^4 \pa{4}{5}^m \sum_{\ell=0}^{k+1}\frac{\left|\phi^{(2\ell)}(0)\right|}{(2\ell)!} x^{2N+4}+e^4 \pa{4}{5}^m\dfrac{\left|\phi^{(2k)}(0)\right|}{(2k)!}x^{2N+2}\\\nonumber
&\hspace{3.5 cm}+\left(\zeta\left(\frac 32\right)-1\right)\frac{\left(1.4^2\left(\zeta\left(\frac 32\right)-1\right)\right)^m}{(m+1)!}\sum_{\ell=0}^{k+1}\frac{\left|\phi^{(2\ell)}(0)\right|}{(2\ell)!}x^{2N+6}p(x)\\\label{bbnewlem3eqn8}
&\hspace{3.5 cm}+\left(\zeta\left(\frac 32\right)-1\right)\frac{\left(1.4^2\left(\zeta\left(\frac 32\right)-1\right)\right)^m\left|\phi^{(2k)}(0)\right|}{(m+1)!(2k)!}x^{2N+4}p(x).
\end{align}
Next, using \eqref{bbnewfact1}, \eqref{bbnewlem1eqn3a}, and \eqref{pndef}, we bound 
\begin{align}
&e^4\sum_{m=0}^{N}\pa{4}{5}^m\sum_{k=m}^{N}   \sum_{\ell=0}^{k+1}\frac{\left|\phi^{(2\ell)}(0)\right|}{(2\ell)!}x^{2N+4}\le 989.7 (N+1) x^{2N+4},\label{bbnewlem3eqn9}\\
&e^4\sum_{m=0}^{N}\pa{4}{5}^m\sum_{k=m}^{N}  \dfrac{\left|\phi^{(2k)}(0)\right|}{(2k)!}x^{2N+2}\le 440.3 x^{2N+2},\label{bbnewlem3eqn10}\\
&\left(\zeta\!\left(\tfrac 32\right)\!-\!1\right)\sum_{m=0}^{N}\tfrac{\left(1.4^2\left(\zeta\left(\frac 32\right)-1\right)\right)^m}{(m+1)!}\sum_{k=m}^{N} \sum_{\ell=0}^{k+1}\frac{\left|\phi^{(2\ell)}(0)\right|}{(2\ell)!}x^{2N+6}p(x)
\le 41.8 (N\!+\!1) x^{2N+6} p(x),\!\label{bbnewlem3eqn11}\\
& \left(\zeta\!\left(\tfrac 32\right)-1\right)\sum_{m=0}^{N}\frac{\left(1.4^2\left(\zeta\left(\frac 32\right)-1\right)\right)^m}{(m+1)!}\sum_{k=m}^{N}\dfrac{\left|\phi^{(2k)}(0)\right|}{(2k)!}x^{2N+4}p(x)\le 21.1 x^{2N+4}p(x).\label{bbnewlem3eqn12}
\end{align}
Thus applying \eqref{bbnewlem3eqn9}, \eqref{bbnewlem3eqn10}, \eqref{bbnewlem3eqn11}, \eqref{bbnewlem3eqn12} to \eqref{bbnewlem3eqn8} and \eqref{bbnewlem3eqn5}, we have for $n\ge 10^4$,
\begin{align}\label{bbnewlem3eqn13}
&\left|\l^{-2N-4}(n)\sum_{m=0}^{N}\sum_{k=m}^{N}\dfrac{C^{[1]}_{k+1}(x)C^{[2]}_{m-k+N+1}(x)}{\l^{2m}(n)}\right|\\\nonumber
&\hspace{0.4cm}\le 1.3^{-2N}\left(176.2 x^{2N+2}+\left(395.9(N+1)+8.5 p(x)\right)x^{2N+4}+16.8(N+1)x^{2N+6}p(x)\right)n^{-\frac{N+2}{2}}.
\end{align}

Applying \eqref{bbnewlem3eqn13}, \eqref{bbnewlem3eqn4}, \eqref{bbnewlem3eqn3}, \eqref{bbnewlem3eqn2} to \eqref{bbnewlem3eqn1}, we conclude the proof.
\end{proof}

Next, we express $\sum_{m=0}^{N+1}C_m(x)\l^{-2m}(n)$ as a polynomial in $x$. Using \eqref{bbnewlem3neweqn}, the definitions of $C^{[1]}_m(x)$ and $C^{[2]}_m(x)$ from Lemma \ref{bbnewlem1} and Lemma \ref{bbnewlem2}, respectively, we obtain 
\begin{align}\nonumber
\sum_{m=0}^{N+1}\frac{C_m(x)}{\l^{2m}(n)}&\!=\!1\!+\!\sum_{m=1}^{N+1}\!\frac{1}{\l^{2m}(n)}\!\sum_{j=0}^{m}\!\frac{2^j\CC_{j,m}}{j!}\!x^{2j+2m}\!-\!\sum_{m=1}^{N+1}\!\frac{1}{\l^{2m}(n)}\!\sum_{k=1}^{m}\!\frac{\phi^{(2k\!-\!2)}(0)}{2(2k-2)!}\!\sum_{j=0}^{m-k}\!\frac{2^j\CC_{j,m-k}}{j!}\!x^{2j+2m-2}\\\label{bbremark1eqn1}
&\hspace{1.4 cm}+\!\sum_{m=1}^{N+1}\frac{1}{\l^{2m}(n)}\sum_{k=1}^{m}\sum_{\ell=0}^{k}(-1)^{\ell+k}\frac{\phi^{(2\ell)}(0)}{(2\ell)!}\binom{\frac 12}{k-\ell}\sum_{j=0}^{m-k}\frac{2^j\CC_{j,m-k}}{j!}x^{2j+2m}.
\end{align}

\begin{remark}\label{bbremark1}
	Note that, for $0\le m\le N$, $C_m(x)\in \mathbb{R}[x]$ has degree at most $4m$. Thus the polynomial $\sum_{m=0}^{N+1}C_m(x)\l^{-2m}(n)$ in $x$  has degree at most $4N+4$ and is even.
\end{remark}

For $n\ge 10^4$, applying \Cref{bbnewlem3} gives that \eqref{bbbeqn11} equals 
\begin{equation}\label{neweqn1}
\frac{e^{\frac{\pi\sqrt{24n+1}}{3\sqrt2}}}{(24n+1)\l(n)} \int_{-\l(n)n^{-\frac18}}^{\l(n)n^{-\frac18}} \left(\sum_{m=0}^{N+1}\dfrac{C_m(x)}{\l^{2m}(n)}+O_{\le 1}\left(\frac{\mathcal{P}^{[1]}(x)}{n^{\frac{N+2}{2}}}\right)\right)e^{-x^2}dx.
\end{equation}

We next approximate the integral in \eqref{neweqn1}. For $m\in \mathbb{N}_0$ even, we have 
\begin{equation}\label{newfact}
\int_{-\infty}^{\infty}x^m e^{-x^2}dx=
\Gamma\left(\frac{m+1}2\right).
\end{equation}
Moreover for $m\in \mathbb{N}_0$, we have
\begin{equation}\label{newfact0}
\Gamma\left(m+\frac{1}2\right) = \sqrt\pi\frac{m!}{4^m}\binom{2m}{m}.
\end{equation}
Thus using the definition of $\mathcal{P}^{[1]}(x)$, \eqref{newfact}, \eqref{newfact0}, and \Cref{prelim1}, we bound 
\begin{equation}\label{newerrorestim}
\int_{-\l(n)n^{-\frac18}}^{\l(n)n^{-\frac18}}  \mathcal{P}^{[1]}(x) e^{-x^2}dx\le E^{[0]}_N,
\end{equation}
where
\begin{align}\label{newdeferror1}
E^{[0]}_N&:=326.6+4\cdot 1.4^N+19.4\cdot \!1.2^N(N+3)!+ 1.5 S(0)+13.3\cdot 1.4^N S(1)\!\nonumber\\
&\hspace{6 cm}+\!0.7^{N}\!\Bigl(5.2 S(N+2)\!+\!44.5\ S(N+3)\Bigr).
\end{align}
Here 
\begin{equation*}
 S(m):=S_N(m)=\begin{cases}
 \sum\limits_{j=0}^{N}\left(\zeta\left(\frac 32\right)-1\right)^j & \quad \text{if}\ m=0,\\
 \frac{m!}{\sqrt{m}}\sum\limits_{j=0}^{N}\binom{j+m}{j}\left(\zeta\left(\frac 32\right)-1\right)^j & \quad \text{if}\ m\in \mathbb{N}.
 \end{cases}
\end{equation*}

Finally $n\ge10^4$, following \eqref{newerrorestim}, \eqref{neweqn1} becomes
\begin{equation}\label{bbneweqn4}
\frac{e^{\frac{\pi\sqrt{24n+1}}{3\sqrt2}}}{(24n+1)\l(n)} \left(\sum_{m=0}^{N+1}\l^{-2m}(n)\int_{-\l(n)n^{-\frac18}}^{\l(n)n^{-\frac18}}C_m(x) e^{-x^2}dx+O_{\le E^{[0]}_N}\left(n^{-\frac{N+2}{2}}\right)\right).
\end{equation}

Using \eqref{newfact} and \cite[equation (8.2.2)]{DLMF}, we have, for $w\in\R_{\ge2}$ and $m\in\N_0$ even, 
\begin{equation*}
\left|\Gamma\left(\frac{m+1}2\right)-\int_{-w}^{w}x^m e^{-x^2}dx\right|\leq  \Gamma\left(\frac{m+1}2\right)w^m e^{-w^2}.
\end{equation*}

Define 
\begin{equation}\label{cutoff}
n_N:=\left(\frac{3(3N+4)\log (6N+8)}{1.3^2}\right)^4,
\end{equation}
and for, $\ell\in \mathbb{N}_0$ even,
\begin{align}\label{bbnewdef1}
\CM_{r,\ell}:=\CM_{r,\ell,N} &=
\sup\left\{\frac{1.4^\ell}{1.3^r}\Gamma\left(\frac{\ell+1}{2}\right) n^{\frac \ell8-\frac r4+\frac{N+2}{2}}e^{-1.3^2n^\frac14} : n\ge n_{N}\right\}.
\end{align}
Consequently, following 
\eqref{bbbeqn4}, for  $\ell\ge 0$ and $n\ge n_N$, consider the following approximation
\begin{equation}\label{bbbeqn10new}
\l^{-2r}(n)\int_{-\l(n)n^{-\frac18}}^{\l(n)n^{-\frac18}} x^{2\ell}e^{-x^2} dx =
\frac{\Gamma\left(\ell+\frac 12\right)}{\l^{2r}(n)} + O_{\le\CM_{2r,2\ell}}\left(n^{-\frac{N+2}{2}}\right). 
\end{equation}

Next, we give an upper bound for $\CM_{2r,2\ell}$ for $n\ge n_N$, $0\le r \le N+1$, and $0\le \ell\le 2N+2$.
\begin{lemma}\label{bbnewlem3a}
For $N\ge 3$, $n \ge n_N$, $0\le r \le N+1$, and $0\le \ell\le 2N+2$, we have
	\begin{equation*}
		\CM_{2r,2\ell} \le
		\begin{cases}
			\frac{0.2^\ell}{1.3^{2r}} &\quad \text{if}\ \ell\geq 1,\\
			\frac{0.7}{1.3^{2r}}&\quad \text{if}\ \ell=0.
		\end{cases}
	\end{equation*}
\end{lemma}
\begin{proof}
	We distinguish two cases: $\ell\in \mathbb{N}$ and $\ell=0$.

	We start with $1\le \ell \le 2N+2$, $0\le r\le N+1$. We have, using Lemma \ref{prelim1}, bounds for $n!$ (see \cite[(1), (2)]{Ro}), and the fact that $(\frac{1.6(2N+2)}{(9N+12)\log(6N+8)})_{N\ge3}$ is decreasing, that 
	\begin{align}\label{bbnewlem3aeqn1}
	\dfrac{1.4^{2\ell}}{1.3^{2r}}\Gamma\left(\ell+\frac 12\right)n^{\frac\ell4-\frac r2+\frac{N+2}{2}}e^{-1.3^2n^{\frac 14}}
	&\le \dfrac{0.1^\ell}{1.3^{2r}} \left(1.3^2n^{\frac 14}\right)^{6N+8}e^{-1.3^2n^{\frac 14}}.
	\end{align}

	Next we note that for $m\ge 26$,
	\begin{equation}\label{bbnewlem3aeqn2}
		y^me^{-y}<1\quad \text{for}\ y\ge \dfrac{153m\log(m)}{103}.
	\end{equation}
	Using \eqref{bbnewlem3aeqn2} with $(y,m)=(1.3^2 n^{\frac 14}, 6N+8)$ (note that $6N+8 \ge 26$), we have
	\begin{equation}\label{bbnewlem3aeqn3a}
	\left(1.3^2n^{\frac 14}\right)^{6N+8}e^{-1.3^2n^{\frac 14}}<1
	\end{equation}
	 for $n\ge (\frac{306(3N+4)\log(6N+8)}{103\cdot1.3^2})^4$. Consequently, applying \eqref{bbnewlem3aeqn3a} to \eqref{bbnewlem3aeqn1}, we have for $n\ge n_N$, 
	\begin{equation}\label{bbnewlem3aeqn3}
	\dfrac{1.4^{2\ell}}{1.3^{2r}}\Gamma\left(\ell+\frac 12\right)n^{\frac\ell4-\frac r2+\frac{N+2}{2}}e^{-1.3^2n^{\frac 14}}< \dfrac{0.1^\ell}{1.3^{2r}}.
	\end{equation}

	Next, we consider $\ell=0$ and give an upper bound for $\CM_{2r,0}$ with $0\le r\le N+1$. We have 
	\begin{equation}\label{bbnewlem3aeqn4}
	\dfrac{ \sqrt{\pi}}{1.3^{2r}}n^{-\frac r2+\frac{N+2}{2}}e^{-1.3^2n^{\frac 14}} \le \dfrac{0.7}{1.3^{2r}}\left(1.3^2n^{\frac 14}\right)^{2N+4}e^{-1.3^2n^{\frac 14}}.
	\end{equation}
	Analogous to \eqref{bbnewlem3aeqn2}, we have for $m \geq 10$,
	\begin{equation}\label{bbnewlem3aeqn5}
	y^me^{-y}<1\quad \text{for}\ y\ge \dfrac{14m\log(m)}{9}.
	\end{equation}
	Applying \eqref{bbnewlem3aeqn5} with $(y,m)=(1.3^2 n^{\frac 14}, 2N+4)$ (as $2N+4\ge10$), we obtain
	\begin{equation}\label{bbnewlem3aeqn6a}
	\left(1.3^2n^{\frac 14}\right)^{2N+4}e^{-1.3^2n^{\frac 14}}<1
	\end{equation}
	 for $n \ge (\frac{28(N+2)\log(2N+4)}{9\cdot1.3^2})^4$. Thus applying \eqref{bbnewlem3aeqn6a}	to \eqref{bbnewlem3aeqn4}, we have 
	\begin{equation}\label{bbnewlem3aeqn6}
	\dfrac{\sqrt{\pi}}{1.3^{2r}}n^{-\frac r2+\frac{N+2}{2}}e^{-1.3^2n^{\frac 14}}\le \dfrac{0.7}{1.3^{2r}},
	\end{equation}
	for $0\le r\le N+1$ and $n \ge n_N$. Combining \eqref{bbnewlem3aeqn3} and \eqref{bbnewlem3aeqn6}, for $n\ge n_N$, and recalling the definition of $\CM_{2r,2\ell}$ from \eqref{bbnewdef1} implies the claim. 
\end{proof}
Using \Cref{bbnewlem3a} and \eqref{bbbeqn10new}, we obtain
\begin{equation}\label{bbbeqn11new}
I_{r,\ell}(n):=\l^{-2r}(n)\int_{-\l(n)n^{-\frac18}}^{\l(n)n^{-\frac18}} x^{2\ell}e^{-x^2} dx= \frac{\Gamma\left(\ell+\frac 12\right)}{\l^{2r}(n)}+
\begin{cases}
O_{\le \frac{0.1^\ell}{1.3^{2r}}}\left(n^{-\frac{N+2}{2}}\right) &\quad \text{if}\ \ell\ge 1,\\\\[-8pt]
O_{\le \frac{0.7}{1.3^{2r}}}\left(n^{-\frac{N+2}{2}}\right)&\quad \text{if}\ \ell=0.
\end{cases}
\end{equation}

We bound the finite sum in \eqref{bbneweqn4} in the following lemma.

\begin{lemma}\label{bbnewlem4}
	For $N\ge 3$ and $n \geq n_N$, we have 
	\[
		\sum_{m=0}^{N+1}\l^{-2m}(n)\int_{-\l(n)n^{-\frac18}}^{\l(n)n^{-\frac18}} C_m(x)e^{-x^2}dx=\sum_{m=0}^{N+1}\dfrac{c(m)}{\l^{2m}(n)}+O_{\le 1.8}\left(n^{-\frac{N+2}{2}}\right),
	\]
	where $c(0):=\sqrt{\pi}$, $c(1):=\frac{\sqrt{\pi}(2\pi^2-45)}{48}$,
	and for $m\in\N_{\ge 2}$,
	\begin{align}\nonumber
	c(m):=&\sum_{j=0}^{m}\dfrac{2^j\CC_{j,m}\Gamma\left(j+m+\frac 12\right)}{j!}-\sum_{k=1}^{m}\dfrac{\phi^{(2k-2)}(0)}{ 2(2k-2)!}\sum_{j=0}^{m-k}\dfrac{2^j\CC_{j,m-k}\Gamma\left(j+m-\frac 12\right)}{j!}\\\label{bbnewlem4eqn0}
	&+\sum_{k=1}^{m}\sum_{\ell=0}^{k}(-1)^{\ell+k}\dfrac{\phi^{(2\ell)}(0)}{(2\ell)!}\binom{\frac 12}{k-\ell}\sum_{j=0}^{m-k}\dfrac{2^j\CC_{j,m-k}\Gamma\left(j+m+\frac 12\right)}{j!}.
	\end{align}
\end{lemma}
\begin{proof}
	From \eqref{bbremark1eqn1}, we get
	\begin{align}\nonumber
		&\sum_{m=0}^{N+1}\l^{-2m}(n)\int_{-\l(n)n^{-\frac18}}^{\l(n)n^{-\frac18}} C_m(x)e^{-x^2}dx\\\nonumber
		&\hspace{.5cm}=I_{0,0}(n)+\sum_{m=1}^{N+1}\sum_{j=0}^{m}\dfrac{2^j\mathcal C_{j,m}}{j!}I_{m,j+m}(n)-\sum_{m=1}^{N+1}\sum_{k=1}^{m}\dfrac{\phi^{(2k-2)}(0)}{2(2k-2)!}\sum_{j=0}^{m-k}\dfrac{2^j\mathcal C_{j,m-k}}{j!}I_{m,j+m-1}(n)\\\label{bbnewlem4eqn1}
		&\hspace{1 cm}+\sum_{m=1}^{N+1}\sum_{k=1}^{m}\sum_{\ell=0}^{k}(-1)^{\ell+k}\dfrac{\phi^{(2\ell)}(0)}{(2\ell)!}\binom{\frac 12}{k-\ell}\sum_{j=0}^{m-k}\dfrac{2^j\CC_{j,m-k}}{j!}I_{m,j+m}(n).
	\end{align}
	Applying \eqref{bbbeqn11new} with $(r,\ell)=(0,0)$, we get
	\begin{equation}\label{bbnewlem4eqn2}
		I_{0,0}(n)=\sqrt{\pi}+O_{\le 0.7}\left(n^{-\frac{N+2}{2}}\right).
 	\end{equation}
	Using \eqref{bbnewlem3eqn2ao} we obtain, for $m$, $j\in \mathbb{N}_0$, 
	\begin{equation}\label{bbnewlem4eqn2a}
		2^j|\CC_{j,m}| \le \left(\zeta\left(\frac 32\right)-1\right)^j\!.
	\end{equation}
	Next, using \eqref{bbnewlem1eqn3a}, for $J\in \mathbb{N}$, we have 
	\begin{equation}\label{bbnewlem4eqn2b}
		\sum_{k=1}^J\dfrac{\left|\phi^{(2k)}(0)\right|}{(2k)!} \le 2.7.
	\end{equation}
	Using this and \eqref{bbnewfact1}, we get 
	\begin{equation}\label{bbnewlem4eqn2c}
		\sum_{k=0}^J\dfrac{\left|\phi^{(2k)}(0)\right|}{(2k)!}\le 3.7.
	\end{equation}
	Next, by \eqref{bbbeqn11new} with $(r,\ell)=(m,j+m)$ and \eqref{bbnewlem4eqn2a}, we obtain 
	\begin{equation*}
		\left|\sum_{m=1}^{N+1}\sum_{j=0}^{m}\dfrac{2^j\CC_{j,m}}{j!}I_{m,j+m}(n)-\sum_{m=1}^{N+1}\sum_{j=0}^{m}\dfrac{2^j\CC_{j,m}\Gamma\left(j+m+\frac 12\right)}{j!\l^{2m}(n)}\right|\le 0.1 n^{-\frac{N+2}{2}}.
	\end{equation*}
Consequently this implies that, splitting into $m=1$ and $m\ge2$, and using \eqref{bbnewlem2def} for $m=1$, 
	\begin{equation}\label{bbnewlem4eqn3}
		\sum_{m=1}^{N+1}\sum_{j=0}^{m}\dfrac{2^j\CC_{j,m}}{j!}I_{m,j+m}(n)
		=-\dfrac{3\sqrt{\pi}}{16\l^{2}(n)}+\sum_{m=2}^{N+1}\sum_{j=0}^{m}\dfrac{2^j\mathcal{C}_{j,m}\Gamma\left(j+m+\frac 12\right)}{j!\l^{2m}(n)}+O_{\le 0.1}\left(n^{-\frac{N+2}{2}}\right).
	\end{equation}
Next, we bound 
	\begin{align}\nonumber
		&\Biggl|-\!\sum_{m=1}^{N+1}\sum_{k=1}^{m}\frac{\phi^{(2k-2)}(0)}{2(2k-2)!}\sum_{j=0}^{m-k}\frac{2^j\CC_{j,m-k}}{j!}I_{m,j+m-1}(n)\!+\!\frac{\sqrt{\pi}}{2\l^2(n)}\!\\\nonumber
		&\hspace{5.5 cm}+\!\sum_{m=2}^{N+1}\sum_{k=1}^{m}\frac{\phi^{(2k-2)}(0)}{2(2k-2)!}\sum_{j=0}^{m-k}\frac{2^j\mathcal{C}_{j,m-k}\Gamma\left(j+m-\frac 12\right)}{j!\l^{2m}(n)}\Biggr|\\\label{bbnewlem4eqn4a}
		&\hspace{1 cm}\le\dfrac{1}{2}\left|\phi^{(0)}(0)\CC_{0,0}I_{1,0}(n)-\dfrac{\sqrt{\pi}}{\l^{2}(n)}\right|\\\nonumber
		&\hspace{1.5 cm}+\frac 12\sum_{m=2}^{N+1}\sum_{k=1}^{m}\dfrac{\left|\phi^{(2k-2)}(0)\right|}{(2k-2)!}\sum_{j=0}^{m-k}\dfrac{2^j|\CC_{j,m-k}|}{j!}\left|I_{m,j+m-1}(n)-\dfrac{\Gamma\left(j+m-\frac 12\right)}{\l^{2m}(n)}\right|.
	\end{align}
Employing \eqref{bbnewfact1}, \eqref{bbnewlem2def} with $m=0$, and \eqref{bbbeqn11new} with $(r,\ell)=(1,0)$, we bound the first term against $\le 0.3n^{-\frac{N+2}2}$. For the remaining sum in \eqref{bbnewlem4eqn4a}, using \eqref{bbbeqn11new}, \eqref{bbnewlem4eqn2a}, and \eqref{bbnewlem4eqn2c}, we obtain
	\begin{equation}\label{bbnewlem4eqn4d}
		\frac 12\sum_{m=2}^{N+1}\sum_{k=1}^{m}\frac{\left|\phi^{(2k-2)}(0)\right|}{(2k-2)!}\sum_{j=0}^{m-k}\frac{2^j|\CC_{j,m-k}|}{j!}\left|I_{m,j+m-1}(n)-\dfrac{\Gamma\left(j+m-\frac 12\right)}{\l^{2m}(n)}\right| \le 0.2 n^{-\frac{N+2}{2}}.\hspace{-0.5cm}
	\end{equation}

	Combining the error for the first term and \eqref{bbnewlem4eqn4d}, and applying to \eqref{bbnewlem4eqn4a}, we get 
	\begin{align}\label{bbnewlem4eqn4}
		-\sum_{m=1}^{N+1}&\sum_{k=1}^{m}\dfrac{\phi^{(2k-2)}(0)}{2(2k-2)!}\sum_{j=0}^{m-k}\dfrac{2^j\CC_{j,m-k}}{j!}I_{m,j+m-1}(n)\\\nonumber
		&=-\dfrac{\sqrt{\pi}}{2\l^{2}(n)}-\sum_{m=2}^{N+1}\sum_{k=1}^{m}\dfrac{\phi^{(2k-2)}(0)}{2(2k-2)!}\sum_{j=0}^{m-k}\dfrac{2^j\CC_{j,m-k}\Gamma\left(j+m-\frac 12\right)}{j! \l^{2m}(n)}+O_{\le 0.5}\left(n^{-\frac{N+2}{2}}\right).
	\end{align}
	Finally, we bound, using \eqref{bbbeqn11new}, \eqref{bbnewlem4eqn2a}, \eqref{bbnewfact1}, \eqref{prelim2}, \eqref{bbnewlem4eqn2b}, and $m\le1.3^{2m}$ for $m\ge 1$, 
	\begin{align}\nonumber
		&\Bigg|\sum_{m=1}^{N+1}\sum_{k=1}^{m}\sum_{\ell=0}^{k}(-1)^{\ell+k}\dfrac{\phi^{(2\ell)}(0)}{(2\ell)!}\binom{\frac 12}{k-\ell}\sum_{j=0}^{m-k}\dfrac{2^j\CC_{j,m-k}}{j!}I_{m,j+m}(n)\\\label{bbnewlem4eqn5a}
		&\hspace{0.2 cm}-\sum_{m=1}^{N+1}\sum_{k=1}^{m}\sum_{\ell=0}^{k}(-1)^{\ell+k}\dfrac{\phi^{(2\ell)}(0)}{(2\ell)!}\binom{\frac 12}{k-\ell}\sum_{j=0}^{m-k}\dfrac{2^j\CC_{j,m-k} \Gamma\left(j+m+\frac 12\right)}{j!\l^{2m}(n)}\Bigg|\le 0.5n^{-\frac{N+2}{2}}.
	\end{align}
Note that
\begin{equation}\label{bbnewfact3}
\operatorname{coeff}_{\left[x^2\right]}\left(\cot\left(\frac{\pi}{2}\left(\frac{x}{\sqrt{6}}+\frac 12\right)\right)\right)=\dfrac{\phi^{(2)}(0)}{2!}=\dfrac{\pi^2}{12}.
\end{equation}
 Splitting into $m=1$ and $m\ge2$, and using \eqref{bbnewlem2def}, \eqref{bbnewfact1}, and \eqref{bbnewfact3}, \eqref{bbnewlem4eqn5a} implies that
	\begin{multline}\label{bbnewlem4eqn5}
		\sum_{m=1}^{N+1}\sum_{k=1}^{m}\sum_{\ell=0}^{k}(-1)^{\ell+k}\dfrac{\phi^{(2\ell)}(0)}{(2\ell)!}\binom{\frac 12}{k-\ell}\sum_{j=0}^{m-k}\dfrac{2^j\CC_{j,m-k}}{j!}I_{m,j+m}(n)=\frac{\sqrt{\pi}\left(\pi^2-6\right)}{24\l^{2}(n)}\\
		+\sum_{m=2}^{N+1}\sum_{k=1}^{m}\sum_{\ell=0}^{k}(-1)^{\ell+k}\dfrac{\phi^{(2\ell)}(0)}{(2\ell)!}\binom{\frac 12}{k-\ell}\sum_{j=0}^{m-k}\dfrac{2^j\mathcal{C}_{j,m-k}\Gamma\left(j+m+\frac 12\right)}{j!\l^{2m}(n)} +O_{\le 0.5}\left(n^{-\frac{N+2}{2}}\right).
	\end{multline}
Applying \eqref{bbnewlem4eqn2}, \eqref{bbnewlem4eqn3}, \eqref{bbnewlem4eqn4}, and \eqref{bbnewlem4eqn5} to \eqref{bbnewlem4eqn1}, we conclude the lemma by \eqref{bbnewlem4eqn0}.
\end{proof}
Applying Lemma \ref{bbnewlem4} to \eqref{bbneweqn4}, for $n\ge \max\{10^4,n_N\}=n_N$ (by \eqref{cutoff}), \eqref{bbbeqn11} reduces to 
\begin{equation}\label{bbneweqn2}
\frac{e^{\frac{\pi\sqrt{24n+1}}{3\sqrt2}}}{(24n+1)\l(n)} \left(\sum_{m=0}^{N+1}\dfrac{c(m)}{\l^{2m}(n)}+O_{\le E^{[0]}_N+1.8}\left(n^{-\frac{N+2}{2}}\right)\right).
\end{equation}

We next estimate $e^{\frac{\pi\sqrt{24n+1}}{3\sqrt2}}$, $((24n+1)\l(n))^{-1}$, and $(\l^{-2m}(n))_{0\le m\le N+1}$ in the following way:
\begin{enumerate}
	\item Extracting $e^{2\pi\sqrt{\frac{n}{3}}}$ from $e^{\frac{\pi}{3\sqrt{2}}\sqrt{24n+1}}$ and truncating the Taylor series at $N+1$, we estimate the error terms in \Cref{bbblem3}.

	\item Separating the factor $n^{-\frac{5}{4}}$ from $((24n+1)\l(n))^{-1}$, after computing the Taylor expansion and truncating it at $N+1$, an upper bound for the absolute value of the error term is given in \Cref{bbblem4}.

	\item Finally for $(\lambda^{-2m}(n))_{0\leq m \leq N+1}$, we truncate the Taylor series of $(1+\frac{1}{24n})^{-\frac m2}$ at $\lfloor\frac{N+1-m}{2}\rfloor$ and estimate the error in \Cref{bbblem5}.
\end{enumerate}

\begin{lemma}\cite[Lemma 3.1]{B}\label{bprs}
	For $j, k\in \mathbb{N}_0$ with $k < 2j$, we have
	\[
	\sum_{\ell=0}^{k}(-1)^\ell\binom{k}{\ell}\binom{\frac\ell2}{j}=
	\begin{cases}
		1 &\quad \text{if}\ j=k=0,\\
		(-1)^j \frac{k\cdot2^k}{j\cdot2^{2j}} \binom{2j-k-1}{j-k} &\quad \text{otherwise}.
	\end{cases}
	\]
\end{lemma}
We now extract $e^{2\pi\sqrt{\frac n3}}$ from $e^{\frac\pi{3\sqrt2}\sqrt{24n+1}}$. For this, we define, for $m\in\N_0$,
\begin{align}\label{bbdef1}
	e_1(m) &:=
	\begin{cases}
		1 &\quad \text{if}\ m=0,\\
		\dfrac{(2m-1)!}{ (-96)^m}\displaystyle\sum_{\nu=1}^{m}\dfrac{\pa{-\pi^2}{18}^\nu}{(2\nu-1)!(\nu+m)!(m-\nu)! }
		 &\quad \text{if}\ m\in \mathbb{N},
	\end{cases}\\\label{bbdef2}
o_1(m) &:=\frac{\pi(2m)!}{24\sqrt{3}(-96)^m}\sum_{\nu=0}^{m}\frac{\pa{-\pi^2}{18}^\nu}{(2\nu)!(m-\nu)!(\nu+m+1)!}.
\end{align}
\begin{lemma}\label{bbblem3}
	For $N\ge 3$ and $n\ge n_N$, we have
	\[
		e^{\frac{\pi\sqrt{24n+1}}{3\sqrt{2}}}=e^{2\pi\sqrt{\frac n3}}\left(\sum_{m=0}^{N+1}\dfrac{b(m)}{n^\frac m2}+O_{\le \frac{2\cdot 10^{-2}}{24^{\frac N2}}}\left(n^{-\frac{N+2}{2}}\right)\right),
	\]
	where  $b(2m):=e_1(m)$ and $b(2m+1):=o_1(m)$ for $m\in \mathbb{N}_0$. 
\end{lemma}
\begin{proof}
	Taylor expanding, we have
	\begin{equation}\label{bbblem3eqn1}
		e^{\frac{\pi\sqrt{24n+1}}{3\sqrt{2}}} =e^{2\pi\sqrt{\frac{n}{3}}}\sum_{k,j\ge0}\sum_{\ell=0}^{k}(-1)^{\ell+k}\dfrac{\left(\frac{2\pi}{\sqrt{3}}\right)^k}{24^jk!}\binom k\ell \binom{\frac\ell2}j n^{\frac k2-j}=:e^{2\pi\sqrt{\frac{n}{3}}}K(n).
	\end{equation}
	Setting $z:=\frac{1}{\sqrt{n}}$ and $\alpha:=\frac{2\pi}{\sqrt{3}}$, we rewrite
\[
K\left(\frac{1}{z^2}\right)=\exp\left(\dfrac{\alpha}{z}\left(\sqrt{1+\frac{z^2}{24}}-1\right)\right).
\]
Then, $K(\frac{1}{z^2})$ is analytic in a neighborhood of $0$ and has a Taylor expansion of the form $K(\frac{1}{z^2})=\sum_{\ell\ge0}a(\ell) z^\ell$. Thus we have
\begin{equation}\label{newcheck1}
K\left(\frac{1}{z^2}\right)=K(n) = \sum_{\ell\ge0} \frac{a(\ell)}{n^\frac\ell2}.
\end{equation}
Comparing coefficients of \eqref{bbblem3eqn1} and \eqref{newcheck1} with respect to powers of $\frac{1}{\sqrt{n}}$, we obtain 
\[
\sum_{\ell=0}^{k}(-1)^\ell\binom{k}{\ell}\binom{\frac\ell2}{j}=0,
\]
for $k>2j$ with $j \in \mathbb{N}_0$. Therefore, we restrict $k$ in \eqref{bbblem3eqn1} to $k\le2j$. We let $\mathcal{B}:=\{(k,\ell,j)\in \mathbb{N}^3_0: 0 \leq \ell \leq k\}$ and for $m\in\N_0$, $\mathcal{B}_m:=\left\{(k,\ell,j)\in \mathcal{B}: k-2j=-m\right\}$. Note that $\mathcal{B}$ is the disjoint union of the $\mathcal{B}_m$. For ${\bm r}=(k,\ell,j)\in \mathcal{B}$, define 
	\[
		c({\bm r}):=\dfrac{(-1)^{\ell+k}\left(\frac{2\pi}{\sqrt{3}}\right)^k}{24^jk!}\binom k\ell\binom			{\frac\ell2}j,\ d(\bm r):=k-2j.
	\]
	Then we write
	\begin{equation}\label{bbblem3eqn2}
		K(n)=\underset{{\bm r}=(k,\ell,j)\in \mathcal{B}}{\sum} c(\bm r) n^{\frac{d(\bm r)}2} =\sum_{m\ge0}\underset{{\bm r}\in \mathcal{B}_m}{\sum}\frac{c(\bm r)}{n^{\frac m2}} =\sum_{m\ge0}\underset{{\bm r}\in \mathcal{B}_{2m}}{\sum}\frac{c(\bm r)}{n^{m}}+\sum_{m\ge0}\underset{{\bm r}\in \mathcal{B}_{2m+1}}{\sum}\frac{c(\bm r)}{n^{m+\frac12}}.
	\end{equation}
	We have $\mathcal{B}_{2m}=\{(2\nu,\ell,\nu+m)\in \mathbb{N}^3_0: 0 \leq \ell \leq 2\nu\}$. Therefore, we obtain
	\begin{equation}\label{bbblem3eqn3}
		\sum_{m\ge0}\underset{{\bm r}\in B_{2m}}{\sum}\frac{c(\bm r)}{n^{m}}=\sum_{m\ge0}(24n)^{-m}\sum_{\nu\ge0}\dfrac{\pa{\pi^2}{18}^\nu}{(2\nu)!}\sum_{\ell=0}^{2\nu}(-1)^\ell\binom{2\nu}\ell \binom{\frac\ell2}{\nu+m}.
	\end{equation}
	We next note that from \Cref{bprs} with $(k,j)= (2\nu,\nu+m)$, we have
	\begin{equation*}
	\sum_{\ell=0}^{2\nu}(-1)^\ell\binom{2\nu}{\ell}\binom{\frac{\ell}{2}}{\nu+m}=
	\begin{cases}
	1 &\quad \text{if}\ \nu=m=0,\\
	0 &\quad \text{if}\ \nu>m,\\
		\frac{(-1)^{\nu+m}\nu (2m)!}{4^m m(\nu+m)!(m-\nu)!} &\quad \text{if}\ 1\le \nu\le m.
	\end{cases}
	\end{equation*}
	Hence it follows, for $m\in \mathbb{N}$,
	\begin{equation*}
		\sum_{\nu\ge0}\dfrac{\pa{\pi^2}{18}^\nu}{(2\nu)!}\sum_{\ell=0}^{2\nu}(-1)^\ell\binom{2\nu}{\ell}\binom{\frac{\ell}{2}}{\nu+m}=\frac{(-1)^m (2m-1)!}{4^m}\sum_{\nu=1}^{m}\frac{\left(-\frac{\pi^2}{18}\right)^\nu}{(2\nu-1)!(\nu+m)!(m-\nu)!}.
	\end{equation*}
Applying this to \eqref{bbblem3eqn3}, we obtain
	\begin{equation}\label{bbblem3eqn5}
		\sum_{m\ge0} \underset{{\bm r}\in B_{2m}}{\sum}\frac{c(\bm r)}{n^m} =1+\sum_{m\ge1}\dfrac{(2m-1)!}{(-96n)^m}\sum_{\nu=1}^{m}\dfrac{\pa{-\pi^2}{18}^\nu}{(2\nu-1)!(\nu+m)!(m-\nu)! }=\sum_{m\ge0}\dfrac{e_1(m)}{n^m}.
	\end{equation} 
Note that $\mathcal{B}_{2m+1}=\left\{(2\nu+1,\nu+m+1,\ell)\in \mathbb{N}^3_0: 0 \leq \ell \leq 2\nu+1\right\}$. Consequently, we have
	\begin{equation}\label{bbblem3eqn6}
		\sum_{m\ge0}\underset{{\bm r}\in B_{2m+1}}{\sum}\frac{c(\bm r)}{n^{m+\frac12}} =-\sum_{m\ge0}\frac{(24n)^{-m}}{\sqrt n}\sum_{\nu\ge0}\frac{\left(\frac{2\pi}{\sqrt{3}}\right)^{2\nu+1}}{(2\nu+1)!24^{\nu+1}}\sum_{\ell=0}^{2\nu+1}(-1)^\ell\binom{2\nu+1}\ell \binom{\frac\ell2}{\nu+m+1}.
	\end{equation}
Now we use \Cref{bprs} with $(k,j)=(2\nu+1,\nu+m+1)$ to obtain
\begin{equation*}
\sum_{\ell=0}^{2\nu+1}(-1)^\ell\binom{2\nu+1}{\ell}\binom{\frac{\ell}{2}}{\nu+m+1}=
\begin{cases}
0 &\quad \text{if}\ \nu>m,\\
\frac{(-1)^{\nu+m+1}(2\nu+1)(2m)!}{2\cdot 4^m(\nu+m+1)!(m-\nu)!}
&\quad 0\le \nu\le m.
\end{cases}
\end{equation*}
Applying the above identity to \eqref{bbblem3eqn6}, we obtain
	\begin{equation}\label{bbblem3eqn7}
		\sum_{m\ge0}\underset{{\bm r}\in 	B_{2m+1}}{\sum}\frac{c(\bm r)}{n^{m+\frac12}}\!=\!\dfrac{\pi}{24\sqrt{3}}\sum_{m\ge0}\hspace{-0.1cm}\frac{(2m)!}{(-96)^m n^{m+\frac 12}}\sum_{\nu=0}^{m}\frac{\pa{-\pi^2}{18}^\nu}{(2\nu)!(m\!-\!\nu)!(\nu\!+\!m\!+\!1)!}\!=\!\sum_{m\ge0}\tfrac{o_1(m)}{n^{m+\frac12}}.
	\end{equation}
	Plugging \eqref{bbblem3eqn5}, \eqref{bbblem3eqn7}, and the definition of $b(m)$ into \eqref{bbblem3eqn2}, and then to \eqref{bbblem3eqn1}, we have
	\begin{align}\label{bbblem3eqn8}
		e^{\frac{\pi}{3\sqrt{2}}\sqrt{24n+1}}=e^{2\pi\sqrt{\frac{n}{3}}}\sum_{m\ge0}\dfrac{b(m)}{n^{\frac m2}}=e^{2\pi\sqrt{\frac{n}{3}}}\left(\sum_{m=0}^{N+1}\dfrac{b(m)}{n^{\frac m2}}+\sum_{m\ge N+2}\dfrac{b(m)}{n^{\frac m2}}\right).
	\end{align}
	To estimate the error term, we bound $e_1(m)$ and $o_1(m)$. From \eqref{bbdef1}, \eqref{bbdef2}, and \Cref{prelim1},
	\begin{equation}\label{bbblem3eqn9}
		|e_1(m)|\le \dfrac{\sqrt{\pi}\sinh\pa{\pi}{3\sqrt2}}{6\sqrt{2}m^{\frac 32} 24^m},\quad 	|o_1(m)| \le \dfrac{\sqrt{\pi}\cosh\pa{\pi}{3\sqrt2}}{\sqrt{3} m^{\frac 32} 24^{m+1}}.
	\end{equation}
Using the definition of $b(m)$, \eqref{bbblem3eqn9}, and \eqref{bbdef2}, for $m\in \mathbb{N}$, we get
	\begin{equation}\label{bbblem3eqn11}
	 |b(m)|\le
	 \dfrac{\sqrt{\pi}\sinh\pa{\pi}{3\sqrt2}}{3\cdot 24^{\frac m2}}.
	\end{equation}
	Therefore, using $n\ge n_N\ge n_3$ (by \eqref{cutoff}), 
	\begin{equation*}
		\left|\sum_{m\ge N+2}\dfrac{b(m)}{n^{\frac m2}}\right| \le \dfrac{\sqrt{\pi}\sinh\pa{\pi}{3\sqrt2}}{72\cdot24^{\frac N2}\left(1-\frac{1}{2\sqrt{6 n_3}}\right)}n^{-\frac{N+2}{2}} \le \frac{2\cdot 10^{-2}}{24^{\frac N2}}n^{-\frac{N+2}{2}}.
	\end{equation*}
	Combining this with \eqref{bbblem3eqn8}, we conclude the proof.
\end{proof}
Next, we bound the error term obtained after truncating the Taylor series of  $\frac{n^{\frac 54}}{(24n+1)\l(n)}$.

\begin{lemma}\label{bbblem4}
	For $N\geq 3$ and $n\ge n_N$, we have
	\[
	\frac{1}{(24n+1)\l(n)} = \frac{1}{8\cdot3^\frac34\sqrt\pi n^\frac54}\left(\sum_{m=0}^{\Flo{\frac{N+1}{2}}} \frac{e_2(m)}{n^m}+O_{\le 5.3\cdot 10^{-2}\pa{5}{96}^{\frac N2}}\left(n^{-\frac{N+2}{2}}\right)\right),
	\]
where $e_2(m):=\smallbinom{-\frac 54}{m}24^{-m}$ for $m\in \mathbb{N}_0$.
\end{lemma}

\begin{proof}
	We have the Taylor expansion
	\[
		\frac{1}{(24n+1)\l(n)} =\frac{1}{8\cdot3^\frac34\sqrt\pi n^\frac54}\sum_{m\ge 0}\frac{e_2(m)}{n^{m}}. 
	\]
	Note that
	\begin{equation}\label{fact1}
		\left|\binom{-\frac54}{m}\right| = \prod_{j=1}^m \left(1+\frac{1}{4j}\right).
	\end{equation}
	To finish the proof, using the above identity, $n\ge n_N$, and using \eqref{cutoff}, we estimate
	\begin{align*}
	\hspace{-0.5 cm}	\left|\sum_{m\ge\Flo{\frac{N+1}{2}}+1}\hspace{-0.4cm} \frac{e_2(m)}{n^m}\right| \le \pa{5}{96}^{\frac{N}{2}+1}n^{-\frac{N}{2}-1}\sum_{m\ge 0}\pa{5}{96n}^m
		&\le 5.3\cdot 10^{-2}\pa{5}{96}^{\frac N2}n^{-\frac{N+2}{2}}.\qedhere
	\end{align*}
\end{proof}

For $m \in \mathbb{N}_0$, define 
\begin{equation}\label{bbblem4aeqn}
 e_1^*(m):=\sum_{k=0}^{m}e_1(k) e_2(m-k), \quad o_1^*(m):=\sum_{k=0}^{m}o_1(k) e_2(m-k).
\end{equation}
\begin{lemma}\label{bbblem4a}
	For $N\geq 3$ and $n\ge n_N$, we have
	\[
		\dfrac{e^{\frac{\pi}{3\sqrt{2}}\sqrt{24n+1}}}{(24n+1)\l(n)}=\dfrac{e^{2\pi\sqrt{\frac{n}{3}}}}{8\cdot  3^{\frac 34}\sqrt{\pi}n^{\frac54}}\left(\sum_{m=0}^{N+1}\dfrac{b^*(m)}{n^{\frac m2}}+O_{\le 0.2 \pa{5}{96}^{\frac N2}}\left(n^{-\frac{N+2}{2}}\right)\right),
	\]
	with $b^*(2m)=e_1^*(m)\ \text{and}\ b^*(2m+1)=o_1^*(m)$.
\end{lemma}
\begin{proof}
Applying Lemma \ref{bbblem3}, Lemma \ref{bbblem4}, \eqref{cauchyproduct}, and \eqref{bbblem4aeqn},  we have
	\begin{align}\nonumber
	&\dfrac{e^{\frac{\pi}{3\sqrt{2}}\sqrt{24n+1}}}{(24n+1)\l(n)}=\dfrac{e^{2\pi\sqrt{\frac{n}{3}}}}{8\cdot 3^{\frac 34}\sqrt{\pi}n^{\frac54}}\left(\rule{0 cm}{0.9cm}\right.\sum_{m=0}^{N+1}\dfrac{b^*(m)}{n^{\frac m2}}+\delta_{2\nmid N} \dfrac{e_2\left(\frac{N+1}{2}\right)}{n^{\frac{N+1}{2}}}\sum_{m=0}^{\frac{N-1}{2}}\dfrac{o_1(m)}{n^{m+\frac12}}\\\nonumber
	&\hspace{0.5cm} +n^{-\left\flo{\frac{N+1}{2}\right}-1}\sum_{m=0}^{\left\flo{\frac{N+1}{2}\right}-1}n^{-m}\sum_{k=m}^{\left\flo{\frac{N+1}{2}\right}-1}e_1(k+1)e_2\left(m-k+\Flo{\frac{N+1}{2}}\right)\\\nonumber
	&\hspace{0.1cm}+\!n^{-\left\flo{\frac{N}{2}\right}-1}\!\sum_{m=0}^{\left\flo{\frac{N}{2}\right}-1}\!n^{-m-\frac12}\!\sum_{k=m}^{\left\flo{\frac{N}{2}\right}-1}\!o_1(k\!+\!1)e_2\left(m\!-\!k\!+\!\Flo{\frac{N}{2}}\right)\!+\!O_{\le 5.3\cdot10^{-2}\pa{5}{96}^{\frac N2}}\!\left(n^{-\frac{N+2}{2}}\right)\! \sum_{m=0}^{N+1}\!\dfrac{b(m)}{n^{\frac m2}}\\\label{bbblem4aeqn3}
	&\hspace{0.1cm}+O_{\le  2\cdot 10^{-2}24^{-\frac N2}}\left(n^{-\frac{N+2}{2}}\right) \sum_{m=0}^{\Flo{\frac{N+1}{2}}} \dfrac{e_2(m)}{n^m}+ O_{\le 1.1\cdot10^{-3}\pa{5}{2304}^{\frac N2}}\left(n^{-N-2}\right)
	\left.\rule{0 cm}{0.9cm}\right),
	\end{align}
	where $\delta_S:=1$, if a statement $S$ holds and $0$ otherwise. Next we give upper bounds for the sums involving the factor $n^{-\frac m2}$ with $m \ge N+2$. Note that for $n \geq n_N$, we have
	\begin{align}\label{bbblem4aeqn4}
	 O_{\le 1.1\cdot10^{-3}\pa{5}{2304}^{\frac N2}}\left(n^{-N-2}\right)= O_{\le 3.5\cdot 10^{-11} \pa{5}{2304n_3}^{\frac N2}}\left(n^{-\frac{N+2}{2}}\right).
	\end{align}
	Using \eqref{fact1} and the definition of $e_2(m)$ (see \Cref{bbblem4}), we estimate, for $m\in\N$,
	\begin{align}\label{bbe2estim}
	|e_2(m)|\le \pa{5}{96}^m.
	\end{align}
	For $n\ge n_N$, we have, using the definition of $e_2(m)$ (see \Cref{bbblem4}) and \eqref{bbe2estim},
	\begin{align}\label{bbblem4aeqn5}
	\frac{2\cdot 10^{-2}}{24^{\frac N2}}	n^{-\frac{N+2}{2}}\left|\sum_{m=0}^{\Flo{\frac{N+1}{2}}} \dfrac{e_2(m)}{n^m}\right| \le \frac{2.1\cdot 10^{-2}}{24^{\frac N2}}n^{-\frac{N+2}{2}}.
	\end{align}
Next, we have, using the fact $b_0=1$ (see \Cref{bbblem3} and \eqref{bbdef1}) and  \eqref{bbblem3eqn11}
	\begin{align}\label{bbblem4aeqn6}
	 5.3\cdot10^{-2}\pa{5}{96}^{\frac N2}\left|\sum_{m=0}^{N+1} \dfrac{b(m)}{n^\frac m2}\right|n^{-\frac{N+2}{2}} &\le 5.4\cdot10^{-2}\pa{5}{96}^{\frac N2}n^{-\frac{N+2}{2}}.
	\end{align}
Next note that, using \eqref{bbblem3eqn9} and \eqref{bbe2estim}, we have
\begin{align}
\left|n^{-\Flo{\frac{N}{2}}-1}\!\!\sum_{m=0}^{\left\flo{\frac{N}{2}\right}-1}\!\!n^{-m-\frac12}\!\!\sum_{k=m}^{\left\flo{\frac{N}{2}\right}-1}\!\!o_1(k\!+\!1)e_2\!\left(m\!-\!k\!+\!\Flo{\frac{N}{2}}\right)\right| &\!\le\! 5.3\!\cdot\! 10^{-2}\!\pa{5}{96}^{\!\!\frac{N}{2}}\!\!n^{-\frac{N+2}{2}}\!,\hspace{-0.15cm}\label{bbblem4aeqn7}\\
\left|n^{-\left\flo{\frac{N+1}{2}\right}-1}\!\!\sum_{m=0}^{\left\flo{\frac{N+1}{2}\right}-1}\!\!n^{-m}\!\!\sum_{k=m}^{\left\flo{\frac{N+1}{2}\right}-1}\!\!e_1(k\!+\!1)e_2\!\left(m\!-\!k\!+\!\Flo{\frac{N\!+\!1}{2}}\right)\right| &\!\le\! 3.7\!\cdot\! 10^{-2}\!\pa{5}{96}^{\!\!\frac N2}\!\!n^{-\frac{N+2}{2}}\!.\hspace{-0.15cm}\label{bbblem4aeqn8}
\end{align}
Next we obtain, for $n\ge n_N$, using \eqref{bbe2estim}, \eqref{bbdef2}, and \eqref{bbblem3eqn9}, 
	\begin{align}\label{bbblem4aeqn11}
	\left|\delta_{2\nmid N} \dfrac{e_2\left(\frac{N+1}{2}\right)}{n^{\frac{N+1}{2}}}\sum_{m=0}^{\frac{N-1}{2}}\dfrac{o_1(m)}{n^{m+\frac12}}\right|
	&\le 1.8\cdot 10^{-2}\pa{5}{96}^{\frac N2}n^{-\frac{N+2}{2}}.
	\end{align}
	Therefore combining \eqref{bbblem4aeqn4}, \eqref{bbblem4aeqn5}, \eqref{bbblem4aeqn6}, \eqref{bbblem4aeqn7}, \eqref{bbblem4aeqn8}, and \eqref{bbblem4aeqn11}, and then applying it to \eqref{bbblem4aeqn3}, the claim follows.
\end{proof}

For $1\leq m\leq N+1$, set
\begin{equation}\label{bbnewdef2}
	R_{N}(m):=\left\lfloor\frac{N+1-m}{2}\right\rfloor\ \ \text{and}\ \ E^{[1]}_N:=\underset{1\leq m\leq N+1}{\max}\left\{\dfrac{\left|
	\begin{pmatrix}
		-\frac{m}{2}\\R_{N}(m)+1
	\end{pmatrix}
	\right|}{24^{R_{N}(m)+1}}\right\}.
\end{equation}

Applying \eqref{bbbeqn3}, we get the Taylor expansion of $\l^{-2m}(n)$ which is an alternating series. Truncating the series at $R_N(m)$, we obtain.

\begin{lemma}\label{bbblem5}
	For $N\geq 3$, $1 \leq m \leq N+1$, and $n>\frac{N+5}{96}$, we have
	\[
		\l^{-2m}(n) = \frac{3^\frac m2}{\pi^m}\left(\sum_{\ell=0}^{R_N(m)} \binom{-\frac m2}{\ell}24^{-\ell}n^{-\ell-\frac m2}+O_{\le E^{[1]}_N}\left(n^{-\frac{N+2}{2}}\right)\right).
	\]
\end{lemma}

Define 
\begin{equation}\label{defen2}
E^{[2]}_N:=6.1 (2N+2)! E^{[1]}_N.
\end{equation}
\begin{lemma}\label{bbblem6}
	For $N\ge 3$ and $n> \frac{N+5}{96}$, we have 
	\begin{equation*}
		\sum_{m=0}^{N+1}\frac{c(m)}{\l^{2m}(n)}=\sum_{m=0}^{N+1}\frac{c^*(m)}{n^{\frac{m}{2}}}+O_{\le E^{[2]}_N}\left(n^{-\frac{N+2}{2}}\right),
	\end{equation*}
	where
	$c^*(0):=c(0)$, and for $m \in \mathbb{N}$,
	\begin{equation*}
		c^*(2m)\!:=\!\sum_{\ell=1}^{m}\!c(2\ell)\!\pa{\sqrt{3}}{\pi}^{\!2\ell}\! \binom{-\ell}{m-\ell}\!24^{\ell-m}, \, c^*(2m+1)\!:=\!\sum_{\ell=0}^{m}\! c(2\ell+1)\!\pa{\sqrt{3}}{\pi}^{\!2\ell+1}\! \binom{-\left(\ell+\frac12\right)}{m-\ell}\!24^{\ell-m}.
	\end{equation*}
\end{lemma}

\begin{proof}
For $n>\frac{N+5}{96}$, by \Cref{bbblem5}, \eqref{bbnewdef2}, and the definition of $c^*(m)$,
 we obtain
\begin{align}\label{bbblem6eqn3}
\sum_{m=0}^{N+1}\frac{c(m)}{\l^{2m}(n)}=\sum_{m=0}^{N+1}\frac{c^*(m)}{n^{\frac{m}{2}}}+\sum_{m=1}^{N+1}c(m)\pa{\sqrt{3}}{\pi}^m O_{\le E^{[1]}_N}\left(n^{-\frac{N+2}{2}}\right).
\end{align}
For $r\in \mathbb{N}$, using \eqref{newfact0} and Lemma \ref{prelim1}, we have 
\begin{equation}\label{bbblem6eqn4}
\Gamma\left(r+\frac 12\right) \le \frac{r!}{\sqrt{r}}.
\end{equation}
For $L\in \mathbb{N}_0$, applying \eqref{bbnewlem4eqn2a}, it follows that
\begin{equation}\label{bbblem6eqn5}
\sum_{j=0}^{L}\frac{2^j\left|\CC_{j,L}\right|}{j!} \le e^{\zeta\left(\frac 32\right)-1}.
\end{equation}
Next, we find an upper bound for $|c(m)|$. Using \eqref{bbnewlem4eqn0}, we have for $m \in \mathbb{N}_{\ge 2}$,
\begin{align}\nonumber
|c(m)|&\le \sum_{j=0}^{m}\dfrac{2^j\left|\CC_{j,m}\right|\Gamma\left(j+m+\frac 12\right)}{j!}+\sum_{k=1}^{m}\dfrac{\left|\phi^{(2k-2)}(0)\right|}{ 2(2k-2)!}\sum_{j=0}^{m-k}\dfrac{2^j\left|\CC_{j,m-k}\right|\Gamma\left(j+m-\frac 12\right)}{j!}\\\nonumber
&\hspace{2 cm}+\sum_{k=1}^{m}\sum_{\ell=0}^{k}\dfrac{\left|\phi^{(2\ell)}(0)\right|}{(2\ell)!}\left|\binom{\frac 12}{k-\ell}\right|\sum_{j=0}^{m-k}\dfrac{2^j\left|\CC_{j,m-k}\right|\Gamma\left(j+m+\frac 12\right)}{j!}\\\label{bbblem6eqn6}
&=:c^{[1]}(m)+c^{[2]}(m)+c^{[3]}(m).
\end{align}
First, for $m \ge 2$, we bound, using \eqref{bbblem6eqn4} and \eqref{bbblem6eqn5},
\begin{equation}\label{bbblem6eqn7}
c^{[1]}(m) \le \frac{e^{\zeta\left(\frac 32\right)-1}}{\sqrt{2}} (2m)!.
\end{equation}
Next, for $m\ge 2$, we estimate, using \eqref{bbblem6eqn4}, \eqref{bbblem6eqn5}, \eqref{bbnewfact1}, and \eqref{bbnewlem1eqn3a},
\begin{align}\label{bbblem6eqn8}
c^{[2]}(m) &\le \left(\frac{1}{2}+\frac{4}{\pi}+\frac{\pi}{45}\right) e^{\zeta\left(\frac 32\right)-1}(2m-2)!.
\end{align}
Finally, for $m\ge 2$, we have, using \eqref{bbblem6eqn4}, \eqref{bbblem6eqn5}, \eqref{bbnewfact1}, \eqref{bbnewfact3}, \eqref{prelim2}, and \eqref{bbnewlem1eqn3a}, 
\begin{equation}\label{bbblem6eqn9}
	c^{[3]}(m) \le \left(\dfrac{\pi^2+6\zeta\left(\frac 32\right)}{48}+\frac{7}{3\pi}+\frac{\left(3\sqrt{2}+1\right)\pi}{540\sqrt{2}}\right) \frac{e^{\zeta\left(\frac 32\right)-1}}{\sqrt{2}}(2m)!.
\end{equation}
Combining \eqref{bbblem6eqn7}, \eqref{bbblem6eqn8}, \eqref{bbblem6eqn9} and applying it to \eqref{bbblem6eqn6}, we get, for $m\ge 2$, 
\begin{equation}\label{bbblem6eqn10}
|c(m)| \le 9 (2m)!,
\end{equation}
using $\frac{1}{2m(2m-1)}\le \frac{1}{12}$. Consequently, we have, employing \Cref{bbnewlem4}, \eqref{bbblem6eqn10}, and \eqref{defen2}, 
\begin{equation*}
\left|\sum_{m=1}^{N+1}c(m)\pa{\sqrt{3}}{\pi}^m\right| E^{[1]}_N  n^{-\frac{N+2}{2}}\le 6.1 (2N+2)! E^{[1]}_N  n^{-\frac{N+2}{2}}= E^{[2]}_N  n^{-\frac{N+2}{2}}.
\end{equation*}
Combining this with \eqref{bbblem6eqn3}, 
we conclude the proof.
\end{proof}

Applying Lemma \ref{bbblem6} to \eqref{bbneweqn2} and following \eqref{newdeferror1} and \eqref{defen2}, for $n\ge \max\{\frac{N+5}{96}, n_N\}=n_N$, \eqref{bbbeqn11} becomes 
\begin{equation}\label{bbneweqn3}
\frac{e^{\frac{\pi\sqrt{24n+1}}{3\sqrt2}}}{(24n+1)\l(n)} \left(\sum_{m=0}^{N+1}\dfrac{c^*(m)}{n^{\frac m2}}+O_{\le E^{[0]}_N+E^{[2]}_N+ 1.8}\left(n^{-\frac{N+2}{2}}\right)\right).
\end{equation}
Finally applying Lemma \ref{bbblem4a} to \eqref{bbneweqn3}, we aim to obtain an expansion of \eqref{bbbeqn11} of the form
\begin{equation*}
\dfrac{e^{2\pi\sqrt{\frac{n}{3}}}}{8\cdot  3^{\frac 34}\sqrt{\pi}n^{\frac54}}\left(\sum_{m=0}^{N+1}\dfrac{A(m)}{n^{\frac m2}}+O_{\le E^{[3]}_N}\left(n^{-\frac{N+2}{2}}\right)\right),
\end{equation*}
where for $0\le m\le N+1$, we define 
\begin{equation}\label{bbblem7eqn1}
\!A(m)\!:=\!\sum_{k=0}^{m}\!b^*(k)c^*(m-k),\ E^{[3]}_N\!:=\!1.3(2N+2)!\!+\!1.2 \left(E^{[0]}_N+E_N^{[2]}+1.8\right).
\end{equation}

\begin{lemma}\label{bbblem7}
For $N\ge 3$ and $n\ge n_N$, we have 
\begin{align*}
&\frac{e^{\frac{\pi\sqrt{24n+1}}{3\sqrt2}}}{(24n+1)\l(n)} \left(\sum_{m=0}^{N+1}\dfrac{c^*(m)}{n^{\frac m2}}+O_{\le E^{[0]}_N+E^{[2]}_N+ 1.8}\left(n^{-\frac{N+2}{2}}\right)\right)\\
&\hspace{6 cm}=\dfrac{e^{2\pi\sqrt{\frac{n}{3}}}}{8\cdot  3^{\frac 34}\sqrt{\pi}n^{\frac54}}\left(\sum_{m=0}^{N+1}\dfrac{A(m)}{n^{\frac m2}}+O_{\le E^{[3]}_N}\left(n^{-\frac{N+2}{2}}\right)\right).
\end{align*}
\end{lemma}

\begin{proof}
Using Lemma \ref{bbblem4a}, \eqref{cauchyproduct}, and \eqref{bbblem7eqn1}, we see that 
\begin{align}\nonumber
&\frac{e^{\frac{\pi\sqrt{24n+1}}{3\sqrt2}}}{(24n+1)\l(n)} \left(\sum_{m=0}^{N+1}\dfrac{c^*(m)}{n^{\frac m2}}+O_{\le E^{[0]}_N+E^{[2]}_N+ 1.8}\left(n^{-\frac{N+2}{2}}\right)\right)\\\nonumber
&=\!\dfrac{e^{2\pi\sqrt{\frac{n}{3}}}}{8\!\cdot\!3^{\frac 34}\sqrt{\pi}n^{\frac54}}\!\Bigg(\sum_{m=0}^{N+1}\frac{A(m)}{n^{\frac m2}}\!+\!n^{-\frac{N+2}{2}}\!\sum_{m=0}^{N}\sum_{k=m}^{N}\frac{b^*(k+1)c^*(m-k+N+1)}{n^{\frac m2}}\!\\\nonumber
&\hspace{3.3 cm}+\!\sum_{m=0}^{N+1}\dfrac{b^*(m)}{n^{\frac m2}}O_{\le E^{[0]}_N+E^{[2]}_N+ 1.8}\!\left(n^{-\frac{N+2}{2}}\right)+\sum_{m=0}^{N+1}\dfrac{c^*(m)}{n^{\frac m2}}O_{\le 0.2 \pa{5}{96}^{\frac N2}}\left(n^{-\frac{N+2}{2}}\right)\\\nonumber
&\hspace{8.3 cm}+O_{\le 0.2 \pa{5}{96}^{\frac N2}\left(E^{[0]}_N+E^{[2]}_N+ 1.8\right)}\left(n^{-N-2}\right)  \Bigg)\\\label{bbblem7eqn2}
&=:\dfrac{e^{2\pi\sqrt{\frac{n}{3}}}}{8\cdot  3^{\frac 34}\sqrt{\pi}n^{\frac54}}\left(\sum_{m=0}^{N+1}\frac{A(m)}{n^{\frac m2}}+S^{[1]}(n)+S^{[2]}(n)+S^{[3]}(n)+S^{[4]}(n)\right).
\end{align}
Using that $n_N$  is increasing, we have, for $n\ge n_N\ge n_3$, 
\begin{equation}\label{bbblem7eqn3}
\left|S^{[4]}(n)\right| \le 6.3\cdot 10^{-9}\pa{5}{96 n_3}^{\frac N2}\left(E^{[0]}_N+E^{[2]}_N+ 1.8\right)n^{-\frac{N+2}{2}}.
\end{equation}
Before we estimate $S^{[1]}(n), S^{[2]}(n)$, and $S^{[3]}(n)$, we bound $|b^*(m)|$ and $|c^*(m)|$. For $m\in \mathbb{N}$, using \Cref{bbblem4a},  \eqref{bbblem4aeqn}, \Cref{bbblem4}, \eqref{bbdef1}, \eqref{bbdef2}, \eqref{bbblem3eqn9}, and \eqref{fact1}, we have
\begin{align*}
|b^*(2m)|&\le\!\left(\!1\!+\!\frac{\!\sqrt{\pi}\!\sinh\!\pa{\pi}{3\sqrt{2}}\!\zeta\!\left(\frac 32\right)\!}{6\!\sqrt{2}}\right)\!\left(\frac{5}{96}\right)^m\!, |b^*(2m+1)|\!\le\! \left(\!\frac{\pi+\!\sqrt{\pi}\!\cosh\!\pa{\pi}{3\sqrt{2}}\!\zeta\!\left(\frac 32\right)\!}{24\!\sqrt{3}}\right)\!\pa{5}{96}^m\!.
\end{align*}

Applying this, using \Cref{bbblem4a}, \eqref{bbblem4aeqn}, \Cref{bbblem4}, and \eqref{bbdef2}, for $m\in \mathbb{N}$, we get 
\begin{align}\label{bbblem7eqn7}
|b^*(m)|&\le \left(1+\frac{\sqrt{\pi}\sinh\pa{\pi}{3\sqrt{2}}\zeta\left(\frac 32\right)}{6\sqrt{2}}\right)\pa{5}{96}^{\frac m2}.
\end{align}

Similarly, we find an upper bound $|c^*(m)|$ by estimating $|c^*(2m)|$ and $|c^*(2m+1)|$. For $m\in \mathbb{N}$, using the definition of $c^*(m)$, 
\eqref{bbblem6eqn10}, $\smallbinom{-\ell}{m-\ell}=(-1)^{m+\ell}\smallbinom{m-1}{\ell-1}$, and $(\frac{1}{24}+\frac{3}{\pi^2})^{m-1}\le 1$,
\begin{equation}\label{bbblem7eqn8}
|c^*(2m)| \le \frac{27}{\pi^2} (4m)! \le 2.8 (4m)!.
\end{equation}
We note that, for $m\in \mathbb{N}$ and $0\le \ell\le m$, using \Cref{prelim1}, we have
\begin{equation}\label{bbblem7eqn12}
\left|\binom{-\left(\ell+\frac12\right)}{m-\ell}\right|\le
\begin{cases}
\frac{1}{\sqrt{\pi}} &\quad \text{if}\ \ell=0,\\
\frac{8}{7}\binom{m}{\ell} &\quad \text{if}\ \ell\ge 1.
\end{cases}
\end{equation}
Next, for $m\in \mathbb{N}$, we have, using \Cref{bbblem6}, 
\Cref{bbnewlem4}, \eqref{bbblem7eqn12}, and \eqref{bbblem6eqn10},
\begin{equation}\label{bbblem7eqn13}
|c^*(2m+1)| \le 2 (4m+2)!.
\end{equation}
Combining \eqref{bbblem7eqn8} and \eqref{bbblem7eqn13}, and using \Cref{bbblem6} and 
\Cref{bbnewlem4}, for $m\in \mathbb{N}$, we get
\begin{equation}\label{bbblem7eqn14}
|c^*(m)|\le
2.8 (2m)!.
\end{equation}


Next, using 
\eqref{bbblem7eqn2}, \eqref{bbblem7eqn7}, and \eqref{bbblem7eqn14}, for $n\ge n_N$, we have
\begin{align}\label{bbblem7eqn15}
\left|S^{[1]}(n)\right| &\le 1.2 (2N+2)! n^{-\frac{N+2}{2}}.
\end{align}
Next, by \Cref{bbblem4a}, \eqref{bbblem4aeqn}, \Cref{bbblem4}, \eqref{bbdef1}, \eqref{bbblem7eqn7}, and \eqref{bbblem7eqn2}, for $n\ge n_N$,
\begin{equation}\label{bbblem7eqn16}
\left|S^{[2]}(n)\right|\le 1.1 \left(E^{[0]}_N+E^{[2]}_N+ 1.8\right) n^{-\frac{N+2}{2}}. 
\end{equation}
Finally, employing \Cref{bbblem6}, \Cref{bbnewlem4}, \eqref{bbblem7eqn14}, and \eqref{bbblem7eqn2}, for $n\ge n_N$, we have
\begin{equation}\label{bbblem7eqn17}
\left|S^{[3]}(n)\right| \le 1.1\cdot 10^{-4}(2N+2)!\pa{5}{96}^{\frac N2} n^{-\frac{N+2}{2}}.
\end{equation}
Using \eqref{bbblem7eqn2}, \eqref{bbblem7eqn3}, \eqref{bbblem7eqn15}, \eqref{bbblem7eqn16}, \eqref{bbblem7eqn17}, and \eqref{bbblem7eqn1}, the lemma follows. 
\end{proof}
By \eqref{bbneweqn3} and Lemma \ref{bbblem7}, \eqref{bbbeqn11} and hence \eqref{bbb1} becomes, for $n\ge n_N$,
\begin{equation}\label{bbeqn1}
\dfrac{e^{2\pi\sqrt{\frac{n}{3}}}}{8\cdot  3^{\frac 34}\sqrt{\pi}n^{\frac54}}\left(\sum_{m=0}^{N+1}\dfrac{A(m)}{n^{\frac m2}}+O_{\le E^{[3]}_N}\left(n^{-\frac{N+2}{2}}\right)\right).
\end{equation}
In order to get an asymptotic expansion of the first term in \eqref{E:k1int} of the following form
\[
\dfrac{e^{2\pi\sqrt{\frac{n}{3}}}}{8\cdot  3^{\frac 34}\sqrt{\pi}n^{\frac54}}\left(\sum_{m=0}^{N+1}\dfrac{A(m)}{n^{\frac m2}}+O_{\le E^{[3]}_N+C}\left(n^{-\frac{N+2}{2}}\right)\right)\ \ \text{for some}\ \ C>0. 
\]
it remains to estimate the upper bound in \eqref{E:error2} as
\[
\frac{14}{24n+1}e^{2\pi\sqrt{\frac n3}-\frac{\pi n^{\frac 14}}{\sqrt{3}}}=\dfrac{e^{2\pi\sqrt{\frac{n}{3}}}}{8\cdot  3^{\frac 34}\sqrt{\pi}n^{\frac54}}O_{\le D}\left(n^{-\frac{N+2}{2}}\right)\ \ \text{for some}\ \ D>0.
\]

For $n\ge n_N$, \eqref{bbnewlem3aeqn5} and \eqref{cutoff}, we further estimate the right-hand side of \eqref{E:error2} as
\begin{equation}\label{bbeqn2}
\frac{14}{24n+1}e^{2\pi\sqrt{\frac n3}-\frac{\pi n^{\frac 14}}{\sqrt{3}}} \le \dfrac{e^{2\pi\sqrt{\frac{n}{3}}}}{8\cdot  3^{\frac 34}\sqrt{\pi}n^{\frac54}} 18.9 n^{-\frac{N+2}{2}}.
\end{equation}
Combining \eqref{E:error2}, \eqref{bbeqn1} and \eqref{bbeqn2}, for $n\ge n_N$  the first term in \eqref{E:k1int} becomes
\begin{equation}\label{bbeqn3}
\dfrac{e^{2\pi\sqrt{\frac{n}{3}}}}{8\cdot  3^{\frac 34}\sqrt{\pi}n^{\frac54}}\left(\sum_{m=0}^{N+1}\dfrac{A(m)}{n^{\frac m2}}+O_{\le E^{[3]}_N+18.9}\left(n^{-\frac{N+2}{2}}\right)\right).
\end{equation}
Using $2\pi\sqrt{\frac n3}>n^{\frac 14}$ for $n\in \mathbb{N}_{\ge 2}$, \eqref{bbnewlem3aeqn5}, for $n\ge n_N$, we further refine the upper bound (see \eqref{E:error1}) of the second term (in \eqref{E:k1int}) as  
\begin{equation}\label{bbeqn4}
\frac{28}{24n+1} \le \dfrac{e^{2\pi\sqrt{\frac{n}{3}}}}{8\cdot  3^{\frac 34}\sqrt{\pi}n^{\frac54}} 37.8 n^{-\frac{N+2}{2}}.
\end{equation}
Combining \eqref{bbeqn3} and \eqref{bbeqn4}, for $n\ge n_N$,  the main term \eqref{E:k1int} takes the form
\begin{equation}\label{bbeqn6}
\dfrac{e^{2\pi\sqrt{\frac{n}{3}}}}{8\cdot  3^{\frac 34}\sqrt{\pi}n^{\frac54}}\left(\sum_{m=0}^{N+1}\dfrac{A(m)}{n^{\frac m2}}+O_{\le E^{[3]}_N+56.7}\left(n^{-\frac{N+2}{2}}\right)\right).
\end{equation}

\section{The terms $k\ge2$}\label{S:kge2}

We recall Lemma 3.1 from \cite{BB}.
\begin{lemma} 
	For $|x|\leq 1$, $k \geq 2$, and $r\in \mathbb{N}$, we have 
	\begin{equation*}
	\left|\sum_{r=0}^{2k-1} K_k(n,r)\cot\left(\frac{\pi}{2k}\left(\frac{x}{\sqrt6}-r-\frac12\right)\right)\right| \le \frac{4k^2}{\pi}(\log(k)+14).
	\end{equation*}
\end{lemma}

Next we bound the sum from $k\ge2$ in \Cref{T:ExactFormula}.

\begin{lemma}\label{L:kgeq2bound}
	For $n\ge n_N$,  we have
	\begin{multline*}
	\left|\frac{\pi}{2^\frac34\sqrt3(24n+1)^\frac34}\sum_{k\ge2} \sum_{r=0}^{2k-1} \frac{K_k(n,r)}{k^2}\int_{-1}^1 \left(1-x^2\right)^\frac34 \cot\left(\frac{\pi}{2k}\left(\frac{x}{\sqrt6}-r-\frac12\right)\right)\right.\\
	\left.\times I_\frac32\left(\frac{\pi}{3\sqrt2k}\sqrt{\left(1-x^2\right)(24n+1)}\right)dx {\vphantom{\sum_{k\ge2}}}\right| = O_{\le{0.4}}\left(e^{\pi\sqrt\frac n3}\right).
	\end{multline*}
\end{lemma}

\begin{proof}
	By \cite[equation (3.1)]{BB}, we can bound the left-hand side by
	\begin{equation*}
	\hspace{-0.4cm}\frac{2^\frac94}{\sqrt3(24n+1)^\frac34}\sum_{k\ge2} (\log(k)+14) \int_0^1 \left(1-x^2\right)^\frac34 I_\frac32\hspace{-0.15cm}\left(\frac{\pi}{3\sqrt2k}\sqrt{\left(1-x^2\right)(24n+1)}\right) dx.\hspace{-0.4cm}
	\end{equation*}
	As is \cite[Lemma 3.2]{BB}, we split the integral into two pieces. By \cite[equation (3.2)]{BB}, the contribution from $1-\frac{18k^2}{\pi^2(24n+1)}<x^2<1$ may be bounded against $10.3$. From \cite[page 6]{BB}, we know that the contribution from the remaining range $0\le x\le 1-\frac{18k^2}{\pi^2(24n+1)}$ can be estimated against 
	\begin{equation*}
	\frac{8}{\pi(24n+1)}\left(\frac{\pi\sqrt{24n+1}}{3\sqrt2}\right)^\frac32 \left(\log\left(\frac{\pi\sqrt{24n+1}}{3\sqrt2}\right)+14\right) e^\frac{\pi\sqrt{24n+1}}{6\sqrt2} = O_{\le 0.3}\left(e^{\pi\sqrt\frac n3}\right),
	\end{equation*}
	for $n\ge n_N$. Combining these two bounds, we conclude the proof of the lemma.
\end{proof}

\section{Proof of \Cref{bbbthm}, \Cref{bbbthmshift}, and \Cref{secondshiftedthm}}\label{S:shift}
We are now ready to prove \Cref{bbbthm}.
Define
\begin{equation}\label{bberrordef}
C_N:=E^{[3]}_N+69.7.
\end{equation}

\begin{proof}[Proof of Theorem \ref{bbbthm}]
	Applying \eqref{bbeqn6} and Lemma \ref{L:kgeq2bound} to Theorem \ref{T:ExactFormula}, we get, for $n\ge n_N$,
	\begin{equation}\label{bbeqn12}
	u(n)\!=\!\dfrac{e^{2\pi\sqrt{\frac{n}{3}}}}{8\cdot  3^{\frac 34}\sqrt{\pi}n^{\frac54}}\!\left(\sum_{m=0}^{N+1}\!\dfrac{A(m)}{n^{\frac m2}}\!+\!O_{\le E^{[3]}_N+56.7}\!\left(n^{-\frac{N+2}{2}}\right)\!+\!O_{\le{13}}\!\left(n^{\frac54} e^{-\pi\sqrt\frac n3}\right)\right).
	\end{equation}
	For $n\ge n_N$,
	using $\pi\sqrt{\frac n3}>n^{\frac 14}$ for $n\in \mathbb{N}$, the monotonicity of the exponential function, and \eqref{bbnewlem3aeqn5}, we bound $n^{\frac54} e^{-\pi\sqrt\frac n3} \le n^{-\frac{N+2}{2}}$. Applying this to \eqref{bbeqn12}, it follows that, for $n\ge n_N$,
	\begin{equation*}
	u(n)=\dfrac{e^{2\pi\sqrt{\frac{n}{3}}}}{8\cdot  3^{\frac 34}\sqrt{\pi}n^{\frac54}}\left(\sum_{m=0}^{N+1}\dfrac{A(m)}{n^{\frac m2}}+O_{\le C_N}\left(n^{-\frac{N+2}{2}}\right)\right).\qedhere
	\end{equation*}
\end{proof}

For $s\in\N$ fixed, we apply Theorem \ref{bbbthm} with $n\mapsto n+s$ and $n\ge n_N-s$ to obtain
\begin{equation}\label{bbshifteqn2}
u(n+s)=\dfrac{e^{2\pi\sqrt{\frac{n+s}{3}}}}{8\cdot 3^{\frac 34}\sqrt{\pi}(n+s)^{\frac{5}{4}}} \left(\sum_{m=0}^{N+1}\dfrac{A(m)}{(n+s)^{\frac m2}}+O_{\leq C_N}\left(n^{-\frac{N+2}{2}}\right)\right).
\end{equation}

To prove Theorem \ref{bbbthmshift} starting from \eqref{bbshifteqn2}, we follow the proofs of \Cref{bbblem3}--\Cref{bbblem7}. 
 
Define, for $m\in\N_0$ 
\begin{align}\label{bbshiftdef1}
e^{[1]}_s(m) &:=
\begin{cases}
1 &\quad \text{if}\ m=0,\\
\frac{s^m (2m-1)!}{(-4)^m}\!\displaystyle\sum_{\nu=1}^{m}\!\frac{\left(-\frac{4\pi^2s}{3}\right)^\nu}{(2\nu-1)!(\nu+m)!(m-\nu)!} &\quad \text{if}\ m\in \mathbb{N},
\end{cases}\\\label{bbshiftdef2}
o^{[1]}_s(m)&:=\frac{\pi s^{m+1} (2m)!}{\sqrt{3} (-4)^m}\sum_{\nu=0}^{m}\frac{\left(-\frac{4\pi^2s}{3}\right)^{\nu}}{(2\nu)!(m-\nu)!(\nu+m+1)!}. 
\end{align}
Similarly as for \Cref{bbblem3}, one can show.
\begin{lemma}\label{bbshiftlem1} For $N\ge 3$ and $n\ge 4s^2$, we have
\[
e^{2\pi\sqrt{\frac{n+s}{3}}}=e^{2\pi\sqrt{\frac n3}}\left(\sum_{m=0}^{N+1}\frac{d^{[1]}_s(m)}{n^{\frac m2}}+O_{\le C^{[1]}_N(s)}\left(n^{-\frac{N+2}{2}}\right)\right),
\]
where for $m\in\N$, $d^{[1]}_s(2m):=e^{[1]}_s(m)$, $d^{[1]}_s(2m+1):=o^{[1]}_s(m)$, $C^{[1]}_N(s)\!:=\!1.5 s^{\frac{N+3}{2}}\!\cosh(2\pi\sqrt{\frac s3})$. 
\end{lemma}
Now, we proceed to bound the error term obtained after truncating the Taylor series of  $(n+s)^{-\frac 54}$ at $\flo{\frac{N+1}{2}}$. As in \Cref{bbblem4}, we obtain:
\begin{lemma}\label{bbshiftlem2}
	For $N\geq 3$ and $n\ge 4s^2$, we have
	\[
	(n+s)^{-\frac 54} = n^{-\frac54}\left(\sum_{m=0}^{\Flo{\frac{N+1}{2}}} \frac{e^{[2]}_s(m)}{n^m}+O_{\le C_N^{[2]}(s)}\left(n^{-\frac{N+2}{2}}\right)\right),
	\]
	where $e^{[2]}_s(m):=\binom{-\frac 54}{m}s^{m}$ for $m\in \mathbb{N}_0, s\in \mathbb{N}$, and $C_N^{[2]}(s):=\frac{20s}{11} (\frac{5s}{4})^{\frac N2}$.
\end{lemma}

For $m \in \mathbb{N}_0$, define
\begin{align}\label{label1}
\!e^{[3]}_s(m)\!&:=\!\sum_{k=0}^{m}\!e^{[1]}_s(k)e^{[2]}_s(m-k),\quad \!o^{[3]}_s(m)\!:=\!\sum_{k=0}^{m}\!o^{[1]}_s(k)e^{[2]}_s(m-k),\\\label{bbshiftnewdef}
n^{[1]}(s)&:=\begin{cases}
2s^4 &\text{if}\ s\ge 2,\\
4 &\text{if}\ s=1.
\end{cases}
\end{align}
Combining \Cref{bbshiftlem1}, \Cref{bbshiftlem2}, \eqref{cauchyproduct}, \eqref{label1}, \eqref{fact1}, \eqref{bbshiftdef1}, and \eqref{bbshiftdef2}, we obtain. 
\begin{lemma}\label{bbshiftlem3}
	For $N\geq 3$ and $n \ge n^{[1]}(s)$, we have
	\[
	\dfrac{e^{2\pi\sqrt{\frac{n+s}{3}}}}{8\cdot 3^{\frac 34}\sqrt{\pi}(n+s)^{\frac{5}{4}}}	=\dfrac{e^{2\pi\sqrt{\frac{n}{3}}}}{8\cdot  3^{\frac 34}\sqrt{\pi}n^{\frac54}}\left(\sum_{m=0}^{N+1}\dfrac{d^{[2]}_s(m)}{n^{\frac m2}}+O_{\le C_N^{[3]}(s)}\left(n^{-\frac{N+2}{2}}\right)\right),
	\]
	with 
	\begin{align*}
	d^{[2]}_s(2m)&:=e^{[3]}_s(m),\quad  d^{[2]}_s(2m+1):=o^{[3]}_s(m),\\
	C_N^{[3]}(s)&:=2.7C_N^{[1]}(s)+\left(1+1.2\sqrt{s}\right)C_N^{[2]}(s)+\left((20.5+12s)C_N^{[2]}(s)+0.7\right)\cosh\left(2\pi\sqrt{\frac s3}\right). 
	\end{align*}
\end{lemma}
Truncating the Taylor expansion (an alternating series) of $(n+s)^{-\frac m2}$ at $R_N(m)$, we obtain. 
\begin{lemma}\label{bbshiftlem4}
Let $R_N(m)$ be as in \eqref{bbnewdef2}. For $N\geq 3$, $1 \leq m \leq N+1$, and $n>\frac{s(N+5)}4$,
	\[
	(n+s)^{-\frac m2} =\left(\sum_{\ell=0}^{R_N(m)} \binom{-\frac m2}{\ell}s^{\ell}n^{-\ell-\frac m2}+O_{\le C^{[4]}_N(s)}\left(n^{-\frac{N+2}{2}}\right)\right),
	\]
where \[C^{[4]}_N(s):=\underset{1\leq m\leq N+1}{\max}\left\{\left|
\begin{pmatrix}
-\frac{m}{2}\\R_{N}(m)+1
\end{pmatrix}
\right|s^{R_{N}(m)+1}\right\}.\]
\end{lemma}

Next, using Lemma \ref{bbshiftlem4}, \eqref{bbblem7eqn1}, \Cref{bbblem4a}, \eqref{bbblem4aeqn}, \Cref{bbblem4}, \eqref{bbdef1}, \Cref{bbblem6}, \Cref{bbnewlem4}, \eqref{bbblem7eqn14}, and \eqref{bbblem7eqn7}, 
we transform $\sum_{m=0}^{N+1}A(m)(n+s)^{-\frac m2}$
into a sum of the form $\sum_{m=0}^{N+1}A^*_m(s)n^{-\frac m2}$.
\begin{lemma}\label{bbshiftlem5}
	For $N\ge 3$, and $n>\frac{s(N+5)}{4}$, we have
	\begin{equation*}
	\sum_{m=0}^{N+1}\dfrac{A(m)}{(n+s)^{\frac m2}}+O_{\leq C_N}\left(n^{-\frac{N+2}{2}}\right)=\sum_{m=0}^{N+1}\frac{A^*_s(m)}{n^{\frac{m}{2}}}+O_{\le C_N^{[5]}(s)}\left(n^{-\frac{N+2}{2}}\right),
	\end{equation*}
	where
	$A^*_s(0):=A(0)$, for $m \in \mathbb{N}$,
	\begin{equation*}
	\!A^*_s(2m)\!:=\!\sum_{\ell=1}^{m}\!A(2\ell)\!\binom{-\ell}{m-\ell}\!s^{m-\ell},\ A^*_s(2m+1)\!:=\!\sum_{\ell=0}^{m}\! A(2\ell+1)\! \binom{-\left(\ell+\frac 12\right)}{m-\ell}\!s^{m-\ell},
	\end{equation*}
	and $C_N^{[5]}(s):=C_N+3.3(2N+2)! C_N^{[4]}(s)$.
\end{lemma}

Applying Lemma \ref{bbshiftlem5} to \eqref{bbshifteqn2}, we get, for $n\ge \max\{n_N-s, \frac{s(N+5)}{4}\}$
\begin{equation}\label{bbshifteqn3}
u(n+s)=\dfrac{e^{2\pi\sqrt{\frac{n+s}{3}}}}{8\cdot 3^{\frac 34}\sqrt{\pi}(n+s)^{\frac{5}{4}}} \left(\sum_{m=0}^{N+1}\dfrac{A^*_s(m)}{n^{\frac m2}}+O_{\leq C^{[5]}_N(s)}\left(n^{-\frac{N+2}{2}}\right)\right).
\end{equation}
For  $s\in \mathbb{N}$ and $0\le m\le N+1$, define
\begin{align}\label{label2}
A_s(m)&:=\sum_{k=0}^{m}d^{[2]}_s(k)A^*_s(m-k),\\\nonumber
C_N(s)&:=\frac{C^{[3]}_N(s) C^{[5]}_N(s)}{(2s)^{N+2}}+7.4(2N+2)!C^{[3]}_N(s)+4 s\cosh\left(2\pi\sqrt{\frac s3}\right)C^{[5]}_N(s)\\\label{bbshiftlem6eqn0a}
&\hspace{3.5 cm}+461.7(2N+2)! s^{\frac 32}\cosh\left(2\pi\sqrt{\frac s3}\right) \pa{5(s+1)}{4}^{\frac N2},\\\label{bbshifteqn4}
n^{[2]}_N(s)&:=\max\left\{n^{[1]}(s),\frac{s(N+5)}{4}\right\}.
\end{align}

Using \eqref{bbshifteqn4}, \eqref{bbshifteqn3}, Lemma \ref{bbshiftlem3}, \eqref{cauchyproduct}, \eqref{label2}, \eqref{label1}, \eqref{bbshiftdef1}, \Cref{bbshiftlem2}, \eqref{bbshiftdef2}, \Cref{bbshiftlem5}, \eqref{bbblem7eqn1}, \Cref{bbblem6}, \Cref{bbblem4a}, \eqref{bbblem4aeqn}, \Cref{bbblem4}, \eqref{bbdef1}, \eqref{bbdef2}, \Cref{bbnewlem4}, \eqref{bbshiftnewdef}, \eqref{bbshiftdef1}, and \eqref{bbshiftlem6eqn0a}, we obtain the following result. 
\begin{lemma}\label{bbshiftlem6}
	For $N\ge 3$ and $n\ge n^{[2]}_N(s)$, we have
	\begin{align*}
	\dfrac{e^{2\pi\sqrt{\frac{n+s}{3}}}}{8\cdot 3^{\frac 34}\sqrt{\pi}(n+s)^{\frac{5}{4}}} \left(\sum_{m=0}^{N+1}\dfrac{A^*_s(m)}{n^{\frac m2}}+O_{\leq C^{[5]}_N(s)}\left(n^{-\frac{N+2}{2}}\right)\right)&\\
	&\hspace{-3 cm}=\dfrac{e^{2\pi\sqrt{\frac{n}{3}}}}{8\cdot 3^{\frac 34}\sqrt{\pi}n^{\frac{5}{4}}} \left(\sum_{m=0}^{N+1}\dfrac{A_s(m)}{n^{\frac m2}}+O_{\leq C_N(s)}\left(n^{-\frac{N+2}{2}}\right)\right).
	\end{align*}
\end{lemma}
We are now ready to prove \Cref{bbbthmshift}.
\begin{proof}[Proof of \Cref{bbbthmshift}]
For $s\in \mathbb{N}$,	define 
	\begin{equation}\label{bbshiftcutoff} 
	n_N(s):=\max\left\{n_N-s,n^{[2]}_N(s)\right\}.
	\end{equation}
 Applying Lemma \ref{bbshiftlem6} to \eqref{bbshifteqn3}, we obtain, for $n\ge n_N(s)$, 
	\[
	u(n+s)=\dfrac{e^{2\pi\sqrt{\frac{n}{3}}}}{8\cdot 3^{\frac 34}\sqrt{\pi}n^{\frac{5}{4}}} \left(\sum_{m=0}^{N+1}\dfrac{A_s(m)}{n^{\frac m2}}+O_{\leq C_N(s)}\left(n^{-\frac{N+2}{2}}\right)\right).\qedhere
	\]
\end{proof}

For $0\le m\le N+1$ and $s\in \mathbb{N}_0$, define 
\begin{align}\label{label3}
A_0(m)&:=A(m),\ C_N(0):=C_N\\\label{bounddef}
P^{\pm}_N(n,s)&:=
\displaystyle\sum_{m=0}^{N+1}\frac{A_s(m)}{n^{\frac m2}}\pm\frac{C_N(s)}{n^{\frac{N+2}{2}}}
\\\label{refinedfinalcutoff1}
\nu_N(s)&:=\begin{cases}
n_N\ &\quad \text{if}\ s=0,\\
\max\left\{n_N,n^{[2]}_N(s)\right\} &\quad \text{if}\ s\in \mathbb{N}.
\end{cases}
\end{align} 

Next, we combine \Cref{bbbthm} and \Cref{bbbthmshift} to get the following inequality for $u(n+s)$ with $s\in \mathbb{N}_0$. For $s\in \mathbb{N}_0$ and for $n\ge \nu_N(s)$, we have 
\begin{equation}\label{unimodalineq}
\dfrac{e^{2\pi\sqrt{\frac{n}{3}}}}{8\cdot 3^{\frac 34}\sqrt{\pi}n^{\frac{5}{4}}} P^{-}_N(n,s)\le u(n+s)\le \dfrac{e^{2\pi\sqrt{\frac{n}{3}}}}{8\cdot 3^{\frac 34}\sqrt{\pi}n^{\frac{5}{4}}} P^{+}_N(n,s).
\end{equation}

For fixed $j\in \mathbb{N}$, define 
\begin{equation}\label{secondshiftedcutoff} 
n_\Delta(j):=2j+\max\left\{n_3,32j^4, \left\lceil\frac{9}{\pi^5j^4}\left(58j^2+C_3(2j)+C_3(j)+C_3\right)^2\right\rceil \right\}.
\end{equation}

We are now ready to prove \Cref{secondshiftedthm}.

\begin{proof}[Proof of \Cref{secondshiftedthm}]
Let $j\in \mathbb{N}$ be fixed. Making the shift $n\mapsto n+2j$ and following \eqref{2ndshift}, note that to prove the theorem, it is equivalent to show that for $n\ge n_\Delta(j)-2j$
\[
\Delta^{[2]}_j(u)(n+2j)=u(n+2j)-2u(n+j)+u(n)>0.
\]

Applying \eqref{unimodalineq} with $N=3, s\in \{0, j, 2j\}$, for $j\in \mathbb{N}$ and $n\ge \max\{n_3, 32j^4\}$,
we have
\begin{align}\label{secondshiftedeqn4}
\Delta^{[2]}_j(u)(n+2j)
&\ge \dfrac{e^{2\pi\sqrt{\frac{n}{3}}}}{8\cdot 3^{\frac 34}\sqrt{\pi}n^{\frac{5}{4}}}\left(\sum_{j=2}^{4}\frac{A^{[1]}_j(m)}{n^{\frac m2}}-\frac{C_3(2j)+C_3(j)+C_3}{n^{\frac 52}}\right),
\end{align}
where, computing in a computer, we have 
\begin{align}\label{coeffcompshift}
&A^{[1]}_j(2):=\frac{\pi^\frac52 j^2}{3} , \quad A^{[1]}_j(3):=\frac{\pi^\frac32 j^2 \left(8 \pi ^2 (6 j+1)-567\right)}{144 \sqrt{3}},\nonumber\\
& A^{[1]}_j(4):=\frac{\sqrt{\pi } j^2 \left(-14256 \pi ^2 (6 j+1)+16 \pi ^4 (24 j (7 j+2)+13)+280665\right)}{41472}.
\end{align}
For $j\in \mathbb{N}$, it is verified by a computer, that
\begin{equation}\label{label4}
A^{[1]}_j(3)\ge -j^2\ \ \text{and}\ \ A^{[1]}_j(4)\ge -57j^2.
\end{equation} 
Thus, applying \eqref{label4}, \eqref{coeffcompshift}, and \eqref{secondshiftedcutoff} to \eqref{secondshiftedeqn4}, for $n\ge n_\Delta(j)-2j$, we have
\[
 \Delta^{[2]}_j(u)(n+2j)>0.
\]
This concludes the proof of \Cref{secondshiftedthm}.\qedhere 
\end{proof}

\section{Proof of Theorem \ref{HigherorderTuranthm}} \label{sec:HigherTuran}
In this section, we prove Theorem \ref{HigherorderTuranthm}. We split the proof into two parts. In Subsection \ref{sec:HigherTuranSuffLarge}, we use Theorem \ref{bbbthmshift} to show Theorem \ref{HigherorderTuranthm} for $n$ sufficiently large, and then in Subsection \ref{sec:computations} we develop a computer algorithm to check the theorem for $n$ small.

\subsection{Proof of Theorem \ref{HigherorderTuranthm} for $n$ sufficiently large}\label{sec:HigherTuranSuffLarge}\hspace{0 cm}

Making the shift $n\mapsto n+1$ in \eqref{Turandef}, we have to show for $n\ge 9.4\cdot 10^9-1$, 
\begin{align}\nonumber
&4\left(u^2(n+1)-u(n)u(n+2)\right)\left(u^2(n+2)-u(n+1)u(n+3)\right)\\\label{unimodalTuran}
&\hspace{8 cm}\ge \left(u(n+1)u(n+2)-u(n)u(n+3)\right)^2.
\end{align}
Fixing $N=12$ and using the definition of $\nu_N(s)$ in \eqref{refinedfinalcutoff1}, we obtain
\begin{equation}\label{nu12}
\underset{0\le s\le 3}\max\{\nu_{12}(s)\}\le 9.4\cdot 10^9-1=:\nu_{12}.
\end{equation}
Following \eqref{label3} and \eqref{bounddef}, we write $P^{\pm}_{12}(n,s)$ explicitly for $0\le s\le 3$. Taking $N=12$ in \eqref{bounddef}, for $0\le s\le 3$,
\begin{align}\label{poly4}
P_{\pm}(n,s):=P^{\pm}_{12}(n,s)&=\sum_{m=0}^{13}\frac{A_s(m)}{n^{\frac m2}}\pm\frac{C_{12}(s)}{n^{7}},
\end{align}
where for $0\le m\le 9$, $A_0(m)=A(m)$ and $A_s(m)$ with $s\in\{1,2,3\}$ 
 are given\footnote{We provide details on $A(m)$ only for $0\le m\le 9$ because $A(m)$ for $10\le m\le12$ is not required for this paper. One could compute them explicitly using computer algebra system.} in \eqref{coeff0seqn} and \eqref{coeff3seqn}. For $N=12$, using \eqref{bberrordef}, \eqref{defen2}, \eqref{bbnewdef2}, \eqref{newdeferror1}, \eqref{bbshiftlem6eqn0a}, \Cref{bbshiftlem5}, \Cref{bbshiftlem4}, \eqref{bbnewdef2}, \Cref{bbshiftlem3}, \Cref{bbshiftlem2}, \Cref{bbshiftlem1}, we obtain (numerically by computer) 
\begin{equation}\label{upbounderrorconst}
C_{12}(0)<1.4\cdot 10^{27},
\quad C_{12}(1)<8.8\cdot 10^{32},\quad C_{12}(2)<1.4\cdot 10^{35},\quad C_{12}(3)<4.8\cdot 10^{36}.
\end{equation}
Applying \eqref{upbounderrorconst} to \eqref{poly4}, we have 
\begin{align}\nonumber
P_{+}(n,0)<\sum_{m=0}^{13}\tfrac{A(m)}{n^{\frac m2}}+\tfrac{1.4\cdot 10^{27}}{n^{7}}=:\mathcal{P}_{+}(n,0)&,\ P_{-}(n,0)>\sum_{m=0}^{13}\tfrac{A(m)}{n^{\frac m2}}-\tfrac{1.4\cdot 10^{27}}{n^{7}}=:\mathcal{P}_{-}(n,0),\\\label{refnpoly1}
P_{+}(n,1)\!<\!\!\sum_{m=0}^{13}\!\tfrac{A_1(m)}{n^{\frac m2}}\!+\!\tfrac{8.8\cdot 10^{32}}{n^{7}}\!=:\!\mathcal{P}_{+}(n,1)&,\ P_{-}(n,1)\!>\!\!\sum_{m=0}^{13}\!\tfrac{A_1(m)}{n^{\frac m2}}\!-\!\tfrac{8.8\cdot 10^{32}}{n^{7}}\!=:\!\mathcal{P}_{-}(n,1),\\\nonumber
P_{+}(n,2)\!<\!\!\sum_{m=0}^{13}\!\tfrac{A_2(m)}{n^{\frac m2}}\!+\!\tfrac{1.4\cdot 10^{35}}{n^{7}}\!=:\!\mathcal{P}_{+}(n,2)&,\ P_{-}(n,2)\!>\!\!\sum_{m=0}^{13}\!\tfrac{A_2(m)}{n^{\frac m2}}\!-\!\tfrac{1.4\cdot 10^{35}}{n^{7}}\!=:\!\mathcal{P}_{-}(n,2),\\\nonumber 
P_{+}(n,3)\!<\!\sum_{m=0}^{13}\!\tfrac{A_3(m)}{n^{\frac m2}}\!+\!\tfrac{4.8\cdot 10^{36}}{n^{7}}\!=:\!\mathcal{P}_{+}(n,3)&,\ P_{-}(n,3)\!>\!\sum_{m=0}^{13}\!\tfrac{A_3(m)}{n^{\frac m2}}\!-\!\tfrac{4.8\cdot 10^{36}}{n^{7}}\!=:\!\mathcal{P}_{-}(n,3).
\end{align}
Thus, using \eqref{unimodalineq}, \eqref{nu12}, \eqref{poly4}, and \eqref{refnpoly1}, to prove \eqref{unimodalTuran}, it suffices to show that for $n\ge \nu_{12}$,
\begin{align}\nonumber
&4\left(\mathcal{P}^2_{-}(n,1)-\mathcal{P}_{+}(n,0) \mathcal{P}_{+}(n,2)\right)\left(\mathcal{P}^{2}_{-}(n,2)-\mathcal{P}_{+}(n,1) \mathcal{P}_{+}(n,3)\right)\\\label{refn1}
&\hspace{6 cm}\ge \left( \mathcal{P}_{+}(n,1)  \mathcal{P}_{+}(n,2)- \mathcal{P}_{-}(n,0)  \mathcal{P}_{-}(n,3)\right)^2.
\end{align}
By computer, we see that for $n\ge  8492967488$,
\begin{equation}\label{lowerrefneqn2}
4\left(\mathcal{P}^2_{-}(n,1)-\mathcal{P}_{+}(n,0) \mathcal{P}_{+}(n,2)\right)\left(\mathcal{P}^{2}_{-}(n,2)-\mathcal{P}_{+}(n,1) \mathcal{P}_{+}(n,3)\right)>4Q_1(n)>0,
\end{equation}
where
\begin{align*}
Q_1(n)&:=\frac{\pi ^4}{12}n^{-3}+\left(\frac{5 \pi ^5}{9 \sqrt{3}}-\frac{35 \pi ^3}{16 \sqrt{3}}\right)n^{-\frac 72}+\left(\frac{1085 \pi ^2}{128}-\frac{215 \pi ^4}{36}+\frac{1609 \pi ^6}{2592}\right)n^{-4}\\
&\hspace{2 cm}+\left(\frac{175 \pi ^3}{2\sqrt{3}}-\frac{19215 \sqrt{3} \pi }{1024}-\frac{83111 \pi ^5}{3456 \sqrt{3}}+\frac{65161 \pi ^7}{46656 \sqrt{3}}\right)n^{-\frac 92}-\frac{2060}{n^5}.
\end{align*}

Next, we find an upper bound for the factor $( \mathcal{P}_{+}(n,1)  \mathcal{P}_{+}(n,2)- \mathcal{P}_{-}(n,0)  \mathcal{P}_{-}(n,3))^2$ appearing on the right-hand side of \eqref{refn1}. Verified with computer, we have for $n\ge 7886464400$, 
\begin{equation}\label{upperrefneqn2}
0<
\left( \mathcal{P}_{+}(n,1)  \mathcal{P}_{+}(n,2)- \mathcal{P}_{-}(n,0)  \mathcal{P}_{-}(n,3)\right)^2
<Q_2(n),
\end{equation}
where
\begin{align*}
Q_2(n)&:=\frac{\pi ^4}{3}n^{-3}+\left(\frac{20 \pi ^5}{9 \sqrt{3}}-\frac{35 \pi ^3}{4 \sqrt{3}}\right)n^{-\frac 72}+\left(\frac{1085 \pi ^2}{32}-\frac{215 \pi ^4}{9}+\frac{1609 \pi ^6}{648}\right)n^{-4}\\
&\hspace{3 cm}+\left(\frac{350 \pi ^3}{\sqrt{3}}-\frac{19215 \sqrt{3} \pi }{256}-\frac{83255 \pi ^5}{864 \sqrt{3}}+\frac{65161 \pi ^7}{11664 \sqrt{3}}\right)n^{-\frac 92}.
\end{align*}
Therefore, to prove \eqref{refn1}, by \eqref{lowerrefneqn2} and \eqref{upperrefneqn2}, it remains to show that
\[
4 Q_1(n)>Q_2(n),
\]
which holds for $n\ge 78304$. Therefore, for $n\ge  8492967488$, we see that \eqref{refn1} holds. 

Using \eqref{nu12} for 
\[
	n\ge \max\left\{\nu_{12}, 8492967488\right\}=\max\left\{9.4\cdot 10^9-1,8492967488\right\}=9.4\cdot 10^{9}-1,
\]
\eqref{unimodalTuran} holds and this concludes the proof of Theorem \ref{HigherorderTuranthm} for $n\geq 9.4\cdot 10^9$.

\subsection{Completing the proof of \Cref{HigherorderTuranthm}} \label{sec:computations}
To complete the proof of Theorem \ref{HigherorderTuranthm}, we develop a computer algorithm to show that $u(n)$ satisfies the Tur\'an inequality \eqref{Turandef} for the $33\leq n\leq 9.4\cdot 10^9$. Namely, we want to show that for every $33\leq n\leq 9.4\cdot 10^9$ we have 
\begin{multline}\label{eqn:unimodalinequality}
	f(n):= 4\left(u^2(n)-u(n-1)u(n+1)\right)\left(u^2(n+1)-u(n)u(n+2)\right)\\
	- \left(u(n)u(n+1)-u(n-1)u(n+2)\right)^2\geq 0.
\end{multline}
To verify \eqref{eqn:unimodalinequality} numerically, we first find bounds for $u(n)$. For this define the Kloosterman sum
\[
A_k^{[2]}(n,m):=\sum_{\substack{0\le h<k\\\gcd(h,k)=1}} e^{2\pi i s(h,k)+\frac{2\pi i}{k}\left(nh+mh'\right)},
\]
where $h'$ is the inverse of $h$ modulo $k$ and the {\it Dedekind sum}
\[
s(h,k):=\sum_{j=1}^{k-1}\left(\frac{j}{k}-\frac{1}{2}\right)\left(\frac{hj}{k}-\left\lfloor \frac{hj}{k}\right\rfloor-\frac{1}{2}\right).
\]

For $M\in\N_0$, we then let
\begin{align*}
f_{M}(n)&:=\frac{\pi^5}{108}+ \delta_{M\leq \frac{\pi\sqrt{12n-1}}3-1}\frac{2\sqrt{6\left(M+1\right)}}{(12n-1)^{\frac{5}{4}}}\left(\frac{\pi\sqrt{12n-1}}3-M\right) e^{\frac{\pi\sqrt{12n-1}}{3\left(M+1\right)}},\\
p^{[2]}_{\pm}\left(M; n\right)&:=\frac{2\pi}{12n-1}\sum_{k=1}^{M} \frac{A_k^{[2]}(n,0)}{k}I_{2}\left(\frac{\pi\sqrt{12n-1}}{3k}\right)\pm f_{M}(n).
\end{align*}
Letting $\varepsilon_m:=+$ if $m$ is even and $-$ if $m$ is odd, for $L, M\in\N_0$ we finally define
\begin{equation}\label{eqn:uM1M2+-def}
\hspace{-0.025 cm}u_{-}(L, M; n)\!:=\!\sum_{m=0}^{2L+1}\!(-1)^{m}\! p_{2,L}^{\varepsilon_{m+1}}\!\left(n-T_m\right)\!,\  u_{+}(L, M; n)\!:=\!\sum_{m=0}^{2L}\!(-1)^{m}\! p_{2,L}^{\varepsilon_{m}}\!\left(n-T_m\right),
\end{equation}
where $T_m:=\frac{m(m+1)}{2}$ denotes the {\it $m$-th triangular number}. The following proposition leads to a parallel computer algorithm used to check \eqref{eqn:unimodalinequality}.
\begin{proposition}\label{prop:computation}
\Cref{eqn:unimodalinequality} holds if there exist $L, M\in\N_0$ for which
\begin{align*}
 3u^2_{-}(L,M; n)u^2_{-}(L, M; n+1)\!&+\!6u_{-}(L,M; n-1)u_{-}(L, M; n)u_{-}(L,M; n+1)u_{-}(L,M; n+2)\\
 &- 4u_{+}(L,M; n-1)u^3_{+}(L, M; n+1)\!-\!4u^3_{+}(L,M; n)u_{+}(L,M; n+2)\\
& \!-\! u^2_{+}(L,M; n-1)u^2_{+}(L,M; n+2)\geq 0.
\end{align*}
\end{proposition}
Before proving Proposition \ref{prop:computation}, using \Cref{lem:Bessel}, we obtain the following bound. 
\begin{lemma}\label{lem:fM2bound}
We have
\[
\frac{2\pi}{12n-1}\sum_{k\geq M+1} \frac{\left|A_k^{[2]}(n,0)\right|}{k}I_{2}\left(\frac{\pi\sqrt{12n-1}}{3k}\right)\leq f_{M}(n).
\]
\end{lemma}

We are now ready to prove Proposition \ref{prop:computation}.
\begin{proof}[Proof of Proposition \ref{prop:computation}]

We expand \eqref{unimodalgen} to obtain
\begin{equation}\label{eqn:unalternate}
u(n)=\sum_{m\geq0} (-1)^m p_2\left(n-T_m\right),
\end{equation}
where $p_2(n)$ is  the number of $2$-colored partitions of $n$ and we define $p_2(n):=0$ for $n<0$. Since the sum in \eqref{eqn:unalternate} is alternating and $m\mapsto p_2(n-T_m)$ is decreasing, for any $L\in\N_0$ we have
\begin{equation}\label{eqn:uM1+-}
u_{-}\left(L; n\right):=\sum_{m=0}^{2L+1} (-1)^m p_2\left(n-T_m\right)\leq u(n)\leq \sum_{m=0}^{2L} (-1)^m p_2\left(n-T_m\right)=:u_{+}\left(L; n\right).
\end{equation}
Noting that $u(n)\geq 0$ and plugging \eqref{eqn:uM1+-} into \eqref{eqn:unimodalinequality}, for any $L\in\N_0$ we have that
\begin{align}\label{eqn:unimodalinequalityM1+-}
f(n)&\geq 3u^2_{-}\left(L; n\right)u^2_{-}\left(L; n+1\right) \!+\!6u_{-}\left(L; n-1\right)u_{-}\left(L; n\right)u_{-}\left(L; n+1\right)u_{-}\left(L; n+2\right)\nonumber \\
 &-\! 4u_{+}\!\left(L; n-1\right)u^3_{+}\!\left( L; n+1\right)\!-\!4u^3_{+}\!\left(L; n\right)u_{+}\!\left(L; n+2\right)\!-\!u^2_{+}\!\left(L; n-1\right)u^2_{+}\!\left(L; n+2\right).
\end{align}
We conclude that if the right-hand side of \eqref{eqn:unimodalinequalityM1+-} is nonnegative, then \eqref{eqn:unimodalinequality} holds.
To bound the right-hand side of \eqref{eqn:unimodalinequalityM1+-} from below, we use the exact formula \cite[Theorem 1.1]{IskanderJainTalvola}
\begin{equation}\label{eqn:p2exact}
p_2(n)=\frac{2\pi}{12n-1}\sum_{k\geq 1} \frac{A_k^{[2]}(n,0)}{k}I_{2}\left(\frac{\pi\sqrt{12n-1}}{3k}\right).
\end{equation}

Using Lemma \ref{lem:fM2bound} to estimate the contribution from $k$ large in \eqref{eqn:p2exact}, we hence have 
\[
p^{[2]}_{-}\left(M; n\right)\leq  p_2(n)\leq p^{[2]}_{+}\left(M; n\right).
\]
Plugging this into \eqref{eqn:uM1M2+-def} then yields the inequalities 
\[
u_{-}(L,M; n)\leq u_{-}(L; n)\qquad \text{ and }\qquad
u_{+}(L,M; n) \geq u_{+}(L; n).
\]
This implies that if the claimed inequality in the proposition holds for some $L, M\in\N_0$, then \eqref{eqn:unimodalinequalityM1+-} is non-negative, and consequently \eqref{eqn:unimodalinequality} holds.
\end{proof}

For $n\leq 10^3$, we directly compute $u(n)$ via the generating function \eqref{unimodalgen}. We then choose $M, L$ and verify the inequality in Proposition \ref{prop:computation} for $10^3<n\leq 9.4\cdot 10^9$ by computer. The necessary calculations took about 4.5 years of CPU time, split in parallel across multiple cores of multiple systems. In particular, about $65\%$ of the computer calculations were carried out on the Mach 2 supercomputer at the Research Institute for Symbolic Comptuation (RISC), $25\%$ of the calculations were performed using research computing facilities offered by Information Technology Services, The University of Hong Kong (HKU ITS), and $10\%$ were performed on personal computers. In Table \ref{tab:M1M2}, we give choices of $M$ and $L$ for ranges $n_1\leq n\leq n_2$ which we use to verify the inequality in Proposition \ref{prop:computation}.
\footnotesize\setlength{\tabcolsep}{2pt}
\begin{longtable}{|c|c|c|c|c|c|c|c|c|c|c|c|c|}
\caption{Choices of $M$ and $L$ for $n_1\leq n\leq n_2$.\label{tab:M1M2}}\\
\hline\rule{0pt}{2.6ex}
$n_1$&$10^3$&$10^5$&$5\cdot 10^5$&$10^6$&$4\cdot 10^6$&$10^7$&$2\cdot 10^7$& $8\cdot 10^7$&$10^8$&$1.1\cdot 10^8$&$7.9\cdot 10^8$&$10^9$\\
$n_2$&$10^5$&$5\cdot 10^5$&$10^6$&$4\cdot 10^6$&$10^7$&$2\cdot 10^7$&$8\cdot 10^7$&$10^8$&$1.1\cdot 10^8$&$7.9\cdot 10^8$&$10^9$&$9.4\cdot 10^9$\\
\hline\rule{0pt}{2.3ex}
$M$&$75$&$110$&$150$&$200$&$400$&$500$&$600$&$700$&$750$&$1000$&$1300$&$2000$\\
$L$&$30$&$30$&$30$&$35$&$35$&$35$&$35$&$35$&$35$&$40$&$40$&$40$\\
\hline
\end{longtable}

\appendix
\section{Constants from the paper}\label{secappend}

The coefficients $A(m)$ ($0\le m\le 9$) are given by 
\begin{align}\nonumber
&A(0)=\sqrt{\pi},\quad A(1)=\frac{\pi ^{\frac32}}{6 \sqrt{3}}-\frac{15 \sqrt3}{16\sqrt\pi},\quad A(2)=-\frac{35 \sqrt{\pi }}{96}+\frac{13 \pi ^{\frac52}}{864}+\frac{315}{512 \pi ^{\frac32}},\\\nonumber
&A(3)=\frac{945 \sqrt{3}}{8192 \pi ^{\frac52}}+\frac{315 \sqrt3}{1024\sqrt\pi}-\frac{91 \pi ^{\frac32}}{512 \sqrt{3}}+\frac{7 \pi ^{\frac72}}{972 \sqrt{3}},\\\nonumber
&A(4)=\frac{5005 \sqrt{\pi }}{16384}+\frac{7441 \pi ^{\frac92}}{4478976}-\frac{77 \pi ^{\frac52}}{1728}-\frac{3465}{16384 \pi ^{\frac32}}+\frac{93555}{524288 \pi ^{\frac72}},\\\nonumber
&A(5)=-\frac{45045 \sqrt{3}}{1048576 \pi ^{\frac52}}+\frac{1216215 \sqrt{3}}{8388608 \pi ^{\frac92}}-\frac{65065 \sqrt3}{262144\sqrt\pi}+\frac{7007 \pi ^{\frac32}}{18432 \sqrt{3}}-\frac{1064063 \pi ^{\frac72}}{23887872 \sqrt{3}}+\frac{200851 \pi ^{\frac{11}2}}{134369280 \sqrt{3}},\\\nonumber
&A(6)\!=\!-\!\frac{175175 \sqrt{\pi }}{294912}\!+\!\frac{5320315 \pi ^{\frac52}}{28311552}\!+\!\frac{31963933 \pi^{\frac{13}2}}{58047528960}\!-\!\frac{2611063 \pi ^{\frac92}}{143327232}\!+\!\frac{2927925}{16777216 \pi ^{\frac32}}\!-\!\frac{1216215}{16777216 \pi ^{\frac72}}\!+\!\frac{127702575}{268435456 \pi ^{\frac{11}2}},\\\nonumber
&A(7)=\frac{9954945 \sqrt{3}}{268435456 \pi ^{\frac52}}-\frac{34459425 \sqrt{3}}{536870912 \pi ^{\frac92}}+\frac{2791213425 \sqrt{3}}{4294967296 \pi ^{\frac{13}2}}+\frac{ 2977975\sqrt{3}}{6291456\sqrt{\pi}}-\frac{633117485 \pi ^{\frac32}}{452984832 \sqrt{3}}+\frac{488268781 \pi ^{\frac72}}{1528823808 \sqrt{3}}\\\nonumber
&\hspace{8.5 cm}-\frac{543386861 \pi ^{\frac{11}2}}{20639121408 \sqrt{3}}+\frac{879016541 \pi ^{\frac{15}2}}{1218998108160 \sqrt{3}},\\\nonumber
&A(8)=\frac{60146161075 \sqrt{\pi }}{28991029248}+\frac{134216554667 \pi ^{\frac92}}{660451885056}+\frac{510826543681 \pi ^{\frac{17}2}}{1404285820600320}-\frac{9277106839 \pi ^{\frac52}}{8153726976}-\frac{283922342743 \pi ^{\frac{13}2}}{19503969730560}\\\nonumber
&\hspace{4cm}-\frac{11316305}{33554432 \pi ^{\frac32}}+\frac{567431865}{8589934592 \pi ^{\frac72}}-\frac{1964187225}{8589934592 \pi ^{\frac{11}2}}+\frac{1750090817475}{549755813888 \pi ^{\frac{15}2}},\\\nonumber
&A(9)=-\frac{79214135 \sqrt{3}}{1073741824 \pi ^{\frac52}}+\frac{8511477975 \sqrt{3}}{137438953472 \pi ^{\frac92}}-\frac{371231385525 \sqrt{3}}{1099511627776 \pi ^{\frac{13}2}}+\frac{53086088130075 \sqrt{3}}{8796093022208 \pi ^{\frac{17}2}}-\frac{ 84204625505\sqrt{3}}{51539607552\sqrt{\pi}}\\\nonumber 
&\hspace{3.3 cm}\!+\!\frac{454578235111 \pi ^{\frac32}}{57982058496 \sqrt{3}}\!+\!\frac{283922342743 \pi ^{\frac{11}2}}{660451885056 \sqrt{3}}\!-\!\frac{10334674709359 \pi ^{\frac72}}{3522410053632 \sqrt{3}}\!-\!\frac{9705704329939 \pi ^{\frac{15}2}}{356644017930240 \sqrt{3}}
\\\label{coeff0seqn}
&\hspace{10.5 cm}+\!\frac{6765647211043 \pi ^{\frac{19}2}}{10833062044631040 \sqrt{3}}.
\end{align}

We note that dividing each $A(m)\ (0\le m\le 4)$ by the factor $8\cdot 3^{\frac 34}\sqrt{\pi}$, we get
\begin{align*}
\frac{A(0)}{8\cdot 3^{\frac 34}\sqrt{\pi}}&=\frac{1}{8\cdot 3^{\frac34}}=A,\quad \frac{A(1)}{8\cdot 3^{\frac 34}\sqrt{\pi}}=\frac{\pi }{144\cdot 3^{\frac 34}}-\frac{5\cdot 3^{\frac34}}{128 \pi }=B,\\
\frac{A(2)}{8\cdot 3^{\frac 34}\sqrt{\pi}}&=\frac{13\cdot \pi ^2}{6912\cdot 3^{\frac 34}}+\frac{105\cdot 3^{\frac 14}}{4096 \pi ^2}-\frac{35}{768\cdot 3^{\frac 34}}=C,\\
\frac{A(3)}{8\cdot 3^{\frac 34}\sqrt{\pi}}&=\frac{7\cdot \pi ^3}{23328\cdot 3^{\frac 14}}+\frac{105\cdot 3^{\frac 34}}{8192\cdot \pi}+\frac{315\cdot 3^{\frac 34}}{65536\cdot \pi^3}-\frac{91\cdot \pi}{12288\cdot 3^{\frac 14}}=D,\\
\frac{A(4)}{8\cdot 3^{\frac 34}\sqrt{\pi}}&=\frac{7441\cdot \pi ^4}{35831808\cdot 3^{\frac 34}}+\frac{5005}{131072\cdot 3^{\frac 34}}-\frac{1155\cdot 3^{\frac 14}}{131072\cdot \pi^2}+\frac{31185\cdot 3^{\frac 14}}{4194304 \pi^4}-\frac{77\cdot \pi^2}{13824\cdot 3^{\frac 34}}=E.
\end{align*}
where the coefficients\footnote{The term $\frac{7441\pi^4}{35831803\cdot 3^{\frac 34}}$ in $E$  (see \cite[equation (4.16)]{BB}) is $\frac{7441\pi^4}{35831808\cdot 3^{\frac 34}}$.} $A, B, C, D, E$ as given in \cite[equation (4.16)]{BB}.

The coefficients $A_1(m)$ ($0\le m\le 9$) are given by
\begin{align*}
&A_1(0)=\sqrt{\pi},\quad A_1(1)=\frac{7 \pi ^{\frac32}}{6 \sqrt{3}}-\frac{15 \sqrt3}{16\sqrt\pi},\quad A_1(2)=-\frac{245 \sqrt{\pi }}{96}+\frac{205 \pi ^{\frac52}}{864}+\frac{315}{512 \pi ^{\frac32}},\\
&A_1(3)=\frac{945 \sqrt{3}}{8192 \pi ^{\frac52}}+\frac{765 \sqrt3}{1024\sqrt\pi}-\frac{1307 \pi ^{\frac32}}{512 \sqrt{3}}+\frac{821 \pi ^{\frac72}}{7776 \sqrt{3}},\\
&A_1(4)=\frac{131815 \sqrt{\pi }}{49152}+\frac{63985 \pi ^{\frac92}}{4478976}-\frac{7463 \pi ^{\frac52}}{13824}-\frac{4095}{16384 \pi ^{\frac32}}+\frac{93555}{524288 \pi ^{\frac72}},\\
&A_1(5)=-\frac{12915 \sqrt{3}}{1048576 \pi ^{\frac52}}+\frac{1216215 \sqrt{3}}{8388608 \pi ^{\frac92}}-\frac{409705 \sqrt3}{262144\sqrt\pi}+\frac{514381 \pi ^{\frac32}}{147456 \sqrt{3}}-\frac{7005599 \pi ^{\frac72}}{23887872 \sqrt{3}}+\frac{925777 \pi ^{\frac{11}2}}{134369280 \sqrt{3}},\\
&A_1(6)\!=\!-\!\frac{20545525 \sqrt{\pi }}{2359296}\!+\!\frac{45749275 \pi ^{\frac52}}{28311552}\!+\!\frac{18943193 \pi ^{\frac{13}2}}{11609505792}\!-\!\frac{12035101 \pi ^{\frac92}}{143327232}\!+\!\frac{46171125}{16777216 \pi ^{\frac32}}\!-\!\frac{8513505}{16777216 \pi ^{\frac72}}\!+\!\frac{127702575}{268435456 \pi ^{\frac{11}2}},\\
&A_1(7)=\frac{156981825 \sqrt{3}}{268435456 \pi ^{\frac52}}-\frac{241215975 \sqrt{3}}{536870912 \pi ^{\frac92}}+\frac{2791213425 \sqrt{3}}{4294967296 \pi ^{\frac{13}2}}+\frac{349273925 \sqrt3}{50331648\sqrt\pi}-\frac{5444163725 \pi ^{\frac32}}{452984832 \sqrt{3}}+\frac{2250563887 \pi ^{\frac72}}{1528823808 \sqrt{3}}\\
&\hspace{10 cm}-\frac{1610171405 \pi ^{\frac{11}2}}{20639121408 \sqrt{3}}+\frac{145963429 \pi ^{\frac{15}2}}{87071293440 \sqrt{3}},\\
&A_1(8)=\frac{8947964025}{8589934592 \pi ^{\frac72}}-\frac{13749310575}{8589934592 \pi ^{\frac{11}2}}+\frac{1750090817475}{549755813888 \pi ^{\frac{15}2}}-\frac{1327240915}{268435456 \pi ^{\frac32}}+\frac{517195553875 \sqrt{\pi }}{28991029248}\\
&\hspace{4 cm}+\frac{397712337035 \pi ^{\frac92}}{660451885056}-\frac{42760713853 \pi ^{\frac52}}{8153726976}+\frac{206426334029 \pi ^{\frac{17}2}}{280857164120064}-\frac{47146187567 \pi ^{\frac{13}2}}{1393140695040},\\\nonumber
&A_1(9)=\frac{134219460375 \sqrt{3}}{137438953472 \pi ^{\frac92}}-\frac{2598619698675 \sqrt{3}}{1099511627776 \pi ^{\frac{13}2}}+\frac{53086088130075 \sqrt{3}}{8796093022208 \pi ^{\frac{17}2}}-\frac{724073775425 \sqrt3}{51539607552\sqrt\pi}-\frac{9290686405 \sqrt{3}}{8589934592 \pi ^{\frac52}}\\
&\!+\!\frac{2095274978797 \pi ^{\frac32}}{57982058496 \sqrt{3}}\!+\!\frac{330023312969 \pi ^{\frac{11}2}}{330225942528 \sqrt{3}}\!-\!\frac{30623849951695 \pi ^{\frac72}}{3522410053632 \sqrt{3}}\!+\!\frac{87061967090347 \pi ^{\frac{19}2}}{75831434312417280 \sqrt{3}}
\!-\!\frac{3922100346551 \pi ^{\frac{15}2}}{71328803586048 \sqrt{3}}.
\end{align*}

The coefficients $A_2(m)$ ($0\le m\le 9$) are given by
\begin{align*}
&A_2(0)=\sqrt{\pi},\quad A_2(1)=\frac{13 \pi ^{\frac32}}{6 \sqrt{3}}-\frac{15 \sqrt3}{16\sqrt\pi},\quad A_2(2)=-\frac{455 \sqrt{\pi }}{96}+\frac{685 \pi ^{\frac52}}{864}+\frac{315}{512 \pi ^{\frac32}},\\
&A_2(3)=\frac{945 \sqrt{3}}{8192 \pi ^{\frac52}}+\frac{4095 \sqrt3}{1024\sqrt\pi}-\frac{4795 \pi ^{\frac32}}{512 \sqrt{3}}+\frac{2305 \pi ^{\frac72}}{3888 \sqrt{3}},\\
&A_2(4)=\frac{263725 \sqrt{\pi }}{16384}+\frac{516241 \pi ^{\frac92}}{4478976}-\frac{25355 \pi ^{\frac52}}{6912}-\frac{45045}{16384 \pi ^{\frac32}}+\frac{93555}{524288 \pi ^{\frac72}},\\
&A_2(5)=-\frac{585585 \sqrt{3}}{1048576 \pi ^{\frac52}}+\frac{1216215 \sqrt{3}}{8388608 \pi ^{\frac92}}-\frac{3428425 \sqrt3}{262144\sqrt\pi}+\frac{2307305 \pi ^{\frac32}}{73728 \sqrt{3}}-\frac{73822463 \pi ^{\frac72}}{23887872 \sqrt{3}} +\frac{7831183 \pi ^{\frac{11}2}}{134369280 \sqrt{3}},\\
&A_2(6)=-\frac{57682625 \sqrt{\pi }}{1179648}+\frac{369112315 \pi ^{\frac52}}{28311552}+\frac{154279125}{16777216 \pi ^{\frac32}}-\frac{15810795}{16777216 \pi ^{\frac72}}+\frac{127702575}{268435456 \pi ^{\frac{11}2}} +\frac{113476313 \pi ^{\frac{13}2}}{11609505792}\\
&\hspace{12.6cm}-\frac{101805379 \pi ^{\frac92}}{143327232},\\
&A_2(7)=\frac{524549025 \sqrt{3}}{268435456 \pi ^{\frac52}}-\frac{447972525 \sqrt{3}}{536870912 \pi ^{\frac92}}+\frac{2791213425 \sqrt{3}}{4294967296 \pi ^{\frac{13}2}}+\frac{980604625 \sqrt3}{25165824\sqrt\pi}-\frac{43924365485 \pi ^{\frac32}}{452984832 \sqrt{3}}\\
&\hspace{6.5 cm}+\frac{19037605873 \pi ^{\frac72}}{1528823808 \sqrt{3}}-\frac{9645486605 \pi ^{\frac{11}2}}{20639121408 \sqrt{3}}+\frac{1469104051 \pi ^{\frac{15}2}}{243799621632 \sqrt{3}},\\
&A_2(8)=\frac{29899294425}{8589934592 \pi ^{\frac72}}-\frac{25534433925}{8589934592 \pi ^{\frac{11}2}}+\frac{1750090817475}{549755813888 \pi ^{\frac{15}2}}-\frac{3726297575}{134217728 \pi ^{\frac32}}+\frac{4172814721075 \sqrt{\pi }}{28991029248}\\
&\hspace{3 cm}+\frac{2382435191435 \pi ^{\frac92}}{660451885056}-\frac{361714511587 \pi ^{\frac52}}{8153726976}+\frac{2538116412481 \pi ^{\frac{17}2}}{1404285820600320}-\frac{474520608473 \pi ^{\frac{13}2}}{3900793946112},\\
&A_2(9)=\frac{448489416375 \sqrt{3}}{137438953472 \pi ^{\frac92}}-\frac{4826008011825 \sqrt{3}}{1099511627776 \pi ^{\frac{13}2}}+\frac{53086088130075 \sqrt{3}}{8796093022208 \pi ^{\frac{17}2}}-\frac{584194060950 \sqrt3}{51539607552\sqrt\pi}-\frac{26084083025 \sqrt{3}}{4294967296 \pi ^{\frac52}}\\
&\!+\!\frac{17724011067763 \pi ^{\frac32}}{57982058496 \sqrt{3}}\!+\!\frac{2372603042365 \pi ^{\frac{11}2}}{660451885056 \sqrt{3}}\!-\!\frac{183447509740495 \pi ^{\frac72}}{3522410053632 \sqrt{3}}\!+\!\frac{174853672083553 \pi ^{\frac{19}2}}{75831434312417280 \sqrt{3}}
\!-\!\frac{48224211837139 \pi ^{\frac{15}2}}{356644017930240 \sqrt{3}}.
\end{align*}

The coefficients $A_3(m)$ ($0\le m\le 9$) are given by
\begin{align}\nonumber
&A_3(0)=\sqrt{\pi},\quad A(1,3)=\frac{19 \pi ^{\frac32}}{6 \sqrt{3}}-\frac{15 \sqrt3}{16\sqrt\pi},\quad A(2,3)=-\frac{665 \sqrt{\pi }}{96}+\frac{1453 \pi ^{\frac52}}{864}+\frac{315}{512 \pi ^{\frac32}},\\\nonumber
&A_3(3)=\frac{945 \sqrt{3}}{8192 \pi ^{\frac52}}+\frac{5985 \sqrt3}{1024\sqrt\pi}-\frac{10171 \pi ^{\frac32}}{512 \sqrt{3}}+\frac{14015 \pi ^{\frac72}}{7776 \sqrt{3}},\\\nonumber
&A_3(4)=\frac{559405 \sqrt{\pi }}{16384}+\frac{2193649 \pi ^{\frac92}}{4478976}-\frac{154165 \pi ^{\frac52}}{13824}-\frac{65835}{16384 \pi ^{\frac32}}+\frac{93555}{524288 \pi ^{\frac72}},\\\nonumber
&A_3(5)=-\frac{855855 \sqrt{3}}{1048576 \pi ^{\frac52}}+\frac{1216215 \sqrt{3}}{8388608 \pi ^{\frac92}}-\frac{7272265 \sqrt3}{262144\sqrt\pi}+\frac{14029015 \pi ^{\frac32}}{147456 \sqrt{3}}-\frac{313691807 \pi ^{\frac72}}{23887872 \sqrt{3}}+\frac{43985869 \pi ^{\frac{11}2}}{134369280 \sqrt{3}},\\\nonumber
&A_3(6)=-\frac{350725375 \sqrt{\pi }}{2359296}+\frac{1568459035 \pi ^{\frac52}}{28311552}+\frac{327251925}{16777216 \pi ^{\frac32}}-\frac{23108085}{16777216 \pi ^{\frac72}}+\frac{127702575}{268435456 \pi ^{\frac{11}2}}+\frac{3702305053 \pi ^{\frac{13}2}}{58047528960}\\\nonumber
&\hspace{12.4 cm}-\frac{571816297 \pi ^{\frac92}}{143327232},\\\nonumber
&A_3(7)=\frac{1112656545 \sqrt{3}}{268435456 \pi ^{\frac52}}-\frac{654729075 \sqrt{3}}{536870912 \pi ^{\frac92}}+\frac{2791213425 \sqrt{3}}{4294967296 \pi ^{\frac{13}2}}+\frac{5962331375 \sqrt3}{50331648\sqrt\pi}-\frac{186646625165 \pi ^{\frac32}}{452984832 \sqrt{3}}\\\nonumber
&\hspace{5.5cm}+\frac{106929647539 \pi ^{\frac72}}{1528823808 \sqrt{3}}-\frac{62939185901 \pi ^{\frac{11}2}}{20639121408 \sqrt{3}}+\frac{1082307041 \pi ^{\frac{15}2}}{30474952704 \sqrt{3}},\\\nonumber
&A_3(8)=\frac{63421423065}{8589934592 \pi ^{\frac72}}-\frac{37319557275}{8589934592 \pi ^{\frac{11}2}}+\frac{1750090817475}{549755813888 \pi ^{\frac{15}2}}-\frac{22656859225}{268435456 \pi ^{\frac32}}+\frac{17731429390675 \sqrt{\pi }}{28991029248}\\\nonumber
&\hspace{2.3 cm}+\frac{15545978917547 \pi ^{\frac92}}{660451885056}-\frac{2031663303241 \pi ^{\frac52}}{8153726976}+\frac{10233536677249 \pi ^{\frac{17}2}}{1404285820600320}-\frac{349585174243 \pi ^{\frac{13}2}}{487599243264},\\\nonumber
&A_3(9)\!=\!\frac{951321345975 \sqrt{3}}{137438953472 \pi ^{\frac92}}\!-\!\frac{7053396324975 \sqrt{3}}{1099511627776 \pi ^{\frac{13}2}}\!+\!\frac{53086088130075 \sqrt{3}}{8796093022208 \pi ^{\frac{17}2}}\!-\!\frac{24824001146945 \sqrt3}{51539607552\sqrt\pi}\!-\!\frac{158598014575 \sqrt{3}}{8589934592 \pi ^{\frac52}}\\\nonumber
&\hspace{2 cm}\!+\!\frac{99551501858809 \pi ^{\frac32}}{57982058496 \sqrt{3}}\!+\!\frac{1747925871215 \pi ^{\frac{11}2}}{82556485632 \sqrt{3}}\!-\!\frac{1197040376651119 \pi ^{\frac72}}{3522410053632 \sqrt{3}}\!+\!\frac{460609498071319 \pi ^{\frac{19}2}}{75831434312417280 \sqrt{3}}\\\label{coeff3seqn}
&\hspace{10 cm}\!-\!\frac{194437196867731 \pi ^{\frac{15}2}}{356644017930240 \sqrt{3}}.
\end{align}

\end{document}